\documentclass{mcom-l}
\usepackage{amsfonts,amsmath, amscd, amssymb, hyperref}
\usepackage[pdftex]{graphicx}
\usepackage{cite, enumerate, multirow, caption, overpic}

\copyrightinfo{2014}{by the authors. This paper may be reproduced, in its entirety, for non-commercial purposes.}

\newtheorem{theorem}{Theorem}[section]
\newtheorem{lemma}[theorem]{Lemma}
\newtheorem{proposition}[theorem]{Proposition}
\newtheorem{corollary}[theorem]{Corollary}

\theoremstyle{definition}
\newtheorem{definition}[theorem]{Definition}

\newtheorem{as}[theorem]{Assumption}

\theoremstyle{remark}
\newtheorem{remark}[theorem]{Remark}

\numberwithin{equation}{section}

\newcommand{\K}{\mathcal{K}}
\newcommand{\Rd}{{\mathord{\mathbb R}^d}}
\newcommand{\R}{{\mathord{\mathbb R}}}
\newcommand{\grad}{\nabla}
\newcommand{\la}{\langle}
\newcommand{\ra}{\rangle}
\newcommand{\supp}{{\mathop{\rm supp\ }}}
\newcommand{\loc}{{\rm loc}}

\begin{document}

\title[A Blob Method for the Aggregation Equation]{A Blob Method for the Aggregation Equation}

\author{Katy Craig}
\address{Department of Mathematics, Rutgers University, 110 Frelinghuysen Road, Piscataway, NJ 08854-8019}
\email{kcraig@math.ucla.edu}
\thanks{This work was supported by
NSF grants DMS-0907931 and EFRI-1024765, as well as NSF grant 0932078 000, which supported Craig's visit and Bertozzi's residence at the Mathematical Sciences Research Institute during Fall 2013. }

\author{Andrea L. Bertozzi}
\address{Department of Mathematics, University of California, Los Angeles, 520 Portola Plaza, Los Angeles, CA 90095-1555}
\email{bertozzi@math.ucla.edu}


\begin{abstract}
Motivated by classical vortex blob methods for the Euler equations, we develop a numerical blob method for the aggregation equation. This provides a counterpoint to existing literature on particle methods. By regularizing the velocity field with a mollifier or ``blob function'', the blob method has a faster rate of convergence and allows a wider range of admissible kernels. In fact, we prove arbitrarily high polynomial rates of convergence to classical solutions, depending on the choice of mollifier. The blob method conserves mass and the corresponding particle system is both energy decreasing for a regularized free energy functional and preserves the Wasserstein gradient flow structure. We consider numerical examples that validate our predicted rate of convergence and illustrate qualitative properties of the method.
\end{abstract}

\maketitle

\section{Introduction}
The aggregation equation describes the evolution of a nonnegative density $\rho$ according a velocity field which is obtained by convolving the density with the gradient of a kernel $K: \Rd \to \R$,
\begin{align} \label{agg eqn}
\begin{cases}
\rho_t + \grad \cdot (\vec{v} \rho) = 0 \ , \quad \quad \quad \quad\vec{v} = -\grad K * \rho  \ , \\
\rho(x,0) = \rho_0(x) \ .
\end{cases}
\end{align}
The dynamics governed by this equation arise in a range of problems, including biological swarming \cite{MogilnerEdelstein,MogilnerEdelsteinBent, TopazBertozzi1 ,TopazBertozzi2}, robotic swarming \cite{PereaGomezElosegui,ChuangHuangDorsognaBertozzi}, molecular self-assembly \cite{DoyeWalesBerry, Wales, RechtsmanStillingerTorquato}, and the evolution of vortex densities in superconductors \cite{AmbrosioSerfaty,Weinan, LinZhang, DuZhang, MasmoudiZhang, Poupaud}. For swarming and molecular self-assembly, common choices of kernel include the repulsive-attractive Morse potential and power law potential,
\begin{align} \label{repulsive attractive potentials}
K(x) = C_r e^{-|x|/l_r} - C_a c^{-|x|/l_a}  \ ,\quad K(x) = |x|^a/a - |x|^b/b \ , \quad a>b\ . 
\end{align}
To model the evolution of vortex densities in superconductors, $K$ is chosen to be the two dimensional Newtonian potential,
\begin{align} \label{Newtonian potential}
K(x) = \frac{1}{2 \pi} \log |x| \ . 
\end{align}

In addition to utility in a range of applications, the aggregation equation possesses several features of current mathematical interest. It is \emph{non-local}---the motion of the density at any point depends on the value of the density at every point---and when $K$ is symmetric it is formally a \emph{gradient flow} in the Wasserstein metric. When $K$ has low regularity at the origin, solutions may blow up in finite time or form rich patterns as they approach a steady state. In recent years, there has been substantial interest in these structures from both analytic and numerical perspectives \cite{BalagueCarrilloLaurentRaoul,BalagueCarrilloLaurentRaoul_Dimensionality,Bertozzietal_RingPatterns, BalagueCarrilloYao, BertozziLaurentLeger, BertozziGarnettLaurent, BertozziBrandman, BertozziCarrilloLaurent, 5person2, CarrilloChipotHuang, Dong, FellnerRaoul, FellnerRaoul2, FetecauHuangKolokolnikov, FetecauHuang, HuangBertozzi, HuangBertozzi2, Kolokolnikovetal_StabilityRingPatterns, Poupaud, Raoul,SunUminskyBertozzi, TopazBertozzi1,HuangWitelskiBertozzi}. For example, Kolokolnikov, Sun, Uminsky, and Bertozzi studied steady states for repulsive-attractive power law kernels using linear stability analysis, complemented with numerical examples computed by a particle method \cite{Kolokolnikovetal_StabilityRingPatterns}. Particle methods have also been used in purely analytic work, due to the close relationship between particle approximations and weak measure solutions from the perspective of Wasserstein gradient flow \cite{5person, 5person2, BonaschiCarrilloDiFrancescoPeletier}.

In spite of the significant activity investigating qualitative properties of solutions, much of it presented alongside numerical examples, analysis of numerical methods has begun only recently. Carrillo, Choi, and Hauray proved that a particle method converges to a weak measure solution when the kernel has power law growth $|x|^a$, $2-d < a\leq 2$ \cite{CarrilloChoiHauray}. Carrillo, Chertock, and Huang developed a finite volume method for a wide class of nonlinear, nonlocal equations, including the aggregation equation, and proved the existence of a related discrete free energy which is dissipated along the scheme \cite{CarrilloChertockHuang}. Most recently, James and Vauchelet developed a finite difference method for a generalization of the one dimensional aggregation equation and proved its convergence to weak measure solutions \cite{JamesVauchelet}.

In this paper, we develop a new numerical method for the multidimensional aggregation equation for a wide range of kernels, including the Newtonian potential, repulsive-attractive Morse potentials, and repulsive-attractive power law potentials. In Section \ref{blob method def section}, we define the numerical method, which is a particle method in which the kernel is regularized by convolution with a smooth, rapidly decreasing mollifier or \emph{blob function}. In Section \ref{conserved quantities section}, we show that the numerical solutions conserve mass and the corresponding particle system is energy decreasing for a regularized free energy functional and preserves the formal Wasserstein gradient flow structure.
 In Section \ref{ConvergenceSection}, we prove that our numerical solutions converge to classical solutions of the aggregation equation. In Section \ref{NumericsSection}, we provide numerical examples which validate our theoretically predicted rate of convergence and illustrate qualitative properties of the method. In Section \ref{ConclusionSection}, we describe directions for future work.

Our numerical method is inspired by classical \emph{vortex blob methods} for the vorticity formulation of the Euler equations, which is structurally similar to the aggregation equation, particularly when the kernel is the Newtonian potential, $K = (-\Delta)^{-1}$ \cite{Chorin,HaldDelPrete,Hald,BealeMajda1,BealeMajda2,CottetRaviart,AndersonGreengard}. This equation describes the evolution of the vorticity $\omega$ according to a velocity field obtained by convolving the vorticity with the two or three dimensional Biot-Savart kernel $K_d$,
\begin{align} \label{voreqn2}
\begin{cases}
\omega_t + \vec{v} \cdot \grad \omega = (\grad \vec{v}) \omega \ ,  \quad \quad \quad \vec{v} = \vec{K}_d * \omega  \ ,\\
\omega(x,0) = \omega_0(x)  \ .
\end{cases}
\end{align} 
The velocity field for (\ref{voreqn2}) may be rewritten as $v = \grad^\perp \Delta^{-1} \omega$, and when $K = (-\Delta)^{-1}$, the velocity field for the aggregation equation is $v = \grad \Delta^{-1} \rho$. Due to these similarities, there has also been interest in behavior of equations for which the velocity field is a combination of $\grad^\perp \Delta^{-1}  \rho$ and $\grad \Delta^{-1} \rho$ \cite{DuZhang, SunUminskyBertozzi_Birkhoff}.

While the main features of our method are analogous to vortex blob methods for the Euler equations, there are a few key differences. First of all, we consider the equation in dimensions $d \geq 1$, and we allow both singular and non-singular kernels. Also, in spite of the structural similarity between the aggregation equation and the Euler equations, an important difference from the perspective of particle methods is that the velocity field in the aggregation equation is not divergence free, but is instead a gradient flow. Adapting blob methods to the compressible case is relatively new, building on Eldredge's results for compressible fluids in the engineering literature \cite{Eldredge} and Duan and Liu's results for the $b$-equation \cite{DuanLiu}. For our purposes, conservation of mass proves to be a sufficient substitute for incompressibility.

The Lagrangian nature of our method offers three main benefits over a finite difference or finite element method. First, it allows us to avoid the main form of numerical diffusion. Second, it only requires computational elements in regions where the density is nonzero. Third, it ensures that the method is inherently adaptive, concentrating computational elements in areas where particles accumulate and thereby increasing resolution near singularities. Our method has further benefits over a particle method since, instead of removing a singularity of the kernel at zero by redefining $\grad K(0) = 0$, we regularize the kernel by convolution with a mollifier. Because of this, we are able to obtain arbitrarily high orders of convergence $\mathcal{O}(h^{mq})$, depending on $m \geq 4$, $\frac{1}{2}< q<1$, which describe the structure of the mollifier. On the other hand, particle methods for the Euler equations only attain $\mathcal{O}(h^{2-\epsilon})$ rate of convergence \cite{GoodmanHouLowengrub, HouLowengrub}, and our numerical simulations  indicate that the same rate of convergence holds for the aggregation equation as well.

Though we are the first to prove quantitative rates of convergence a numerical method for the aggregation equation and we are the first to implement a blob method numerically, we are not the first to consider regularized methods for aggregation and aggregation-like equations. Lin and Zhang \cite{LinZhang} used blob methods to prove the existence of weak solutions to the two dimensional aggregation equation when the kernel is the Newtonian potential. Bhat and Fetecau \cite{BhatFetecau1, BhatFetecau15, BhatFetecau2} and Norgard and Mohseni \cite{NorgardMohseni1,NorgardMohseni2} studied a similar regularization for Burger's equation, which is related to the aggregation equation when $K$ is the Newtonian potential \cite[Section 4]{BertozziGarnettLaurent}.
Bhat and Fetecau compute several exact solutions to the regularized problem and observe similar phenomena near blowup time to our simulations \cite{BhatFetecau2}.

\section{Blob Method} \label{blobmethoddef}
We begin by recalling some basic properties of the aggregation equation. It is a \emph{continuity equation}, describing the evolution of a density $\rho$ according to a velocity field $\vec{v}$ so that the mass of $\rho$ is conserved. Conservation of mass is the key property which allows us to adapt vortex blob methods from the classical fluids case, playing the same role that incompressibility plays for the Euler equations.

Let $X^t(\alpha)$ be the particle trajectory map induced by the velocity $\vec{v} = -\grad K*\rho$,
\begin{align*}
\frac{d}{dt} X^t(\alpha) = - \grad K * \rho(X^t(\alpha),t) \ , \quad \quad
X^0(\alpha) = \alpha \ .
\end{align*}
Rewriting the aggregation equation in terms of the material derivative gives
\[ \frac{D\rho}{Dt}  = - (\grad \cdot \vec{v}) \rho  \ . \] 
Thus, $\rho$ evolves along particle trajectories according to
\begin{align*}
\begin{cases}
\frac{d}{dt} \rho(X^t(\alpha),t) &=  \left(\Delta K * \rho(X^t(\alpha),t) \right) \ \rho(X^t(\alpha),t) \ , \\
\rho(X^0(\alpha),0) &= \rho_0(\alpha) \ .
\end{cases}
\end{align*}
If $J(\alpha, t) = \det( \grad_\alpha X^t(\alpha))$ is the Jacobian determinant of the particle trajectories, conservation of mass implies
\begin{align*}
\rho(X^t(\alpha),t) |J^t(\alpha)| = \rho(\alpha,0) \ .
\end{align*}
This allows us to formally rewrite the velocity field and the divergence of the velocity field in terms of integration in the Lagrangian coordinates,
\begin{align} v(x,t) = -\grad K * \rho(x,t) = -\int_\Rd \grad K (x-X^t(\alpha)) \rho_0(\alpha) d \alpha \label{velocityform} \ , \\
\grad \cdot v(x,t) = - \Delta K * \rho(x,t) = - \int_\Rd \Delta K (x-X^t(\alpha)) \rho_0(\alpha) d \alpha \ . \nonumber
\end{align}

\subsection{Definition of blob method} \label{blob method def section}
Let $h \mathbb{Z}^d$ be a d-dimensional integer grid with spacing $h$. Suppose that for $t \in [0, T]$, the density $\rho(x,t)$ is compactly supported.
Our blob method provides a way to compute approximate particle trajectories starting at the grid points $ih \in h\mathbb{Z}^d$ and then compute the approximate density and velocity along these particle trajectories. We write $X_i(t)$ for $X(ih,t)$, and in general we use a subscript $i$ to denote a quantity transported along the particle trajectory beginning at $ih$, e.g. $v_i(t)$ for $v(X_i(t),t)$ and $\grad \cdot v_i(t)$ for $\grad \cdot v(X_i(t),t)$.

The approximate particle trajectories, densities, and velocities are prescribed by a finite system of ordinary differential equations. Solutions of this system may be computed numerically by a variety of methods to a high degree of accuracy, so that the dominant error of the blob method comes from the reduction of the original aggregation equation to the system of ODEs. 
This reduction has two steps. First, to avoid a possible singularity of $\grad K$, we regularize $\grad K$ by convolution with a smooth, radial, rapidly decreasing mollifier or ``blob function'' $\psi$. For $\delta >0$, we write $\psi_\delta(x) = \delta^{-d} \psi(x/\delta)$ and $K_\delta = K * \psi_\delta$. Second, we use a particle approximation for the initial density. Specifically, we place a Dirac mass of weight ${\rho_0}_j h^d$ at each point of the grid $h \mathbb{Z}^d$,
\begin{align*}
\rho_0^{\text{particle}}(\alpha) = \sum_{j} \delta(\alpha - jh) {\rho_0}_j h^d  \ .
\end{align*}
Combining this regularization and discretization with the integral form for the velocity (\ref{velocityform}) leads to the following approximate velocity.
\begin{definition}[approximate velocity along exact particle trajectories]  \label{approx vel exact part}
\begin{align*}
v^h(x,t) &= -\int_\Rd \grad K_\delta (x-X^t(\alpha)) \rho_0^{\text{particle}}(\alpha) d \alpha = -  \sum_{j} \grad K_\delta(x-X_j(t)) {\rho_0}_j h^d \ .
\end{align*}
\end{definition}

We now turn to the definition of the blob method. Following the fluids literature, we use tildes to distinguish approximate quantities from their exact counterparts.\begin{definition}[blob method] \label{blob method def}
\ \\ Approximate particle trajectories:
$\begin{cases}
\frac{d}{dt} \tilde{X}_i(t) &= -\sum_{j} \grad K_\delta (\tilde{X}_i(t) - \tilde{X}_j(t)) {\rho_0}_j h^d \\
\tilde{X}_i(0) &= ih
\end{cases}$
Approximate velocity field: $ \tilde{v}_i(t) = - \sum_{j} \grad K_\delta(\tilde{X}_i(t)-\tilde{X}_j(t)) {\rho_0}_j h^d$ \\
Approximate divergence of velocity field:
$\grad \cdot \tilde{v}_i(t) = - \sum_{j} \Delta K_\delta(\tilde{X}_i(t)-\tilde{X}_j(t)) {\rho_0}_j h^d$
Approximate density:$\begin{cases}
\frac{d}{dt} \tilde{\rho}_i(t) &= -\grad \cdot \tilde{v}_i(t)  \tilde{\rho}_i(t) \\
\tilde{\rho}_i(0) &= {\rho_0}_i 
\end{cases}$
\end{definition}

Due to the regularization of the kernel, $\grad K_\delta$ and $\Delta K_\delta$ are locally Lipschitz. Thus, for any $\delta >0$, there exists a unique solution to this system of ODEs on some time interval $[0,T_0]$. It is part of our result that this time interval $[0,T_0]$ must be at least as large as the interval of existence for the corresponding classical solution to the aggregation equation.

When $K$ is the Newtonian potential $(\Delta)^{-1}$, there is a simple heuristic interpretation of the blob method: it approximates the density by a sum of blobs that follow particle trajectories. This follows by taking the divergence of the velocity, so $\rho = \Delta K*\rho = -\grad \cdot \vec{v}$. Using this relationship we define an alternative approximate density
 \begin{align*}
\tilde{\rho}(x,t) =  \sum_{j} \Delta K_\delta (x- \tilde{X}_j(t)) {\rho_0}_j h^d = \sum_{j} \psi_\delta (x- \tilde{X}_j(t)) {\rho_0}_j h^d \ .
\end{align*}
This is the analogue of two dimensional vortex blob methods for the Euler equations. However, as the purpose of this paper is to devise a numerical method for a variety of kernels, we will focus on the more general method of computing the approximate density from Definition \ref{blob method def}. 

\subsection{Conserved Quantities} \label{conserved quantities section}
For computational purposes, we only calculate the blob method along particle trajectories originating at grid points $h \mathbb{Z}^d$. Still, a simple extension of the method allows one to compute approximate particle trajectories, velocity, and density starting from anywhere in Euclidean space.
\begin{definition}[blob method: off the grid] \label{off grid def}
\ \\ Approximate particle trajectories: $
\begin{cases}
\frac{d}{dt} \tilde{X}(\alpha,t) = -\sum_{j} \grad K_\delta (\tilde{X}(\alpha,t) - \tilde{X}_j(t)) {\rho_0}_j h^d \\
\tilde{X}(\alpha,0) \hspace{3mm} = \alpha
\end{cases}
$ \\
Approximate velocity field: $
\tilde{v}(x,t) \hspace{.5mm} = - \sum_{j} \grad K_\delta(x-\tilde{X}_j(t)) {\rho_0}_j h^d$ \\
Approximate density: $ \begin{cases}
\frac{d}{dt} \tilde{\rho}(\tilde{X}(\alpha,t),t) \hspace{-3mm} &= - \grad \cdot \tilde{v}(\tilde{X}(\alpha,t),t) \tilde{\rho}(\tilde{X}(\alpha,t),t) \\
\tilde{\rho}(\tilde{X}(\alpha,0), 0) &=  \rho_0(\alpha)
\end{cases}$
\end{definition}

From this perspective, the blob method preserves the continuity equation structure of the aggregation equation and consequently conservation of mass,
\begin{align}
\frac{D}{Dt} \tilde{\rho} = -(\grad \cdot \tilde{v}) \tilde{\rho} \implies \frac{d}{dt} \int_{\tilde{X}(\Omega,t)} \rho dx = 0 \ . \label{blob continuity equation}
 \end{align}

For the remainder of the section, we suppose $K$ and $K_\delta$ are even. In this context, the particle system corresponding to the blob method 
\begin{align} \label{particle system} \hat{\rho}(x,t) = \sum_{j } \delta(x - \tilde{X}_j(t)) {\rho_0}_j h^d
\end{align}
preserves the aggregation equation's Wasserstein gradient flow structure and is energy decreasing for a regularized free energy functional. Recall that the aggregation equation is formally the gradient flow of the interaction energy,
 \[ E(\rho) = \frac{1}{2} \int_{\Rd \times \Rd}  \rho(x) K(x-y) \rho(y) dx  dy \ . \]
To see this, recall that the Wasserstein gradient is defined by ${\grad_W E(\rho) = - \grad \cdot (\rho \grad \frac{\partial E}{\partial \rho})}$, where $\frac{\partial E}{\partial \rho}$ is the functional derivative of $E$ at $\rho$. Applying this to $E$, we recover the aggregation equation as the gradient flow,
\[ \rho_t = - \grad_W E(\rho) = \grad \cdot (\rho (\grad K * \rho)) \ . \]
This gradient flow structure may be made rigorous given sufficient convexity, regularity, and decay of the kernel \cite{AGS, 5person, 5person2}.

In analogy with the aggregation equation, the particle system corresponding to the blob method is formally the Wasserstein gradient flow of the regularized energy 
 \begin{align} \label{regularized energy} E_\delta(\rho) = \frac{1}{2} \int_{\Rd \times \Rd}\rho(x) K_\delta(x-y) \rho(y) dx dy \ .
 \end{align}
In particular, the particle system (\ref{particle system}) is a weak measure solution of ${\frac{d}{dt} \hat{\rho} + \grad \cdot (\tilde{v} \hat{\rho}) = 0}$, where the velocity may be rewritten as
 \begin{align*} \tilde{v}(x,t) &= - \sum_{j} \grad K_\delta(x-\tilde{X}_j(t)) {\rho_0}_j h^d = - \grad K_\delta * \hat{\rho}(x,t) = \grad \frac{ \partial E_\delta}{\partial \hat{\rho}} \ .
 \end{align*}

Though the gradient flow structure may be purely formal, the regularized energy (\ref{regularized energy}) always decreases along particle solutions corresponding to the blob method.  Rewriting the regularized energy in Lagrangian coordinates,
\begin{align*} E_\delta(\hat{\rho}(t)) &= \frac{1}{2}\int_{\Rd \times \Rd}\hat{\rho}(\alpha,0) K_\delta(\tilde{X}(\alpha, t) - \tilde{X}(\beta,t)  \hat{\rho}(\beta, 0) d \alpha  d \beta \ , \\
&=  \frac{1}{2} \sum_{i,j} K_\delta(\tilde{X}_i(t) - \tilde{X}_j(t)) {\rho_0}_i {\rho_0}_j h^d h^d \ .
\end{align*}
Differentiating with respect to time,
\begin{align*} \frac{d}{dt} E_\delta(\hat{\rho}(t))
&= \frac{1}{2} \sum_{i,j} \grad K_\delta(\tilde{X}_i(t) - \tilde{X}_j(t)) \cdot \left(\frac{d}{dt} \tilde{X}_i(t)  - \frac{d}{dt} \tilde{X}_j(t) \right) {\rho_0}_i {\rho_0}_j h^d h^d  \ .
\end{align*}
The terms of the sum are the dot product between two $d$ dimensional vectors. For $l = 1, \dots, d$, define 
\[ M^l = \{M_{ij}^l \} = \left\{ \left[\grad K_\delta(\tilde{X}_i(t) - \tilde{X}_j(t))\right]_l \right \} \ , \quad\quad v = \{ v_i \} = \{{\rho_0}_i  h^d\}  \ , \]
where $[ \cdot ]_l$ denotes the $l$th component of the vector, $l = 1, \dots, d$. By definition of the particle trajectories, $[\frac{d}{dt} \tilde{X}_i(t)]_l = [-M^lv]_i$, and by the symmetry of $K$, $M_{ij}^l = -M_{ji}^l$. Letting $\diamond$ denote the element-wise product of vectors, we have
\begin{align*} \frac{d}{dt} E_\delta(\hat{\rho}(t))
&= \sum_{l = 1}^d v^t M^l (-M^l v \diamond v)  = - \sum_{l = 1}^d  (M^l v)\cdot (M^l v \diamond v)  = - \sum_{l=1}^d \sum_i [M^l v]_i^2 {\rho_0}_ih^d \ .
\end{align*}
Since ${\rho_0}_i \geq 0$, the regularized energy $E_\delta$ decreases along the particle system associated to the blob method.

\begin{remark}[$\hat{\rho}$ versus $\tilde{\rho}$]
Blob methods for the Euler equations often only consider approximate particle trajectories, setting aside the approximate density \cite{BealeMajda1,BealeMajda2, AndersonGreengard}. For our purpose of numerically approximating classical solutions, we follow Beale \cite{Beale} and define the approximate density along particle trajectories as a function $\tilde{\rho}:h \mathbb{Z}^d \to \R$ (Definition \ref{blob method def}). However, from the perspective of Wasserstein gradient flow, given the approximate particle trajectories $\tilde{X}$, the natural choice of approximate density  would be the particle system $\hat{\rho}$ (\ref{particle system}). This is also likely the best choice for approximating weak measure solutions to the aggregation equation, though we leave the topic to future work. (See Lin and Zhang \cite{LinZhang}.)
\end{remark}

\section{$L^p$ Convergence of blob method to smooth solutions}\label{ConvergenceSection}
 We now prove the convergence of the blob method to classical solutions of the aggregation equation. Our approach is strongly influenced by results on the convergence of vortex blob methods for the Euler equations \cite{BealeMajda1, BealeMajda2, Beale, AndersonGreengard}, though our proof has different features due to the gradient flow structure of our problem and the fact that we allow a wider range of kernels.

Let $D^+_j$ denote the
forward difference operator in the $j^{th}$ coordinate direction. For $1 \leq p < +\infty$, we consider the following discrete $L^p$ and Sobolev norms of grid functions $u:h \mathbb{Z}^d\to \R$:
\begin{align*}
\|u\|_{L^p_h} = \left( \sum_{i \in \mathbb{Z}^d} |u_i|^p h^d \right)^{1/p} , \quad 
\|u\|_{W^{1,p}_h} = \left(\|u\|_{L^p_h}^p + \sum_{j=1}^d \|D_j^+ u\|_{L^p_h}^p \right)^{1/p}.
\end{align*}
Likewise, we define an $L^\infty_h$ norm and inner product by
\[\|u\|_{L^\infty_h} = \sup_{i \in \mathbb{Z}^d} |u_i| \ , \quad (u,g)_h = \sum_{i \in \mathbb{Z}^d} u_i g_i h^d \ . \]

Given any function $u(x)$ defined on all of $\Rd$, we may consider it as a function on $h \mathbb{Z}^d$ by defining $u_i = u(ih)$. Thus, we may also consider the size of any function $u(x)$ in the above discrete norms.
We may also define the discrete $L^p$ norm on a subset $\Omega \subseteq \Rd$ by $\|u\|_{L^p_h(\Omega)} = \| 1_{\Omega} u\|_{L^p_h}$, where $(1_{\Omega})_i = 1$ if $ih \in \Omega$ and $0$ otherwise. We say that $u \in L^p_h$ is supported in $\Omega$ if $u = 1_\Omega u$.

We define the dual Sobolev norm with duality pairing $\la \cdot, \cdot \ra$ by \[ \|u\|_{W^{-1,p}_h} = \sup_{g \in W^{1,p'}_h} \frac{|\la u_i,g_i \ra|}{\|g_i\|_{W^{1,p'}_h}} \  ,\]
where $1/p + 1/p' = 1$.
We may consider any  $u \in L^p_h$ as a linear functional on $W^{1,p'}_h$ by the duality pairing $\la u, g \ra = (u,g)_h$.

The discrete $L^p$, Sobolev, and dual Sobolev norms are related by the following inequalities. See appendix Section \ref{discrete norms section} for proofs and references.
\begin{proposition} \label{discrete norm prop}
Suppose $1 \leq q \leq p \leq +\infty$, $0< h\leq 1$, and $u \in L^p_h$. Define $B_R = \{x \in \Rd : |x| < R\}$.
 \begin{enumerate}[(a)]
\item $\|u\|_{W^{-1,p}_h} \leq \|u\|_{L^p_h}$.
\item $\|u\|_{L^p_h} \leq (1+2d/h)  \|u\|_{W^{-1,p}_h}$.
\item For any $\Omega \subseteq B_R$, $ \|u\|_{L^{q}_h(\Omega)} \leq C_{p,q,R} \|u\|_{L^p_h(\Omega)} $.
\item If $u$ is supported in $B_R$, $\|u\|_{W^{-1,q}_h} \leq C_{p,q,R} \|u\|_{W^{-1,p}_h} $.
\item Given $l \in \mathbb{Z}^d$, let $T^l$ denote translation on the grid $h \mathbb{Z}^d$ in the direction $l$. Any finite difference operator of the form 
\[ \grad_i^h = \frac{1}{h} \sum_{|l| \leq l_0} \vec{a_l}(h)T^l \]
with $|\vec{a}_l (h)| \leq C_0$ satisfies $\|\grad_i^h u\|_{W^{-1,p}_h} \leq C \|u\|_{L^p_h}$.
\end{enumerate}
\end{proposition}

Measuring the convergence of the particle trajectories and velocity in discrete $L^p$ norms allows us to apply the classical theory of integral operators, both singular and otherwise. We measure the convergence of the density in the discrete $W^{-1,p}$ norm to reflect the fact that, in the most singular case, when $K = (\Delta)^{-1}$, the velocity $v = \grad K * \rho$ has one more derivative than the density: $- \grad \cdot v = \rho$. In general, the velocity may have more regularity with respect to the density, but we prefer to use the same norms for all kernels and reflect the improved regularity in better convergence estimates.

We now turn to the assumptions we place on the kernel, mollifier, and exact solution. These depend on a regularity parameter $L \geq \max \{ d+2, 4\}$ and an accuracy parameter $m\geq 4$.

 \begin{as}[kernel] \label{kernel assumption}
 Suppose that $K = \sum_{n =1}^N K_n$, where for each $K_n$, there exists $S_n \geq 1-d$ so that
 \[ |\partial^\beta \grad K_n(x)| \leq C|x|^{S_n-|\beta|} , \quad \forall x \in \Rd \setminus \{0\}  , \quad 0 \leq |\beta|\leq \max \{L, s + d-1\} \ , \] 
 for $s = \min_n S_n$ and $S = \max_n S_n$.
 If $S_n = 1-d$, we require $K_n(x)$ to be a constant multiple of the Newtonian potential.
\end{as}
The Newtonian potential, repulsive-attractive Morse potential, and repulsive-attractive power law potential (\ref{repulsive attractive potentials},\ref{Newtonian potential}) all satisfy this assumption.

\begin{as}[mollifier] \label{mollifier assumption}
$\psi$ is radial, $\int \psi = 1$, and the following hold:
\begin{enumerate}
\item Accuracy:  $\int x^\gamma \psi(x) dx = 0$ for $1 \leq |\gamma| \leq m-1$ and $\int |x|^m |\psi(x)| dx < +\infty$.
\item Decay: $\exists \  \epsilon >0$ such that $ |x|^{d+\epsilon} | \psi(x)| \leq C$.
\item Regularity: $\psi \in C^L$ and $|x|^{d+|\beta|} |\partial^\beta \psi(x)| \leq C$ for all $|\beta| \leq L$.\\ 
If $S >0$, $\exists \  \epsilon >0$ such that $ |x|^{d+S+\epsilon} | \partial^\beta \psi(x)| \leq C$ for all $|\beta| \leq L$.
\end{enumerate}
If the above holds for $L$ arbitrarily large, we say it holds for $L = +\infty$. If $s = 1-d$, we require $L= +\infty$.
\end{as}
\begin{remark}[accuracy of mollifier] \label{mollifierAccuracy}
The accuracy assumption on the mollifier ensures that for any multiindex $\gamma$ with $|\gamma| < m$,
\begin{align*}
 \int(x-y)^{\gamma} \psi(y) dy = \sum_{\nu \leq \gamma} \binom{\gamma}{\nu} x^{\gamma - \nu} \int y^\nu \psi(y)dy = x^{\gamma} \int \psi(y) dy = x^{\gamma} \ .
 \end{align*}
Thus,  convolution with the mollifier preserves polynomials of order less than $m$. In particular, if $K(x)$ is a polynomial of order at most $m$, $\grad K_\delta = \grad K$ and $\Delta K_\delta = \Delta K$. 
\end{remark}
The following mollifier satisfies Assumption \ref{mollifier assumption} with $d=1$, $m=4, L = +\infty$:
\begin{align} \label{m4d1mollifier}
\psi(x) = \frac{4}{3\sqrt{\pi}}e^{-|x|^2} -\frac{1}{6\sqrt{\pi}} e^{-|x|^2/4} \ .
\end{align}
See Majda and Bertozzi for an algorithm which allows one to construct mollifiers satisfying Assumption \ref{mollifier assumption} for arbitrarily large $m$ \cite[Section 6.5]{MajdaBertozzi}.

  \begin{as}[exact solution] \label{exact sol reg as}
Suppose $\exists \ T>0$, $r \geq \max \{m-(s+d-2),L\}$ so that $\rho \in C^1([0,T],C^{r}(\Rd))$ is a solution to the aggregation equation. Suppose also that $\exists \ R_0 > 1$ so that the support of $\rho(x,t)$ remains bounded in $B_{R_0-1}$
 and, for all $\alpha \in B_{R_0+2}$, $|X^t(\alpha)|$ is bounded for $t \in [0,T]$. If $s=1-d$, for all $\alpha \in \Rd$, $|X^t(\alpha)|$ is bounded for $t \in [0,T]$. \end{as}
If $K$ is a repulsive-attractive power law kernel,
 \[ K(x) = \frac{|x|^a}{a} - \frac{|x|^b}{b} \quad 2-d < b<a \ , \]
and the initial data is smooth, compactly supported, and radially symmetric, there exists an exact solution satisfying this assumption $\forall \ T>0$ \cite[Theorems 7 and 8]{BalagueCarrilloLaurentRaoul}.

The above assumptions guarantee the following regularity of the velocity field and particle trajectories. We give the proofs of these lemmas in Section \ref{reg velocity field section}.
\begin{lemma}[regularity of velocity field]\label{reg velocity field lemma}
The velocity field $\vec{v} = \grad K * \rho$ and its divergence $\grad \cdot \vec{v} = \Delta K*\rho$ belong to $C^m(\Rd) \cap C^{L}(\Rd)$ and for any $|\beta| \leq m$,
\[ |\partial^\beta \vec{v}(x,t)| \leq C (1+|x|^{(S-|\beta|)_+}) \ ,  \quad  |\partial^\beta \grad \cdot \vec{v}(x,t)| \leq C (1+|x|^{(S-1 - |\beta|)_+}) \ . \]
The constant depends on the kernel, exact solution, dimension, $\beta$, $T$, and $R_0$.\end{lemma}

\begin{lemma}[regularity of particle trajectories] \label{particlelemma}
For $\alpha \in B_{R_0+2}$, the particle trajectories $X^t(\alpha)$ and their temporal inverses $X^{-t}(\alpha)$ uniquely exist, are continuously differentiable in time, and are $C^L$ in space. The Jacobian determinants ${J^t(\alpha) = \det \grad_\alpha X^t(\alpha)}$ and their inverses $J^{-t}(\alpha)$ are $C^{L-1}$ in space and satisfy
\begin{align} \label{conservationofmassformula}|J^t(\alpha)| \rho(X^t(\alpha),t) = \rho_0(\alpha) \ , \quad \forall \alpha \in B_{R_0+2} \ .
\end{align}
When $s=1-d$, the above holds with $B_{R_0+2}$ replaced by $\Rd$.
\end{lemma}

We now state our main theorem, quantifying the convergence of the blob method.
\begin{theorem} \label{convergence theorem}
Suppose that the kernel, mollifier, and exact solution satisfy Assumptions  \ref{kernel assumption}, \ref{mollifier assumption}, and \ref{exact sol reg as} for $m\geq 4$ and $L \geq \max \{ d+2, 4\}$. Define 
\begin{align} \label{delta growth us}
G_L(\delta) =\begin{cases}
1 & \mbox{if } L < s + d \ , \\ 
|\log(\delta)| & \mbox{if } L = s +d \ ,  \\
\delta^{-(L-s-d)} &\mbox{if } L > s +d \ . \end{cases} 
\end{align}
 Suppose $1\leq p<+\infty$ and for some $\frac{1}{2}<q<1$, $0 < h^q < \delta \leq \frac{1}{2}$.
Then the quantities $(\tilde{X}, \tilde{v}, \grad \cdot \tilde{v}, \tilde{\rho})$ which comprise the blob method exist for all $t \in [0,T]$ and satisfy
\begin{align*}
\|X(t) - \tilde{X}(t)\|_{L^p_h(B_{R_0})} &\leq C (\delta^m + G_L(\delta)h^L )\ ,\\ 
\|v(t) - \tilde{v}(t)\|_{L^p_h(B_{R_0})} &\leq C(\delta^m + G_{L+1}(\delta)h^L)\ , \\ 
\|\rho(t) - \tilde{\rho}(t)\|_{W^{-1,p}_h} &\leq C(\delta^m +   G_{L+1}(\delta)h^L)\ ,
\end{align*}
provided that for some $\epsilon >0$,
 \begin{align}
 \label{tech cond}
C(1+2d)(\delta^m + G_{L+1}(\delta) h^L) &< \delta^2 h^{1+\epsilon} /2  \ .
\end{align}
The constant depends on the exact solution, the kernel, the mollifier, the dimension, $T$, $R_0$, $q$, and $p \in [1, +\infty)$.
\end{theorem}

\begin{remark}[polynomial kernels]
By Remark \ref{mollifierAccuracy}, if $K(x)$ is a polynomial of order no more than $m$, $\grad K_\delta = \grad K$ and $\Delta K_\delta = \Delta K$. Thus, the error due to regularizing the kernel is zero, and the error of the blob method consists entirely of the discretization error. In this case, all error bounds in the theorem become $C h^L$.
\end{remark}

If $L = +\infty$, the following corollary shows that the blob method provides arbitrarily high order rates of convergence, depending on the accuracy of the mollifier.

\begin{corollary} \label{firstconvergencethm}
Suppose that Assumptions \ref{mollifier assumption}, \ref{kernel assumption}, and that \ref{exact sol reg as} hold for $m \geq 4$, $L=+\infty$, and $\frac{1}{2} <q<1$.
Then for $1\leq p<+\infty$ and $\delta = h^q$, there exists $h_0$ such that for all $0< h\leq h_0$,
\begin{align*}
\|X(t) - \tilde{X}(t)\|_{L^p_h(B_{R_0})} \leq Ch^{mq} \ , \quad \|\rho(t) - \tilde{\rho}(t)\|_{W^{-1,p}_h} \leq Ch^{mq} \ .
\end{align*}
\end{corollary}
\begin{proof}
We first verify that condition (\ref{tech cond}) from Theorem \ref{convergence theorem} holds. Since $\frac{1}{2} < q< 1$ and $m \geq 4$, there exists $ \epsilon >0$ so that
\begin{align} \label{tech cond 1}
\frac{1}{2} + \frac{\epsilon}{2} < q \implies 2q+1 + \epsilon < 4q \implies  2q+1 + \epsilon < mq \ .
\end{align}
Likewise, since the mollifier satisfies Assumption \ref{mollifier assumption} for all $L$ and the kernel satisfies Assumption \ref{kernel assumption} for $s \geq 1-d$, we may choose $L$ large enough so that $L>s+d$ and
\begin{align} \label{tech cond 2}
 L-q(L +1-s -d) = L(1-q) + q(s + d -1) > 2q + 1  + \epsilon \ .
 \end{align}
Combining (\ref{tech cond 1}) with (\ref{tech cond 2}) shows that $\exists \ h_0$ so that, for all $0< h\leq h_0$, (\ref{tech cond}) holds.

Finally, choosing $L$ large enough so that
\[ L-q(L +1-s -d) = L(1-q) + q(s + d -1) > mq\ , \]
we conclude that for all $0 < h< h_0$,
 \begin{alignat*}{2}
\|X(t) - \tilde{X}(t)\|_{L^p_h(B_{R_0})} &\leq C (\delta^m + G_L(\delta)h^L ) &&\leq C h^{mq} \\ 
\|\rho(t) - \tilde{\rho}(t)\|_{W^{-1,p}_h} &\leq C(\delta^m +   G_{L+1}(\delta)h^L) &&\leq C h^{mq}.
\end{alignat*}
\end{proof}

The proof of Theorem \ref{convergence theorem} relies on the following propositions concerning the consistency and stability of the blob method. All constants depend on the exact solution, the kernel, the mollifier, the dimension, $T$, $R_0$, $q$, and $p \in (1, + \infty)$.
\begin{proposition}[consistency] \label{consistency prop}
For $0\leq t \leq T$ and $G_L(\delta)$ defined by (\ref{delta growth us}),
\begin{align*}
\|v(t) - v^h(t)\|_{L^\infty_h(B_{R_0})}  &\leq  C (\delta^m+  G_L(\delta)h^L) \ , \\
\|\grad \cdot v(t) \rho(t)- \grad \cdot v^h(t)\rho(t)\|_{L^\infty_h}  &\leq  C ( \delta^m+  G_{L+1}(\delta) h^L) \ .
\end{align*}
\end{proposition}
\begin{proposition}[stability of velocity] \label{stability velocity prop}
For $0 \leq t \leq T$, $1< p<+\infty$, \\ if $\|X(t) - \tilde{X}(t)\|_{L^\infty_h(B_{R_0})} \leq \delta$, then $\|v^h(t) - \tilde{v}(t)\|_{L^p_h(B_{R_0})} \leq C \|X(t) - \tilde{X}(t)\|_{L^p_h(B_{R_0})}$.
\end{proposition}

\begin{proposition}[stability of divergence of velocity] \label{stability density prop}
For $0 \leq t \leq T$, ${1< p<+\infty}$, if $\|X(t) - \tilde{X}(t)\|_{L^\infty_h(B_{R_0})} \leq \delta^2$ and $\|\rho(t) - \tilde{\rho}(t)\|_{L^\infty_h} \leq \delta^2$, then 
\[ \|\grad \cdot v^h( t)\rho( t) - \grad \cdot \tilde{v}(t)\tilde{\rho}(t) \|_{W^{-1,p}_h} \leq C \|X(t) - \tilde{X}(t)\|_{L^p_h(B_{R_0})} +  C \|\rho(t) - \tilde{\rho}(t)\|_{W^{-1,p}_h} \ . \]
\end{proposition}

We now show how Theorem \ref{convergence theorem} follows from these propositions.
\begin{proof}[Proof of Theorem \ref{convergence theorem}]
By Proposition \ref{discrete norm prop}, (c) and (d), it suffices to prove the result for $p$ sufficiently large. Let $T_0>0$ be small enough so that the quantities $(\tilde{X}, \tilde{v},  \grad \cdot \tilde{v},\tilde{\rho})$  exist for $t \in [0,T_0]$.
Define the particle error $e_i(t) = X_i(t) - \tilde{X}_i(t)$, the density error $f_i(t) = \rho_i(t) - \tilde{\rho}_i(t)$, and
\begin{align} \label{defofTstar}
T^*= \min \{ T, T_0, \inf\{ t: \|e(t)\|_{L^\infty_h(B_{R_0})} \geq \delta^2 \}, \inf\{ t: \|f(t)\|_{L^\infty_h(B_{R_0})} \geq \delta^2 \} \} \ .
\end{align}

To bound $e(t)$ and $f(t)$, we first bound their time derivatives and then apply Gronwall's inequality. Since the $L^p_h(B_{R_0})$ norm of $e(t)$ is a finite sum, we may pass the time derivative under the norm to obtain \begin{align} \label{conv pf 00}
 \frac{d}{dt}\|e(t)\|_{L^p_h(B_{R_0})}  \leq \left\|\frac{d}{dt} e(t) \right\|_{L^p_h(B_{R_0})} = \|v(t) - \tilde{v}(t)\|_{L^p_h(B_{R_0})}  \ .
 \end{align}
By Proposition \ref{discrete norm prop} (a) and the fact that bounded support of the density causes the $L^p_h$ norm below to be a finite sum,
\[ \lim_{h\to 0}  \left\| \frac{f(t+h) - f(t)}{h} - \frac{d}{dt} f(t) \right\|_{W^{-1,p}_h} \leq \lim_{h\to 0}  \left\| \frac{f(t+h) - f(t)}{h} - \frac{d}{dt} f(t) \right\|_{L^p_h} =0 \ . \]
Thus, by the reverse triangle inequality
 \begin{align}
 \label{conv pf 0}
 \frac{d}{dt}\|f(t)\|_{W^{-1,p}_h} &=\lim_{h\to 0} \frac{\|f(t+h)\|_{W^{-1,p}_h} - \|f(t)\|_{W^{-1,p}_h}}{h} \leq \lim_{h\to 0} \left\|\frac{f(t+h) - f(t)}{h}  \right\|_{W^{-1,p}_h} \nonumber \\
 &= \left\| \frac{d}{dt} f(t)\right\|_{W^{-1,p}_h}  = \|\grad \cdot v(t)\rho(t) - \grad \cdot \tilde{v}(t)\tilde{\rho}(t) \|_{W^{-1,p}_h}  \ . \end{align}
For $0 \leq t \leq T^*$, we combine the consistency estimates of Proposition \ref{consistency prop} with stability estimates of Propositions \ref{stability velocity prop} and \ref{stability density prop} to obtain for $0 \leq t \leq T^*$,
\begin{align} \label{conv pf 1}
\|v(t) - \tilde{v}(t)\|_{L^p_h(B_{R_0})} &\leq C_1( \|e\|_{L^p_h(B_{R_0})} + \delta^m+ G_{L}(\delta)h^L ) \ , \\
 \label{conv pf 2}  \|\grad \cdot v(t)\rho(t) - \grad \cdot \tilde{v}(t)\tilde{\rho}(t) \|_{W^{-1,p}_h} &\leq C_1 (\|e\|_{L^p_h(B_{R_0})} +  \|f\|_{W^{-1,p}_h} + \delta^m+  G_{L+1}(\delta)h^L) . \hspace{-2mm}
\end{align}
Applying Gronwall's inequality to (\ref{conv pf 00}) and (\ref{conv pf 1}), we conclude for $0 \leq t \leq T^* \leq T$,
\[ \|e\|_{L^p_h(B_{R_0})} \leq C_1Te^{C_1 T}(\delta^m + G_L(\delta)h^L) \ .  \]
Substituting this into (\ref{conv pf 1}) gives $\|v(t) - \tilde{v}(t)\|_{L^p_h(B_{R_0})} \leq C(\delta^m + G_L(\delta)h^L)$.\\
{Applying} Gronwall's inequality a second time to (\ref{conv pf 0}) and (\ref{conv pf 2}),
\[ \| f\|_{W^{-1,p}_h} \leq C_1Te^{C_1 T} (\|e\|_{L^p_h(B_{R_0})} + \delta^m + G_{L+1}(\delta)h^L) \leq C(\delta^m + G_{L+1}(\delta)h^L)\ . \]

We now show that, in fact, $T_0 = T^*$ and $T^* = T$, so the above inequalities hold on the interval $[0,T]$. First, by the bounded support of $f$ and Proposition \ref{discrete norm prop} (b),
\begin{align} \label{relation to uniform norm}
 \|e\|_{L^\infty_h(B_{R_0})} \leq h^{-d/p} \|e\|_{L^p_h(B_{R_0})}  \ \text{ and } \ 
 \|f\|_{L^\infty_h} \leq (1+2d)h^{-1-d/p} \|f\|_{W^{-1,p}_h}  \ . 
\end{align}
If $T_0 < T^*$, then at least one of the quantities $(\tilde{X}, \tilde{v},  \grad \cdot \tilde{v},\tilde{\rho})$ becomes unbounded at $t=T_0$. Both $\tilde{v}$ and $\grad \cdot \tilde{v}$ remain bounded as long as the approximate particle trajectories remain bounded, and both $\tilde{X}$ and $\tilde{\rho}$ must remain bounded on $[0,T^*]$ by the above inequalities. Thus, $T_0 = T^*$.

Now, we show $T^* = T$. Fix $\epsilon>0$ so that (\ref{tech cond}) holds.
 Let $p\geq 1$ be large enough so that $d/p < \epsilon$. Then,
\begin{alignat*}{2}
 \|e\|_{L^\infty_h(B_{R_0})} &< h^{-\epsilon}\|e\|_{L^p_h(B_{R_0})}  &&<\delta^2 h/2 < \delta^2/2  \ ,
  \\   \|f\|_{L^\infty_h} &<  (1+2d) h^{-1-\epsilon} \|f\|_{W^{-1,p}_h}  &&<\delta^2/2 \ .
\end{alignat*}
Thus, for all $t \in [0,T^*]$, $\|e(t)\|_{L^\infty_h(B_{R_0})},  \|f(t)\|_{L^\infty_h} < \delta^2/2$, so by (\ref{defofTstar}), $T^* = T$.
\end{proof}

To complete our proof of Theorem \ref{convergence theorem}, it remains to show Propositions \ref{consistency prop}, \ref{stability velocity prop}, and \ref{stability density prop}. We prove these in Sections \ref{consistency section} and \ref{stability section}. We conclude the current section with three lemmas that play an important role in the remaining estimates. The first lemma is a standard result estimating quadrature error.
\begin{lemma}[quadrature error] \label{trap quad}
Given $g \in C_c^l(\Rd)$, $l > d$,
\[ \left| \int_\Rd g(x) dx - \sum_{j \in \mathbb{Z}^d} g(jh)h^d \right| \leq C_{l,d} \|g\|_{W^{l,1}(\Rd)} h^l \ . \]
\end{lemma}

\begin{proof}
See Anderson and Greengard \cite[Lemma 2.2]{AndersonGreengard}.
\end{proof}

Next, we quantify the regularity of $\grad K_\delta$. 
\begin{lemma}[regularity of $\grad K_\delta$ and $\Delta K_\delta$] \label{regofgradKdelta}
$\grad K_\delta$ and $\Delta K_\delta$ belong to $C^L(\Rd)$, and 
$\partial^\beta \grad K_\delta(x) = {\grad K * \partial^\beta \psi_\delta(x)}$ for all $|\beta| \leq L$.
\end{lemma}

The third lemma provides pointwise and $L^1$ estimates on $\grad K_\delta$. For the $L^1$ estimates, we allow an error term $g(x)$ of order $\delta$.
\begin{lemma}[regularized kernel estimates]\label{reg ker est lem}
Define $G_L(\delta)$ as in equation (\ref{delta growth us}), and fix $C'>0$. For $|\beta| \leq L$, $|g(x)| \leq C' \delta$, and $R >0$, there exists $C>0$ depending on the kernel, mollifier, dimension, $\beta$, $R$, and $C'$ so that 
\begin{align*}
\|\partial^\beta \grad K_\delta(x + g(x))\|_{L^1(B_R)} \leq C G_{|\beta|}(\delta) \ .
\end{align*}
\end{lemma}
See appendix Section \ref{reg kernel estimates} for the proof of Lemmas \ref{regofgradKdelta} and \ref{reg ker est lem}.

\subsection{Consistency} \label{consistency section}
To prove the consistency estimate of Proposition \ref{consistency prop}, we decompose the consistency error into a regularization error, due to the convolution with a mollifier, and a discretization error, due to the quadrature of the integral:
\begin{align} \label{consistencybreakup}
&|v(x,t) - v^h(x,t)|  \\
&\quad \leq |\grad K*\rho(x,t) - \grad K_\delta*\rho(x,t) |  + \left|\grad K_\delta*\rho(x,t) - \sum \grad K_\delta(x-X_j(t)){\rho_0}_j h^d \right|  \ ,\nonumber\\
&|\grad \cdot v(x,t) - \grad \cdot v^h(x,t)| \nonumber\\
&\quad \leq |\Delta K*\rho(x,t) - \Delta K_\delta*\rho(x,t) | + \left|\Delta K_\delta*\rho(x,t) - \sum \Delta K_\delta(x-X_j(t)){\rho_0}_j h^d \right| \nonumber\ . 
\end{align}
First, we bound the moment error.
\begin{proposition}[regularization error] \label{moment prop}
Fix $R>1$. Under the hypotheses of Theorem \ref{convergence theorem}, for $|x| < R$ and $0 \leq t \leq T$, 
\begin{align*}
 \left| \grad K* \rho(x,t) - \grad K_\delta*\rho(x,t) \right|   \leq C \delta^m \ , \quad 
  \left|\Delta K*\rho(x,t) - \Delta K_\delta* \rho(x,t) \right|  \leq C\delta^m
\end{align*}
The constant depends on the exact solution, the kernel, the mollifier, the dimension, $T$, $R_0$, and $R$.
\end{proposition}

\begin{proof}
It is a standard result (see for example Ying and Zhang \cite[Lemma 3.2.6]{YingZhang}) that if a mollifier $\psi$ is accurate of order $m$ and $f \in C^m(\Rd)$ has bounded derivatives, $|f(x) - f* \psi_\delta(x)| \leq C \delta^m$ for all $x \in \Rd$. Under our assumptions on $\psi$, the result continues to hold for $|x|< R$ if we merely require $f \in C^m(\Rd)$ and 
\begin{align} \label{reg pf 1}
 |\partial^\beta f(x)| \leq C (1+|x|^{(S-m)_+}) \ , \quad |\beta| = m \ . 
 \end{align}

This follows from Taylor's formula with integral remainder,
\[ f(x-y)  \leq f(x)+ \sum_{|\gamma|=1}^{m-1}  \frac{(-1)^{|\gamma|} y^\gamma}{\gamma!}\partial^\gamma f(x) + m\sum_{|\beta| = m} \frac{(-1)^m y^\beta}{\beta!} \int_0^1 (1-t)^{m-1}\partial^\beta f(x-ty)dt  \ . \]
Inequality (\ref{reg pf 1}) implies for $|x|<R$, $0<t<1$,
\[|\partial^\beta f(x-ty) | \leq C(1+ |x-ty|^{(S - m)_+}) \leq C(1+(|x| + t|y|)^{(S-m)_+})\leq C_R(1+|y|^{(S-m)_+}) \ . \]
Thus,
 \[ \left| \int_0^1 (1-t)^{m-1}\partial^\beta f(x-ty)dt  \right| \leq C_{R}(1+|y|^{(S -m)_+} )\ . \]
By Assumption \ref{mollifier assumption} on the $m$th order accuracy of $\psi$,
\[\int y^\gamma \psi_\delta(y) dy = \delta^{-d} \int y^\gamma \psi(y/\delta) dy = \delta^{|\gamma|} \int y^\gamma \psi(y) dy =  \begin{cases} \delta^{|\gamma|} &\text{ if } |\gamma| =0 \ , \\0 &\text{ if } |\gamma| \in [1, m-1] \ . \end{cases} \]
and $\int |y|^m |\psi_\delta(y)| dy < C \delta^m$.
We also have that ${\int |y|^{S} |\psi_\delta(y)|dy < C \delta^S}$ when $S > 0$. If $S> m$, $\delta^S < \delta^m$. Thus, integrating Taylor's formula against $\psi_\delta(y)$ gives
\[ |f*\psi_\delta(x) - f(x)|\leq  C \delta^m   \ . \]

It remains to show that (\ref{reg pf 1}) holds for $f = \grad K *\rho$ and $f = \Delta K *\rho$. This follows from Lemma \ref{reg velocity field lemma}.
\end{proof}

\begin{proposition}[discretization error] \label{discretization prop}
Fix $R>1$. Under the hypotheses of Theorem \ref{convergence theorem}, with $G_L(\delta)$ defined by (\ref{delta growth us}), for $|x| < R$  and $0 \leq t \leq T$,  
\begin{align*}
\left|\int \grad K_\delta(x-y) \rho(y,t) dy - \sum \grad K_\delta(x-X_j(t)){\rho_0}_j h^d \right|  &\leq  C G_L(\delta) h^L \\
\left|\int \Delta K_\delta(x-y) \rho(y,t) dy - \sum \Delta K_\delta(x-X_j(t)){\rho_0}_j h^d \right| &\leq  CG_{L+1}(\delta) h^L
\end{align*}
The constant depends on the exact solution, the kernel, the mollifier, the dimension, $T$, $R_0$, and $R$.
\end{proposition}

\begin{proof}
We prove the two estimates simultaneously by bounding
\[ e_d(x,t) = \int \K_\delta(x-y) \rho(y,t) dy - \sum \K_\delta(x-X_j(t)){\rho_0}_j h^d \ , \]
where $\K_\delta = \grad K_\delta$ or $\K_\delta = \Delta K_\delta$.
By Lemma \ref{particlelemma}, for all $\alpha \in B_{R_0+2}$, the particle trajectories, $X^t(\alpha)$ belong to $C^L$ and satisfy
\[  \int \K_\delta(x-y) \rho(y,t) dy  = \int \K_\delta(x-X^t(\alpha)) \rho_0(\alpha) d\alpha \ . \]
By Assumption \ref{exact sol reg as}, $\rho \in C^L(\Rd)$, and by Lemma \ref{regofgradKdelta}, $\grad K_\delta$ and $\Delta K_\delta$ also belong to $C^L(\Rd)$. Since $L \geq d+2$, we may bound the error using  Lemma \ref{trap quad},
\begin{align*}
|e_d(x,t)| &= \left| \int \K_\delta(x-X^t(\alpha)) \rho_0(\alpha) d\alpha - \sum \K_\delta(x-X_j(t)){\rho_0}_j h^d \right|  \ , \\
&\leq C_{L, d} \|\K_\delta(x-X(\cdot,t)) \rho_0(\cdot) \|_{W^{L,1}(\Rd)} h^L \ .
\end{align*} 
By  Assumption \ref{exact sol reg as}, $|X^t(\alpha)|<C$ for all $\alpha \in \supp \rho_0$, $t \in [0,T]$, so for $|x|< R$, ${|x - X^t(\alpha))| \leq R + C}$.
Thus, applying the chain and product rules give 
\begin{align*}
|e_d(x,t)| \leq C \| \K_\delta\|_{W^{L,1}(B_{R+C})} h^L \ .
\end{align*}
The result then follows from the regularized kernel estimates, Lemma \ref{reg ker est lem}.
\end{proof}

Finally, we combine the previous two propositions to prove Proposition \ref{consistency prop}.
\begin{proof}[Proof of Proposition \ref{consistency prop}]
 By (\ref{consistencybreakup}), Proposition \ref{moment prop}, and Proposition \ref{discretization prop}, taking $x = \tilde{X}_i(t)$ for $|ih| < R_0$,
\begin{align*}
| v_i(t) - v^h_i(t)|  &\leq  C (\delta^m+  G_L(\delta)h^L) \ , \\
|\grad \cdot v_i(t) - \grad \cdot v_i^h(t)| &\leq  C ( \delta^m+  G_{L+1}(\delta) h^L) \ .
\end{align*}
By Assumption \ref{exact sol reg as}, $\rho$ is bounded and supported in $B_{R_0}$. Hence,
\begin{align*}
\|v(t) - v^h(t)\|_{L^\infty_h(B_{R_0})}  &\leq  C (\delta^m+  G_L(\delta)h^L) \ , \\
\|\grad \cdot v(t) \rho(t)- \grad \cdot v^h(t)\rho(t)\|_{L^\infty_h} &\leq  C ( \delta^m+  G_{L+1}(\delta) h^L) \ .
\end{align*}
\end{proof}

\subsection{Stability} \label{stability section}
We now turn to the proof of stability, which relies on the following lemma relating the $L^p_h$ norm of a discrete convolution to the $L^p$ norm of a convolution. This lemma allows us to apply classical results for integral operators to conclude stability of the method. Our approach is strongly influenced by previous work on the stability of classical vortex blob methods for the Euler equations by Beale \cite{Beale} and Beale and Majda \cite{BealeMajda1, BealeMajda2}.

Let $Q_i$ be the $d$-dimensional cube with side length $h$ centered at $ih \in h \mathbb{Z}^d$ and define $Q^t_i = X^t(Q_i)$. Since $h  < \frac{1}{2}$, if we define $\Omega = \cup_{|ih|< R_0+1} Q_i$, then $\Omega \subseteq B_{R_0+2}$ and Lemma \ref{particlelemma} ensures that
\begin{align} \label{bimagnitude0}
|Q^t_i| = \int_{Q^t_i} dy = \int_{Q_i} |J^t(\alpha)| d\alpha  
\end{align}
and ${C_1h^d \leq |Q^t_i| \leq C_2 h^d}$, so $\{Q^t_i\}$ partitions $X^t(\Omega)$ for all $t \in [0,T]$.

\begin{lemma} \label{discreteconvolutionlemma} Let $G_L(\delta)$ be defined as in equation (\ref{delta growth us}) and let $J_j(t) = J^t(jh)$. Consider $y_{ij}(t) \in C([0,T],L^\infty_h(\Rd \times \Rd))$ and $g_j(t)\in C([0,T], L^\infty_h(\Rd))$,
where $\|y(t)\|_{L^\infty_h} \leq 2\delta$ and the support of $g_j(t)$ is contained in $B_{R_0-1}$ for all $t \in [0,T]$.

Then for any multiindex $|\beta| \leq L-1$  and  $1 < p \leq + \infty$, there exists $C, R>0$ depending on the exact solution, the kernel, the mollifier, the dimension, $\beta$, $T$, $R_0$, $p$, and $g$ so that for all $t \in [0,T]$,
\begin{align*}
& \left\| \sum_{|jh| < R_0} \partial^\beta \grad K_\delta(X_i(t) - X_j(t) + y_{ij}(t))g_j(t) h^d  \right\|_{L^p_h(B_{R_0+1})} \\
 &\quad \leq C \left(\|\partial^\beta \grad K_\delta *g(t)\|_{L^p(B_{R})} + \delta G_{|\beta|+1}(\delta) \|g(t)\|_{L^p} \right)\ .
 \end{align*}
 If $s = 1-d$, the above holds with $L^p(B_{R_0 +1})$ replaced by $L^p(B_{R_0 + C'})$ for all $C'\geq 0$ and the constants depend on $C'$.
\end{lemma}

\begin{proof}
Define $w_i(t) = \sum_{|jh|<R_0} \partial^\beta \grad K_\delta(X_i - X_j + y_{ij}(t))g_j(t) J_j(t) h^d$, and for $x \in Q^t_i$, $y \in Q^t_j$, define  $G(x,y,t) = \partial^\beta \grad K_\delta(X_i - X_j + y_{ij}(t))$ and $g(y,t) = g_j(t)$. 
Since $J_j(t)$ is bounded below, it is enough to bound $\|w_i(t)\|_{L^p_h(B_{R_0+1})}$.
For $x \in Q^t_i$,
\begin{align*}
w_i(t)
&= \int \partial^\beta \grad K_\delta(x-y) g(y,t) dy + \int [G(x,y,t) - \partial^\beta \grad K_\delta(x-y)] g(y,t) dy \nonumber \\ &\quad \quad + \sum_{j} G(x,X_j(t),t) g_j(t) (|J_j(t)| h^d - |Q_j^t|) = a(x,t) +b(x,t) + c(x,t)  \ .
\end{align*}
By definition, $|w_i(t)| \leq \|a(t) +b(t) + c(t)\|_{L^\infty(Q_i^t)}$,
and for $1< p< +\infty$,
\begin{align*}
 |w_i(t)|^p h^d \leq \frac{h^d}{|Q_i^t|} \int_{Q_i^t} |a(x,t)+b(x,t)+c(x,t)|^p dx \leq C \|a(t)+b(t)+c(t)\|_{L^p(Q_i^t)}^p \ .
 \end{align*}
 Thus, for $1< p \leq + \infty$,
\begin{align*}
 \|w(t) \|_{L^p_h(B_{R_0+1})} \leq  C \|a(t)+b(t)+c(t)\|_{L^p(X^t(\Omega))}\ .
 \end{align*}
 By Assumption \ref{exact sol reg as}, there exists $R>0$ so that for $\alpha \in \Omega$, $|X^t(\alpha)|<R$.
Since $\|a(t)\|_{L^p(B_{R})} = \|  \partial^\beta \grad K_\delta*g(t) \|_{L^p(B_{R})}$,
 this gives the first term in our bound.
 
 It remains to control $b(t)$ and $c(t)$.
By Lemma \ref{particlelemma}, there exists $C>0$ so that for $x \in Q_i$, $y \in Q_j$,  $ih, jh \in \Omega$, $t \in [0,T]$,
\[ |X_i(t)-x|+|Y_i(t)-y| \leq Ch < C \delta \ . \]
Since $\|y(t)\|_{L^\infty_h} \leq 2\delta$, by the mean value theorem, there exists $z(x,y,t)$ with $|z(x,y,t)| \leq (C+2) \delta$ so that for $x \in X^t(\Omega)$,
\begin{align*}
 |b(x,t) |\leq \sum_{|\gamma| = |\beta| +1} \int \left|\partial^\gamma \grad K_\delta (x-y +z(x,y,t)) \right| (C+2)\delta |g(y,t)| dy \ .
 \end{align*}
By the regularized kernel estimates, Lemma \ref{reg ker est lem}, for all $y \in B_{R_0}$, $|\gamma| = |\beta|+1$,
\[ \|\partial^\gamma \grad K_\delta(x-y + z(x,y,t))\|_{L^1(B_R)} \leq C G_{|\beta|+1}(\delta) \ . \]
Therefore, by a classical inequality for integral operators \cite[Theorem 6.18]{Folland},
\begin{align*}
 \|b(x,t) \|_{L^p(X^t(\Omega))}\leq C \delta G_{|\beta|+1}(\delta) \|g(t)\|_{L^p} \ . 
 \end{align*}

Finally, we bound $c(t)$. By Lemma \ref{particlelemma}, $J^t(\alpha) \in C^{1}(\Rd)$. Hence,
\begin{align*}
 \|Q_j^t| - |J_j(t)| h^d | \leq \left|\int_{Q_j} |J^t(\alpha)| - |J_j(t)| d \alpha  \right| \leq C h^{d+1} \ .
\end{align*}
Therefore,
\begin{align*} 
|c(x,t)|   &= \left| \sum_{j} G(x,X_j(t),t) g_j(t) (|J_j(t)| h^d - |Q_j^t|) \right|  \leq C h^{d+1-d} \left| \sum_{j} G(x,X_j(t),t) g_j(t) h^d \right|    \\
&\leq C \left|\int G(x,y,t) g(y,t) dy \right|
 \leq C (|a(x,t)| + |b(x,t)|) \ .\nonumber
\end{align*}
Combining this with the bounds on $a(x,t)$ and $b(x,t)$ gives the result.

If $s = 1-d$, the above holds with ${R_0 +1}$ replaced by ${R_0 + C'}$ for all $C'\geq 0$.
\end{proof}

We now consider stability of the velocity, Proposition \ref{stability velocity prop}.
\begin{proof}[Proof of Proposition \ref{stability velocity prop}]
Define $e_i(t) = X_i(t) - \tilde{X}_i(t)$. As our estimates are uniform in $t \in [0,T]$, we suppress the dependence on time.

First, we decompose the difference between $v^h$ and $\tilde{v}$, isolating the effects of approximate and exact particle trajectories. By the mean value theorem, there exists $|y^{(1)}_{ij}| \leq |e_j| \leq \delta$ and $|y^{(2)}_{ij}| \leq 2 \delta$ so that
\begin{align}  \label{v1v2def}
v^h_i - \tilde{v}_i &=   \Sigma_{j} \grad K_\delta(X_i - \tilde{X}_j) {\rho_0}_j h^d- \Sigma_{j } \grad K_\delta(X_i-X_j) {\rho_0}_j h^d  \\
&\quad + \Sigma_{j } \grad K_\delta(\tilde{X}_i - \tilde{X}_j) {\rho_0}_j h^d - \Sigma_{j } \grad K_\delta(X_i - \tilde{X}_j) {\rho_0}_j h^d \nonumber \\
  &= \Sigma_{j } D^2 K_\delta \left(X_i - X_j + y^{(1)}_{ij} \right) e_j {\rho_0}_j h^d \nonumber\\ 
  &\quad + e_i \Sigma_{j} D^2 K_\delta \left(X_i - X_j + y^{(2)}_{ij}\right) {\rho_0}_j h^d  \ . \nonumber \\
&= v^{(1)}_i +e_iv^{(2)}_i  \ . \nonumber
\end{align}
Because it will be useful in the next proposition, we bound the $L^p_h$ and $L^\infty_h$ norms of $v^{(1)}$  and $v^{(2)}$ over $B_{R_0+1}$, instead of $B_{R_0}$.
By two applications of Lemma \ref{discreteconvolutionlemma} with $|\beta| = 1$, $g^{(1)}_j = e_j {\rho_0}_j$, and $g^{(2)}_j = {\rho_0}_j$,  there exists $C,R>0$ so that
\begin{align*} 
\|v^{(1)}\|_{L^p_h(B_{R_0+1})} &\leq C \left(\|D^2 K_\delta *g^{(1)}\|_{L^p(B_{R})} + \delta G_{2}(\delta) \|g^{(1)}\|_{L^p} \right)\ ,  \\
\|v^{(2)}\|_{L^\infty_h(B_{R_0+1})} &\leq C \left(\|D^2 K_\delta *g^{(2)}\|_{L^\infty(B_{R})} + \delta G_{2}(\delta) \|g^{(2)}\|_{L^\infty} \right)\ .
\end{align*}

To complete the proof, it suffices to show the first term is bounded by $C\|e\|_{L^p_h(B_{R_0})}$ and the second term is bounded by $C$.
We bound $v^{(1)}$ by showing that there exists $C>0$ so that $\|D^2 K_\delta*g^{(1)}\|_{L^p(B_{R})} \leq C \|g^{(1)}\|_{L^p}$. By the linearity of convolution and differentiation, if we show the result for $K= \sum_{n=1}^N K_n$ for $N=1$, this implies the result for $N>1$. Thus, we may assume $K= K_1$ for $s = S \geq 1-d$.

When $s = 1-d$, we can apply the Calder\'on Zygmund inequality. When $s>1-d$, $|D^2K(x)| \leq C|x|^{s-1} \in L^1_\loc(\Rd)$, and we can apply Young's inequality. Thus,
\[\|D^2 K_\delta *g^{(1)}\|_{L^p(B_{R})} \leq \|D^2 K*g^{(1)}\|_{L^p(B_{R})} \|\psi_\delta\|_{L^1} \leq C \|g^{(1)} \|_{L^p} \ . \]
Since $\delta G_{2}(\delta) \leq  C$, we use the definition  $g_j^{(1)} = e_j \rho_j$, to conclude
 \begin{align} \label{v1bound}
 \|v^{(1)}\|_{L^p_h(B_{R_0+1})} \leq C \|g^{(1)}\|_{L^p} &\leq C \|e\|_{L^p_h(B_{R_0})}  \ . 
\end{align}

We now turn to $v^{(2)}$. For $y \in X^t(\Omega)$, 
\[ |g^{(2)}(y) - \rho_0(X^{-t}(y))| \leq \sup_{j, \alpha \in Q_j} |{\rho_0}_j - \rho_0(\alpha)| \leq Ch \ . \]
Hence,
\begin{align*}
\left|\int D^2 K_\delta(x-y) g^{(2)}(y) dy \right| &\leq C \max_{|\gamma|=1} \left| \int \grad K_\delta(x-y) \partial^\gamma \rho_0(X^{-t}(y))dy \right| \\
&\quad +  \left| \int D^2 K_\delta(x-y) (g^{(2)}(y)- \rho_0(X^{-t}(y)))dy \right| \ .
\end{align*}
By Assumption \ref{exact sol reg as}, Lemma \ref{particlelemma}, and Lemma \ref{reg ker est lem}, and the fact that $h < \delta$, the above quantity is bounded by a constant for $|x|<R$.
Thus, 
\begin{align} \label{linftyboundstability}
 \|v^{(2)} \|_{L^\infty_h(B_{R_0+1})}  \leq \|D^2 K_\delta *g^{(2)}\|_{L^\infty(B_{R})} + \delta G_{2}(\delta) \|g^{(2)}\|_{L^\infty} \leq C \ . 
\end{align}
\end{proof}

We now prove stability of the divergence of the velocity.
\begin{proof}[Proof of Proposition \ref{stability density prop}]
As in the previous proof, define $e_i(t) = X_i(t) - \tilde{X}_i(t)$ and $f_i(t) = \rho_i(t) - \tilde{\rho}_i(t)$. Since our estimates are uniform in time, we suppress dependence on $t$.

We decompose the difference between $\grad \cdot v^h_i\rho_i$ and $\grad \cdot \tilde{v}_i\tilde{\rho}_i$ as
\begin{align}
\grad \cdot v^h_i\rho_i - \grad \cdot \tilde{v}_i\tilde{\rho}_i  & \leq \left(\grad \cdot v^h_i - \grad \cdot \tilde{v}_i \right)\rho_i   + \grad \cdot v^h_i \left( \rho_i - \tilde{\rho}_i \right)  + \left(\grad \cdot \tilde{v}_i -\grad \cdot v^h_i \right) \left( \rho _i - \tilde{\rho}_i \right) \nonumber  \\
&\leq a_i + b_i + c_i \label{abc def}
\end{align}

First we bound the $W^{-1,p}_h$ norm of $a$ in terms of the $L^p_h$ norm of $e$. As in the proof of the stability of the velocity, we further decompose $\grad \cdot v^h_i - \grad \cdot \tilde{v}_i$,
\begin{align} \label{div vel decomp1}
\grad \cdot v^h_i - \grad \cdot \tilde{v}_i &= \sum_{j } \Delta K_\delta(X_i - \tilde{X}_j) {\rho_0}_j h^d - \sum_{j } \Delta K_\delta(X_i-X_j) {\rho_0}_j h^d   \\
&\quad + \sum_{j } \Delta K_\delta(\tilde{X}_i - \tilde{X}_j) {\rho_0}_j h^d - \sum_{j} \Delta K_\delta(X_i-\tilde{X}_j) {\rho_0}_j h^d   \nonumber \ .
\end{align}
By Taylor's theorem, there exists $|y^{(1)}_{ij}|\leq |e_j| \leq \delta^2$ and $|y^{(2)}_{ij}| \leq 2 \delta^2$ so that the above can be further decomposed as
\begin{align} \label{div vel decomp2}
%& \grad \cdot v^h_i - \grad \cdot \tilde{v}_i \nonumber \\
 & \sum_{j }\grad \Delta K_\delta \left(X_i - X_j \right) e_j {\rho_0}_j h^d + \sum_{|\gamma|=2}\sum_{j } \frac{1}{\gamma!}\partial^\gamma \Delta K_\delta(X_i - X_j + y^{(1)}_{ij}) e_j^\gamma {\rho_0}_j h^d  \\
&  +    \sum_{j }\grad \Delta K_\delta \left(X_i - X_j \right) e_i {\rho_0}_j h^d + C\sum_{|\gamma|=2}\sum_{j } \frac{1}{\gamma!} \partial^\gamma \Delta K_\delta(X_i - X_j + y^{(2)}_{ij}) e_i^\gamma {\rho_0}_j h^d \nonumber \\
 &= a^{(1)}_i + A^{(1)}_i + a^{(2)}_i +A^{(2)}_i \nonumber
 \end{align}
 With this decomposition, $a_i = \rho_i (a^{(1)}_i + A^{(1)}_i + a^{(2)}_i +A^{(2)}_i)$.
 
First, consider $A^{(1)}$ and $A^{(2)}$. By Lemma \ref{discreteconvolutionlemma} with $|\beta| = 3$, $g^{(1)}_j = e_j^\gamma {\rho_0}_j$, and $g^{(2)}_j = {\rho_0}_j$,  there exists $C,R>0$ so that
\begin{align*}
 \| A^{(1)}\|_{L^p_h(B_{R_0})} &\leq C \|D^4 K_\delta* g^{(1)}\|_{L^p(B_R)} + \delta G_4(\delta) \|g^{(1)}\|_{L^p} \ , \\
  \| A^{(2)}\|_{L^p_h(B_{R_0})} &\leq  C\|e^\gamma\|_{L^p_h(B_{R_0})} \left( \|D^4 K_\delta* g^{(2)}\|_{L^\infty(B_R)} + \delta G_4(\delta) \|g^{(2)}\|_{L^\infty} \right) \ .
 \end{align*}
 Since $|e^\gamma_i| \leq C  |e_i|^2 \leq C \delta^2 |e_i|$, by the regularized kernel estimates Lemma \ref{reg ker est lem}, we have for both $A^{(1)}$ and $A^{(2)}$,
\begin{align} \label{A1A2bound}
  \| A^{(\cdot)}\|_{L^p_h(B_{R_0})} \leq C (\delta^2 G_3(\delta)+ \delta^3 G_4(\delta)) \|e\|_{L^p_h(B_{R_0})} \leq C \|e\|_{L^p_h(B_{R_0})} \ .
 \end{align}
By Proposition \ref{discrete norm prop} (a) and the fact that $\rho_i$ is bounded and supported in $B_{R_0}$, this implies $\|\rho A^{(1)}\|_{W^{-1,p}_h}$ and $\|\rho A^{(2)}\|_{W^{-1,p}_h} \leq C \|e\|_{L^p_h(B_{R_0})}$.

Now, consider $\rho a^{(1)}$. For $\alpha \in Q_i$, define
  \begin{align*}
 F(\alpha) &= \sum_j \Delta K_\delta(X^t(\alpha) - X_j(t)) e_j \rho_{0_j} h^d \ .
  \end{align*}
Let $\grad^h_i$ be a finite difference operator of the form in Proposition \ref{discrete norm prop} (e) and suppose it is $l$th order accurate, i.e. $\|\grad^h_i F - \grad_\alpha F\|_{L^p_h(B_{R_0})} \leq Ch^l \|D_\alpha^{l+1} F\|_{L^p_h(B_{R_0 +lh})}$. Define $Z(\alpha,t) = \grad_\alpha X^t(\alpha)$ and $Z_i(t) = \grad_\alpha X^t(ih)$. Rewriting $a^{(1)}$ with a Lagrangian derivative and approximating it by this finite difference operator,
\begin{align} \label{stab pf 2}
 a^{(1)}_i = (Z_i^{-1}) \grad_\alpha F(ih) = (Z_i^{-1}) \grad_i^h F(ih) + (Z_i^{-1}) \left( \grad_\alpha F(ih)- \grad_i^h F(ih)\right)\ .
  \end{align}

To bound the first term in $\rho a^{(1)}$, let $\psi:\Rd \to [0,1]$ be a smooth function satisfying $\psi(x) = 1$ for $|x| \leq R_0+ 1/2$ and $\psi(x) = 0$ for $|x| \geq R_0+1$, and let $\psi_i = \psi(ih)$. Since $h<1/2$ and $\rho_i$ is supported in $B_{R_0}$, $\rho_i \grad^h_i F(ih) = \rho_i \grad^h_i (\psi_i F(ih))$. Therefore, by Proposition \ref{discrete norm prop} (e), Assumption \ref{exact sol reg as}, and Lemma \ref{particlelemma},
\begin{align*}
\| (\rho_i Z_i^{-1})\grad^h_i F(ih) \|_{W^{-1,p}_h} &\leq \|\rho Z^{-1}\|_{W^{1,\infty}_h} \|\grad^h_i (\psi_i F(ih)) \|_{W^{-1,p}_h} \leq C \| F(ih) \|_{L^p_h(B_{R_0+1})} \ .
\end{align*}
By the definition of $v^{(1)}$ (\ref{v1v2def}) and inequality (\ref{v1bound}) from the previous proof, $\|\rho a^{(1)} \|_{W^{-1,p}_h}$ is bounded by $C \|e\|_{L^p_h(B_{R_0})} $.

Next, we bound the second term in $\rho a^{(1)}$. Combining Proposition \ref{discrete norm prop} (a), Assumption \ref{exact sol reg as}, Lemma \ref{particlelemma}, and the $l$th order accuracy of $\grad_i^h$,
\begin{align*}
&\| (\rho_i Z_i^{-1}) \left( \grad_\alpha F(ih)- \grad_i^h F(ih)\right)\|_{W^{-1,p}_h}  \leq \|\rho Z^{-1}\|_{W^{1,\infty}_h} \|\grad_\alpha F(ih)- \grad_i^h F(ih)\| _{W^{-1,p}_h} \ , \\
&\leq C \|\grad_\alpha F(ih)- \grad_i^h F(ih)\|_{L^p_h(B_{R_0})} \leq Ch^l \|D^{l+1} F\|_{L^q_h(B_{R_0} +lh)}   \ .
\end{align*}

When $s = 1-d$, we choose $l \geq q/(1-q)$. Otherwise, we choose $l =1$.
We can now apply Lemma \ref{discreteconvolutionlemma} with $|\beta| = l+2$, $g_j = e_j {\rho_0}_j$ to control the right hand side. 
Combining this with Young's inequality and Lemma \ref{reg ker est lem} gives
\begin{align*} 
Ch^l \|D^{l+1} F\|_{L^q_h(B_{R_0} +lh)} &\leq C h^l \left(\|D^{l+3} K_\delta *g\|_{L^p(B_{R})} + \delta G_{l+3}(\delta) \|g\|_{L^p} \right) \\
&\leq C h^l G_{l+2}(\delta) \|g\|_{L^p} \leq C h^lG_{l+2}(\delta)  \|e\|_{L^p_h(B_{R_0})}\ .
\end{align*}
Since $\frac{1}{2}<q<1$ and $\delta \geq h^q$,
\[ h^lG_{l+2}(\delta) \leq\begin{cases}
h^l & \mbox{if } l+2 < s + d \\ 
q h^l |\log(h)| & \mbox{if } l+2 = s +d \\
h^l h^{q(s + d -l-2)} &\mbox{if } l+2 > s +d \  . \end{cases}  \]
When $s = 1-d$, $l \geq q/(1-q)$, and when $s>1-d$, $l =1$. Thus, $l \geq \max \{ 1, q(l+2- s -d) \}$ and the above quantity is bounded by a constant.

It remains to bound $ \| \rho a^{(2)}\|_{W^{-1,p}_h} \leq C \|a^{(2)}\|_{L^p_h(B_{R_0})}$. 
Since
\begin{align}\label{den stab pf 0}
\|a^{(2)}\|_{L^p_h(B_{R_0})} \leq C  \|e\|_{L^p_h(B_{R_0})} \left\|\Sigma_{j }\grad \Delta K_\delta \left(X_i - X_j \right) {\rho_0}_j h^d \right\|_{L^\infty_h(B_{R_0})} \ ,
\end{align}
it suffices to show
\begin{align} \label{den stab pf 1}
\left\|\Sigma_{j }D^2 \grad K_\delta \left(X_i - X_j \right) {\rho_0}_j h^d \right\|_{L^\infty_h(B_{R_0 +1})} \leq C \ .
\end{align}
 For any $N >d$, the quadrature Lemma \ref{trap quad} and the regularized kernel estimates Lemma \ref{reg ker est lem} imply
\begin{align*}
 &\left| \int_\Rd D^2 \grad K_\delta (x-X^t(\alpha)) \rho_0(\alpha) d\alpha - \sum_{j \in \mathbb{Z}^d} D^2 \grad K_\delta(x-X_j(t)) {\rho_0}_j h^d \right|  \\
 &\quad \leq C \| D^2 \grad K_\delta (x-X^t(\cdot)) \rho_0(\cdot)\|_{W^{N,1}(\Rd)} h^N \leq C h^N G_{N+2}(\delta)\ .
 \end{align*}
 As argued above, since $\frac{1}{2}<q<1$, $\delta \geq h^q$, choosing $N$ large enough so $N+2>s+d$ and $N >q/(1-q)$, the above quantity is bounded by a constant. Finally,
\begin{align*}
\left| \int_\Rd D^2 \grad K_\delta (x-X^t(\alpha)) \rho_0(\alpha) d\alpha \right| &=\left| \int_\Rd D^2 \grad K_\delta (x-y) \rho(y,t) dy \right| \ , \\
&\leq C \left| \int_\Rd  \grad K_\delta (x-y) D^2 \rho(y,t) d y \right| \leq C \ .
\end{align*}
Combining our estimates, we conclude
\[ \|a \|_{W^{-1,p}_h} \leq  \|\rho(a^{(1)} +A^{(1)} + a^{(2)} +A^{(2)}) \|_{W^{-1,p}_h} \leq C \| e\|_{L^p_h(B_{R_0})} \ . \]

Now that we have controlled $a$, the second and third terms in (\ref{abc def}) follow quickly. We seek to bound $\|b_i\|_{W^{-1,p}_h}$ in terms of $\|f_i\|_{W^{-1,p}_h}$. By assumption, $\rho_i$ vanishes outside of $B_{R_0}$, hence it suffices to show $\|\grad \cdot v^h\|_{W^{1,\infty}_h(B_{R_0})}$ is bounded by a constant. The fact that $\|\grad \cdot v^h\|_{L^\infty(B_{R_0})} \leq C$ follows from the bound on $v^{(2)}$ from the previous proof (\ref{v1v2def}, \ref{linftyboundstability}). The fact that $\| D^+_j\grad \cdot v^h\|_{L^\infty(B_{R_0})} \leq C$ follows from inequality (\ref{den stab pf 1}) above.

Finally, we turn to the last term in (\ref{abc def}). To bound $\|c\|_{W^{-1,p}_h}$ in terms of $\|e\|_{L^p_h}$, we may use the decomposition of $\grad \cdot v^h - \grad \cdot \tilde{v}  = a^{(1)} + A^{(1)} + a^{(2)} +A^{(2)}$ given by (\ref{div vel decomp1}, \ref{div vel decomp2}). By (\ref{A1A2bound}, \ref{den stab pf 0}, \ref{den stab pf 1}), and the fact that $\|  \rho_i - \tilde{\rho}_i\|_{L^\infty_h} \leq \delta^2 \leq C$,
\[ \|(A^{(1)} + a^{(2)} +A^{(2)})( \rho_i - \tilde{\rho}_i) \|_{W^{-1,p}_h} \leq  C \|A^{(1)} + a^{(2)} +A^{(2)}\|_{L^{p}_h(B_{R_0})}\leq C \|e\|_{L^p_h(B_{R_0})} \ . \]
By Lemma \ref{discreteconvolutionlemma} with $|\beta| = 2$ and $g_j = e_j {\rho_0}_j$,  there exists $C,R>0$ so that
\begin{align*}
&\|a^{(1)} ( \rho_i - \tilde{\rho}_i) \|_{W^{-1,p}_h} \leq \|\Sigma_{j }\grad \Delta K_\delta \left(X_i - X_j \right) e_j {\rho_0}_j h^d \|_{L^p_h(B_{R_0})}  \|  \rho_i - \tilde{\rho}_i\|_{L^\infty_h} \\
&\leq C\delta^2 ( \|D^3 K_\delta* g\|_{L^p(B_R)} + \delta G_3(\delta) \|g\|_{L^p}) \leq C (\delta^2 G_2(\delta)+ \delta^3 G_3(\delta)) \|e\|_{L^p_h(B_{R_0})}
\end{align*}
The above is bounded by $C \|e\|_{L^p_h(B_{R_0})} $. This completes the proof.
\end{proof}

\section{Numerics} \label{NumericsSection}
We now present several numerical examples in one and two dimensions for a range of kernels and initial data. These examples confirm the rate of convergence obtained in Corollary \ref{firstconvergencethm} and illustrate the varied phenomena  of solutions to the aggregation equation, including blowup and pattern formation. 

\subsection{Numerical implementation}
We implement the blob method in Python, using the NumPy, SciPy, and matplotlib libraries \cite{SciPy}. We approximate solutions to the ordinary differential equations which comprise our method using the VODE solver \cite{VODE}, which uses either a backward differentiation formula (BDF) method or an implicit Adams method, depending on the stiffness of the problem.

In one dimension, we use the mollifiers
\begin{align} \label{both 1d mollifiers}
\psi^{(4)}(x) = \frac{4}{3 \sqrt{\pi}} e^{-|x|^2} - \frac{1}{6 \sqrt{\pi}} e^{-|x|^2/4} \ , \ \psi^{(6)}(x) = \frac{16}{15} \psi^{(4)}(x) - \frac{1}{30} \psi^{(4)}(x/2) \ .
\end{align}
These satisfy Assumption \ref{mollifier assumption} with $m =4$, $6$, $L = +\infty$. In two dimensions, we use
\begin{align} \label{2d mollifier}
\psi^{(4)}(x) = \frac{2}{\pi} e^{-|x|^2} - \frac{1}{2 \pi} e^{-|x|^2/2} \ ,
\end{align}
which satisfies  Assumption \ref{mollifier assumption} with $m =4$ and $L = +\infty$.

After selecting a mollifier, one next computes $\grad K_\delta = \grad K * \psi_\delta$ and $\Delta K_\delta = \Delta K * \psi_\delta$. If $\grad K$ and $\Delta K$ are polynomials of degree less than $m$, convolution with $\psi$ preserves the polynomial and $\grad K_\delta = \grad K$, $\Delta K_\delta = \Delta K$. (See Remark \ref{mollifierAccuracy}.) If $K$ is the Newtonian potential $(\Delta)^{-1}$, we have $\Delta K_\delta = \psi_\delta$ and $\grad K_\delta(x) = \frac{x}{|x|^d} \int_0^r s^{d-1} \psi_\delta(s)ds$. 
 
Aside from these special cases, in which an exact expression for the mollified kernel may be found, we compute the convolution numerically using a fast Fourier transform in radial coordinates on a ball of radius 2.5 centered at the origin. Depending on the accuracy we seek, we partition the domain into between 100 and $2 \times 10^6$ grid points and interpolate to obtain $\grad K_\delta(x)$ and $\Delta K_\delta(x)$.
 
Finally, when $K$ is the Newtonian potential, we may also compare our approximate numerical solutions to exact solutions. We are able to compute exact solutions for radial initial data by rewriting the aggregation equation in mass coordinates. This gives the following formula for the particle trajectories in radial coordinates \cite[Section 4]{BertozziGarnettLaurent} and the density along particle trajectories \cite{BertozziLaurentLeger},
\begin{align*}
r(t)^d &= r(0)^d-(d t) m(r(0),0) \ , \quad 
 \rho(X^t(\alpha),t) &= \begin{cases}
 \left( \frac{1}{\rho_0(\alpha)} - t \right)^{-1} &\mbox{ if $\rho_0(\alpha) \neq 0$ ,} \\
0 &\mbox{ if $\rho_0(\alpha) = 0$ .}
\end{cases}
\end{align*}

\subsection{One dimension, $K$ = Newtonian Potential = $(\Delta)^{-1}$}
\begin{figure}[h]
\centering
\begin{overpic}[trim={1.1cm .7cm .5cm 1cm},clip,width=.4\textwidth]{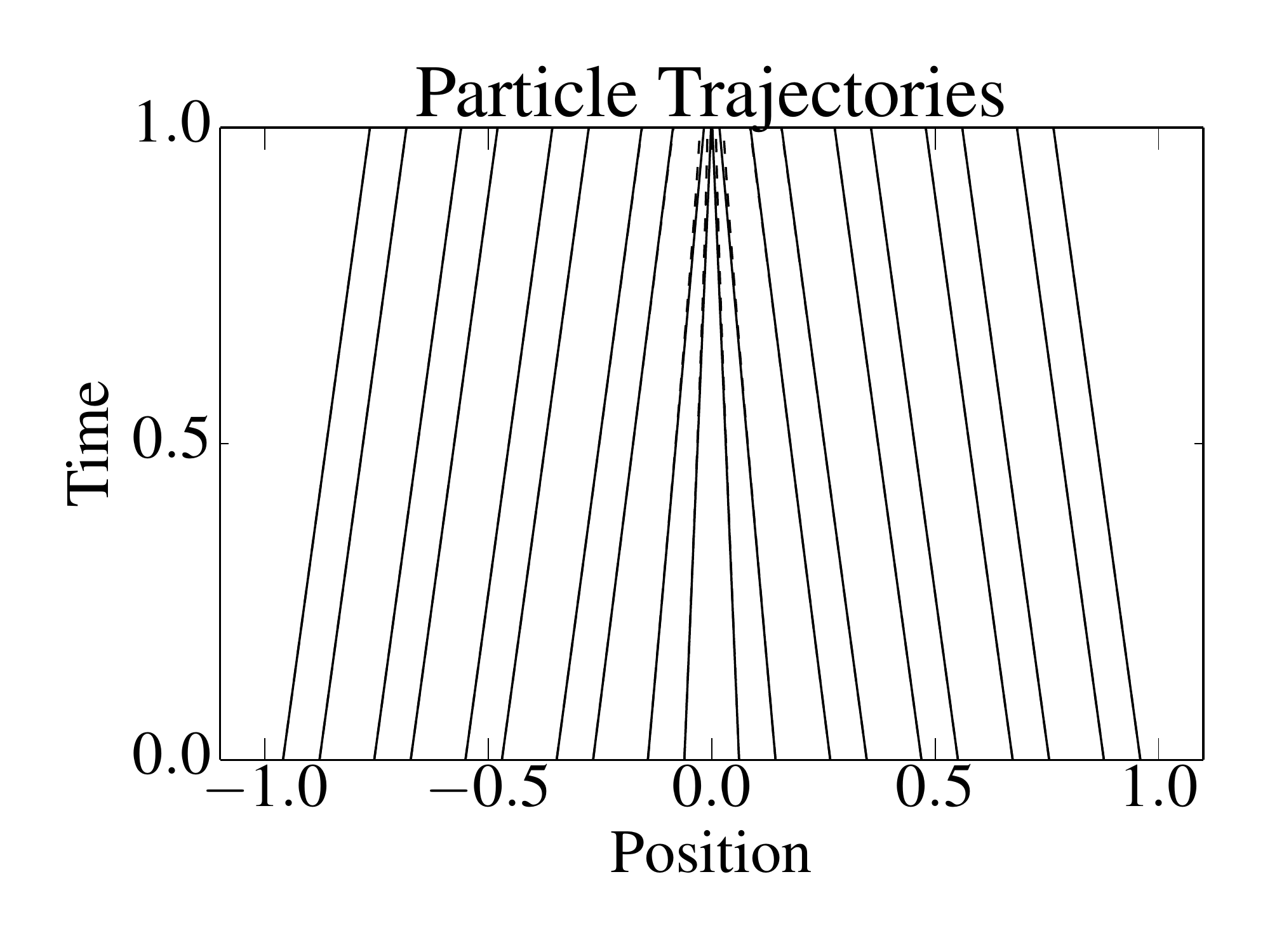}
\put(22.5,67.5){\small A.}
\end{overpic}
\begin{overpic}[trim={1.1cm .7cm .5cm 1cm},clip,width=.4\textwidth]{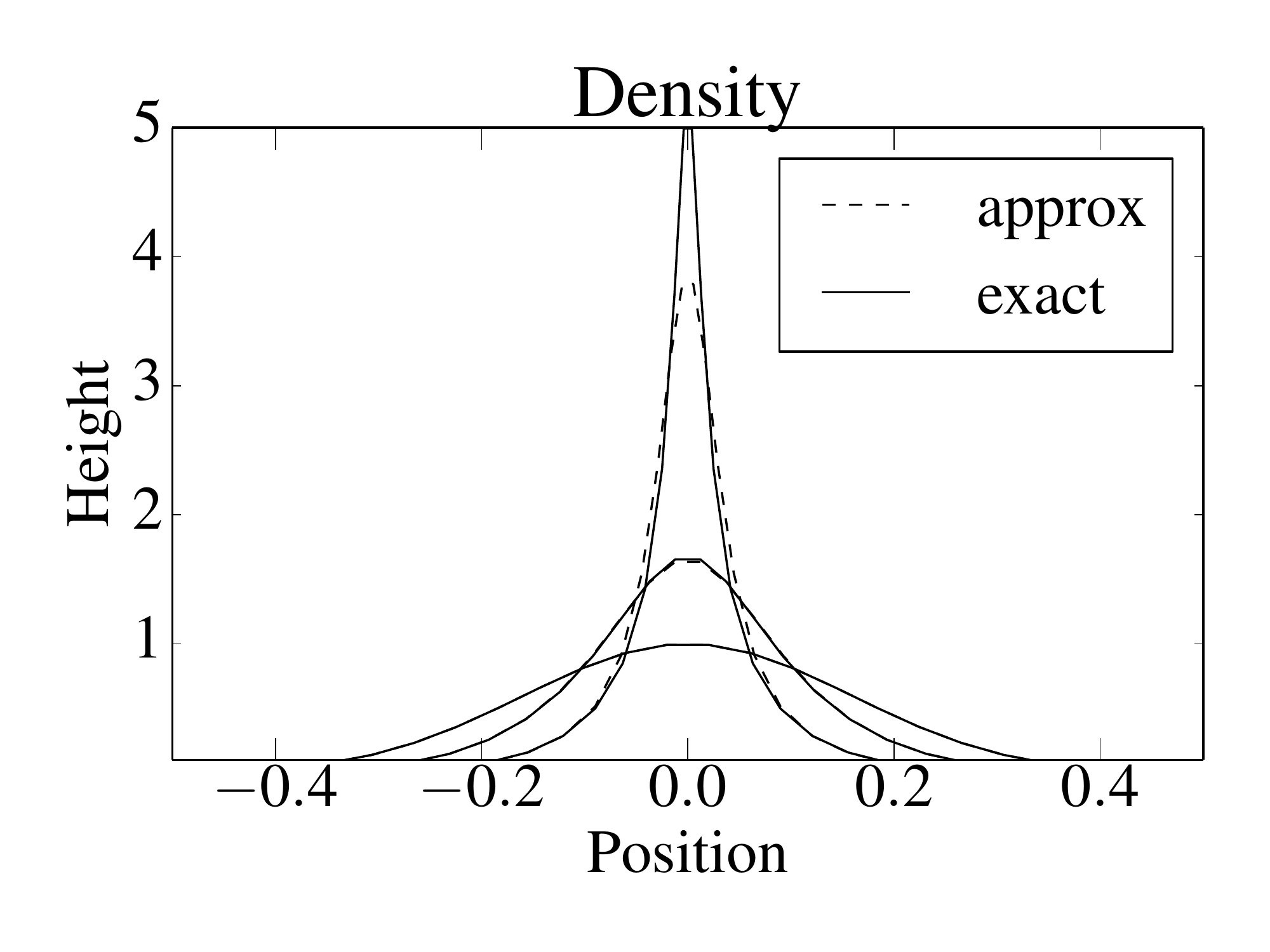}
\put(36,67.5){\small B.}
\end{overpic}
\\
\begin{overpic}[trim={.7cm 1.2cm .8cm .85cm},clip,width=.4\textwidth]{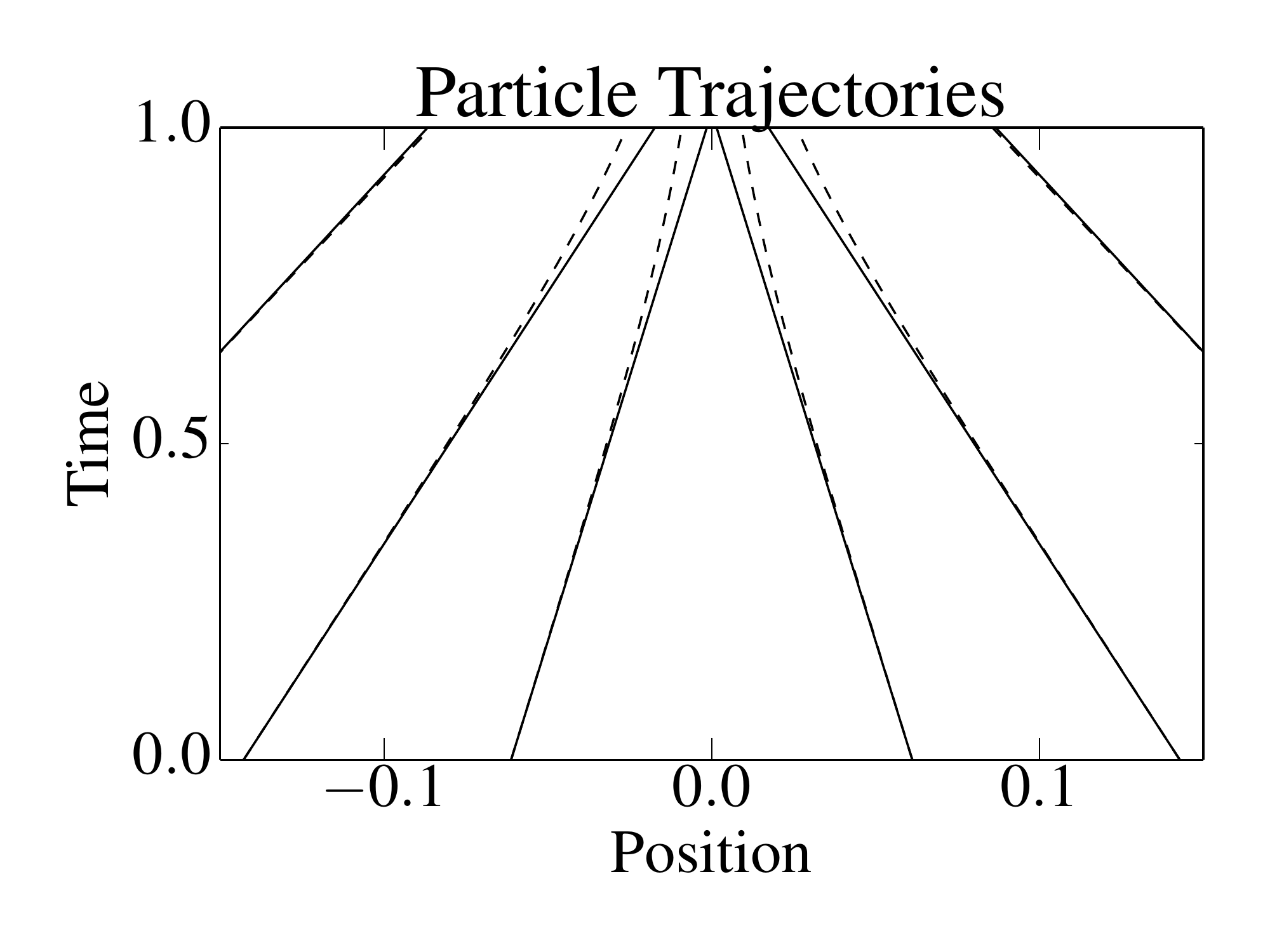}
\put(25,64.5){\small C.}
\end{overpic}
\begin{overpic}[trim={.5cm 1.2cm 1cm .85cm},clip,width=.415\textwidth]{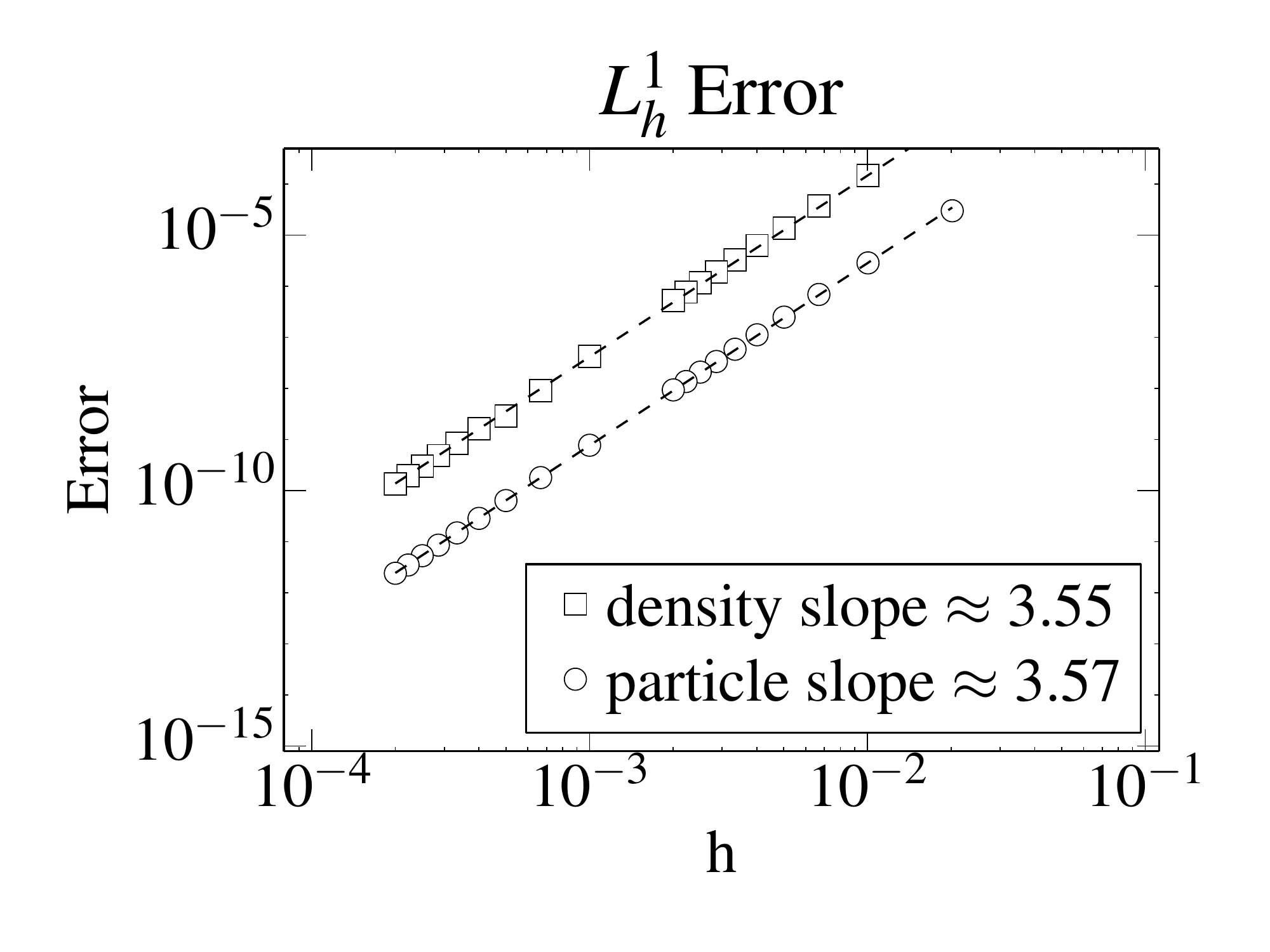} 
\put(41,64.7){\small D.}
\end{overpic}
\caption{A comparison of exact solutions (solid lines) with blob method solutions (dashed lines). The densities (B) are shown at times $t=0, 0.4,0.8$. The log-log plot  (D) corresponds to $t=0.5$.}
\label{1dsmooth}
\end{figure}

\quad \\ \noindent
\textbf{Blob method, regular initial data: }Figure \ref{1dsmooth} compares exact and blob method solutions to the one dimensional aggregation equation when $K$ is the Newtonian potential and the initial initial data is
\begin{align} \label{pfn20rho0}
\rho_0(x) = \begin{cases} (1-x^2)^{20} &\mbox{if } |x|\leq 1 \ , \\ 
0  & \mbox{otherwise} \ . \end{cases}
\end{align}
For this initial data, finite time blowup for the classical solution occurs at ${t=1}$.

We discretize the domain $[-1,1]$ using $h = 0.04$. With this refinement of the grid,  the approximate and exact particle trajectories are visually indistinguishable (A), though the approximate density loses resolution at $t= 0.8$ (B).
Focusing on a smaller spatial scale  (C) reveals that the approximate particle trajectories bend away from the exact solution to avoid collision at $t=1$. This is due to the regularization of the kernel, which causes the velocity field to be globally Lipschitz. Bhat and Fetecau observed the same bending effect for an analogous regularization in their work on Burgers equation \cite{BhatFetecau2}.

In spite of the bending, blob method solutions converge to exact solutions with a high order rate of convergence for $t<1$. We choose the $m=4$ mollifier $\psi^{(4)}$ (\ref{both 1d mollifiers}) and $\delta = h^q$ for $q = 0.9$. A log-log plot of the $L^1_h$ error of the particle trajectories and density (D) reveals a numerical rate of convergence close to the theoretically predicted rate of $mq = 3.6$. Note that the numerical result is slightly stronger than the theoretical result, which measures the error of the approximate density in $W^{-1,1}_h$ and requires the exact solution to be smooth.

\begin{figure}[h]
\centering
\includegraphics[trim={.7cm 1.2cm .8cm 1cm},clip,width=.4\textwidth]{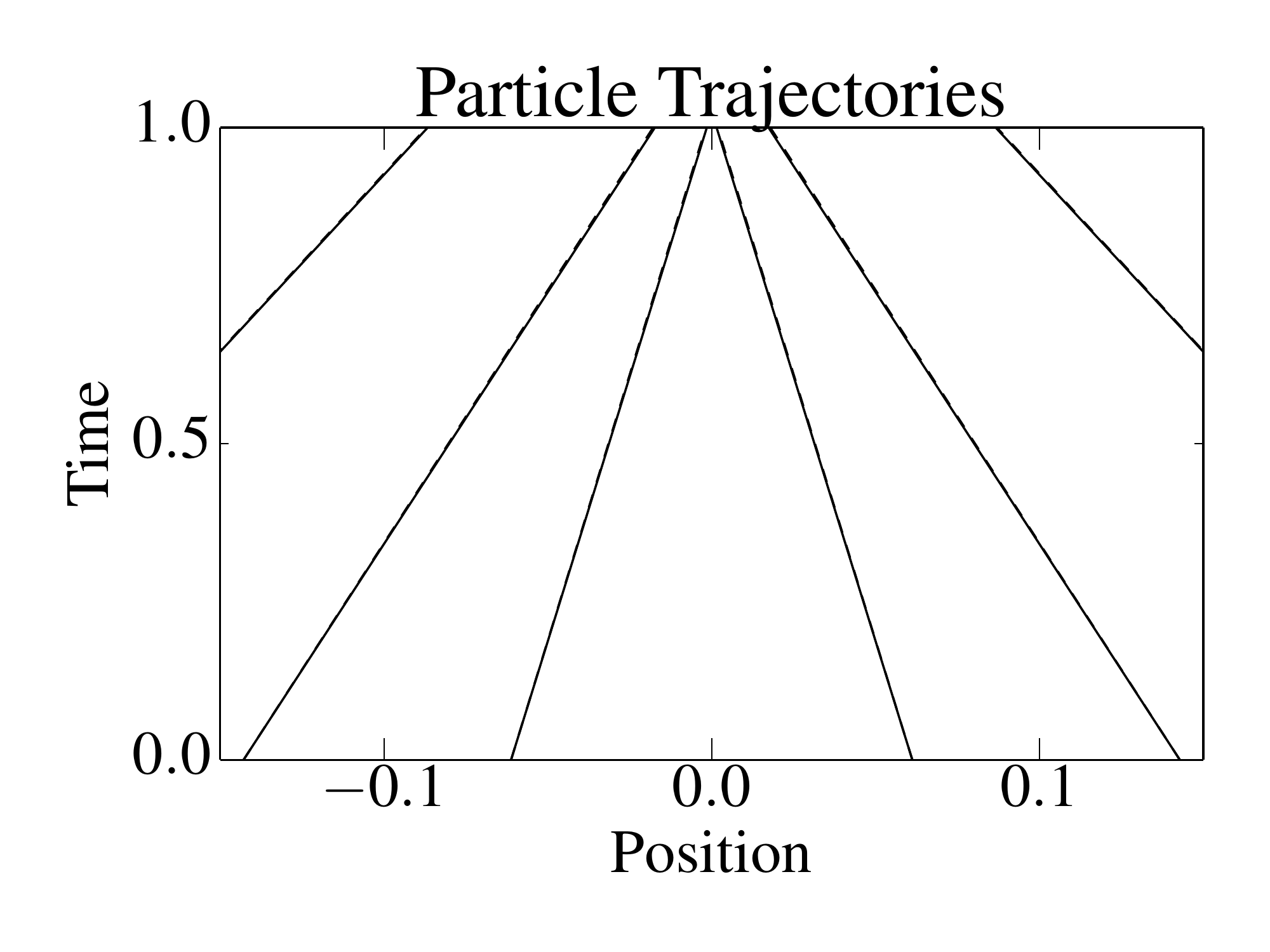}
\includegraphics[trim={.5cm 1.2cm 1cm .9cm},clip,width=.415\textwidth]{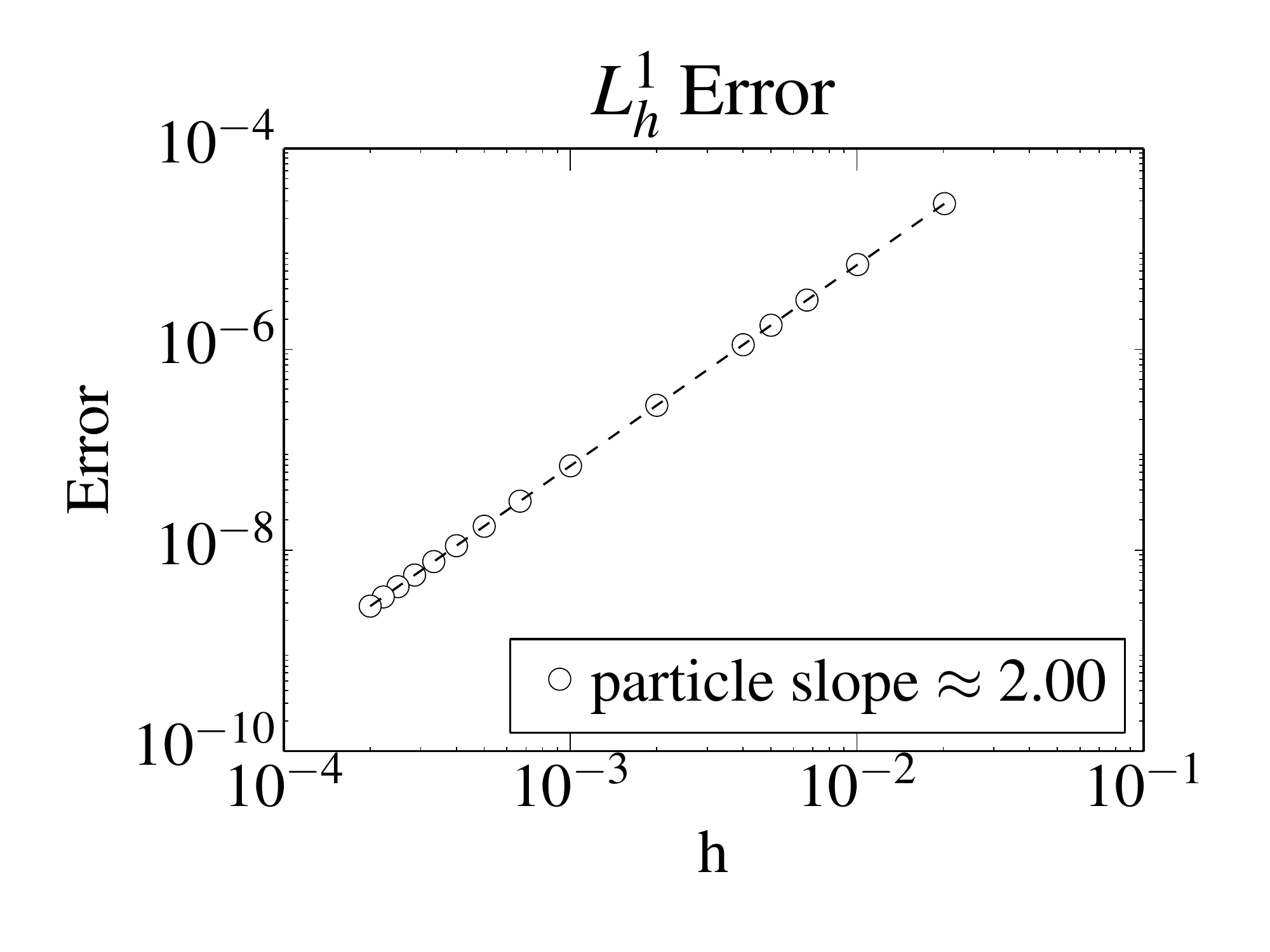}
\caption{A comparison of exact solutions (solid lines) and particle method solutions (dashed lines). The log-log plot of the error (right) corresponds to $t=0.5$.}  \label{monkey}
\end{figure}

\quad \\ \noindent
\textbf{Particle method, regular initial data: }
Figure \ref{monkey} 
compares a particle method approximation with an exact solution for initial data (\ref{pfn20rho0}). To compute the particle method approximation, we remove the singularity at zero by setting $\grad K(0)= 0$, instead of regularizing the kernel. Since $K$ is the Newtonian potential, the particle method densities are point masses, which do not belong to $L^p_h$. Thus, we only consider the particle trajectories defined by this method.

Unlike the blob method solution from in the previous example, the particle method trajectories do not bend away at blowup time. However, a log-log plot of the $L^1_h$ error reveals a slower rate of convergence than for the blob method, consistent with the rate of $\mathcal{O}(h^{2-\epsilon})$ for particle approximations of the Euler equations \cite{GoodmanHouLowengrub, HouLowengrub}.

\begin{figure}[h]
\begin{centering}
\begin{overpic}[trim={1.1cm 1cm .5cm 1cm},clip,width=.4\textwidth]{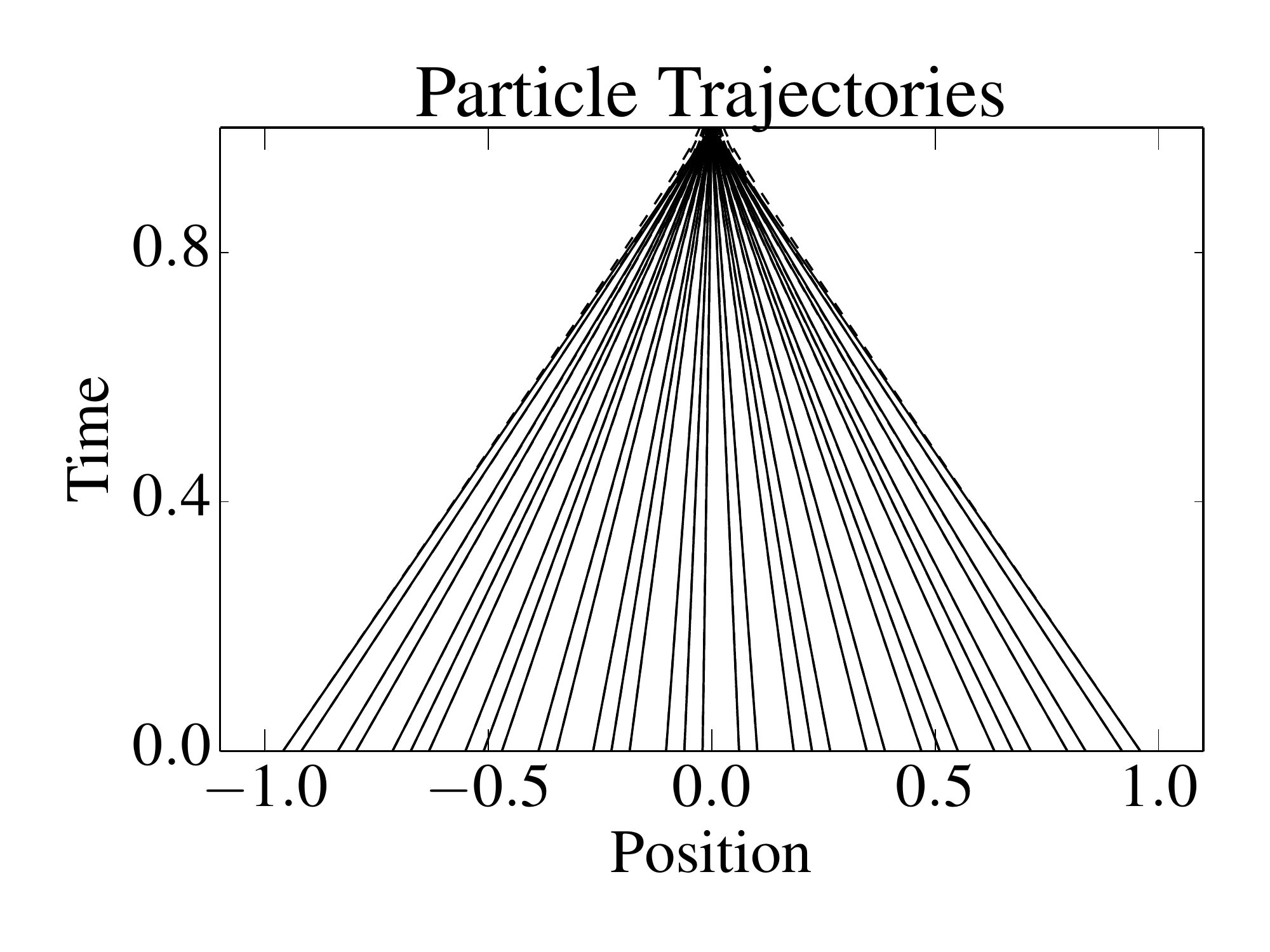}
\put(22.5,66.2){\small A.}
\end{overpic}
\begin{overpic}[trim={1.1cm 1cm .5cm 1cm},clip,width=.4\textwidth]{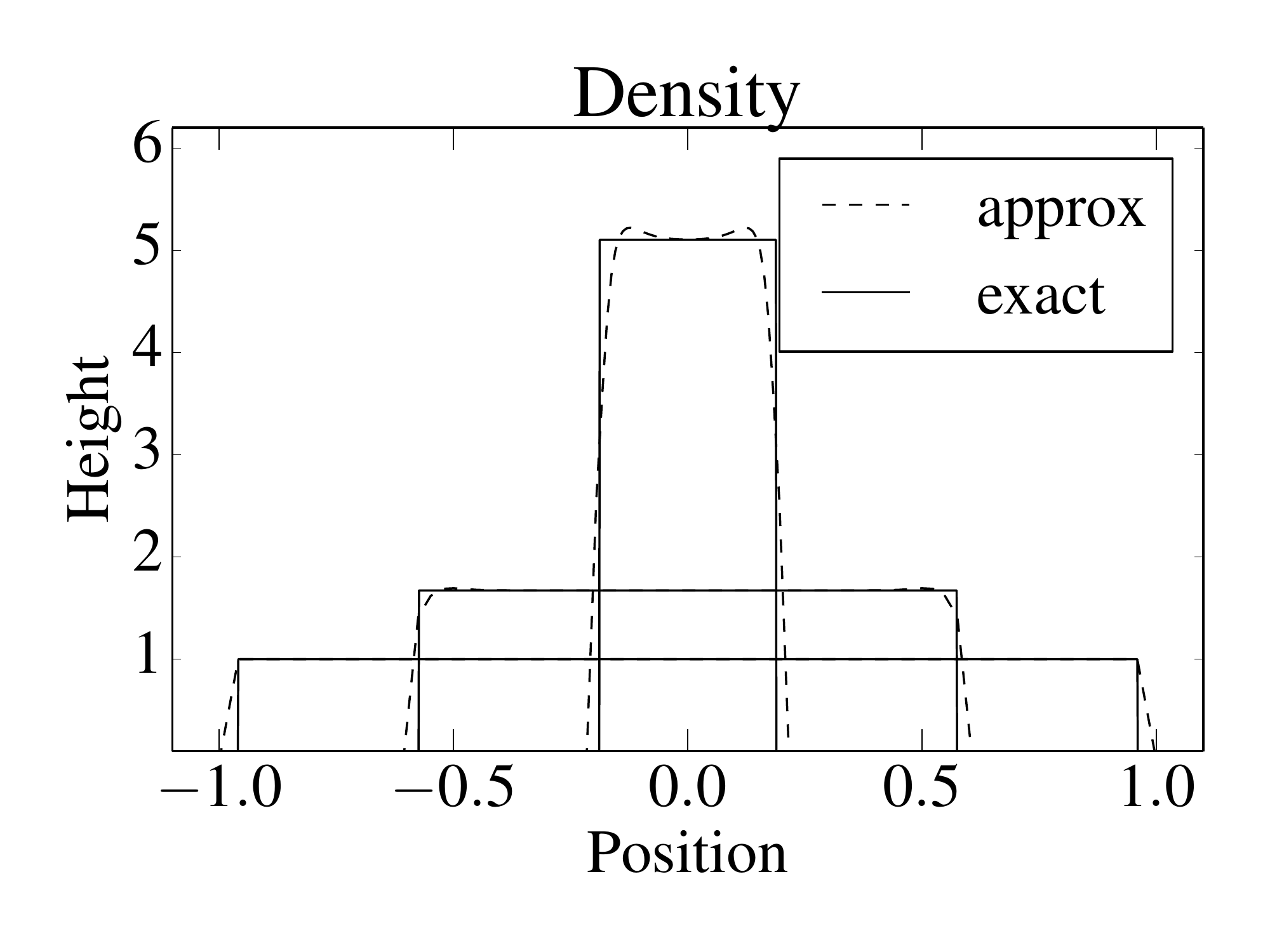}
\put(36,66.2){\small B.}
\end{overpic}
\\
\begin{overpic}[trim={1.1cm 1.1cm .8cm 1cm},clip,width=.4\textwidth]{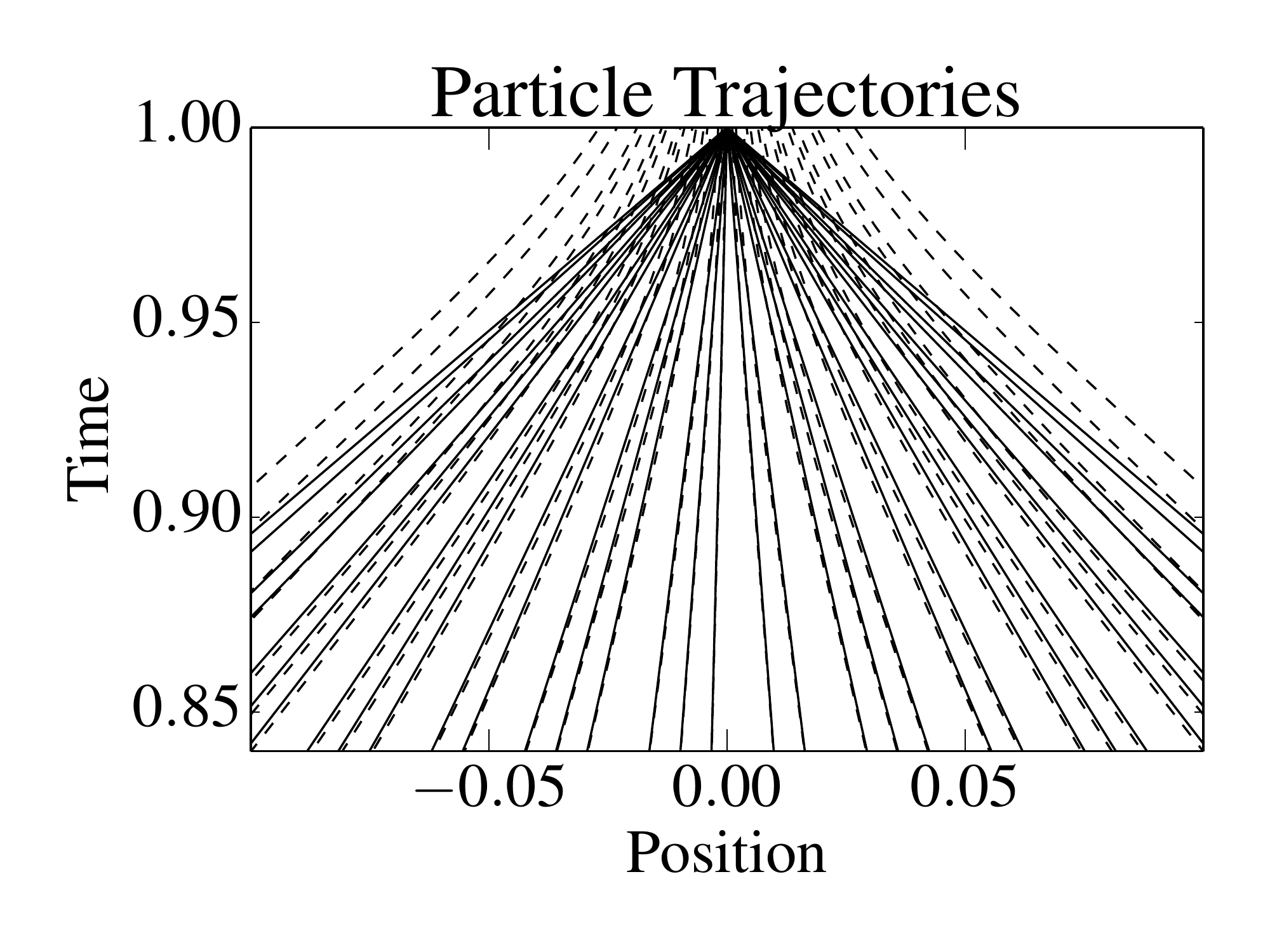}
\put(24.5,66.5){\small C.}
\end{overpic}
\begin{overpic}[trim={1cm 1.08cm .3cm .7cm},clip,width=.42\textwidth]{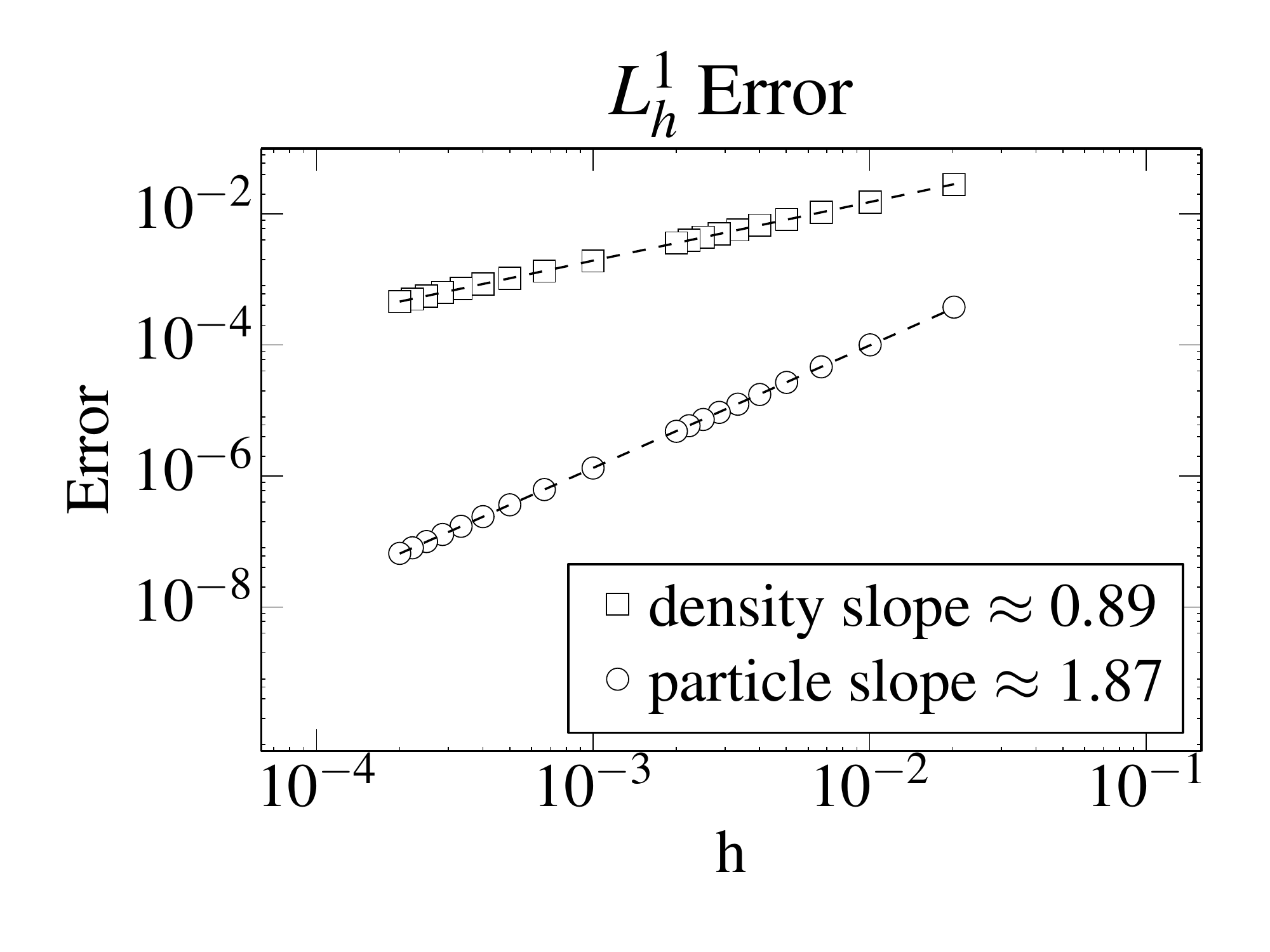}
\put(38,65){\small D.}
\end{overpic}

\end{centering}
\caption{A comparison of exact solutions (solid lines) and blob method solutions (dashed lines) with $\delta = h^q$ for $q=0.9$. The densities (B) are shown at  $t=0, 0.4,0.8$. The log-log plot (D) corresponds to $t=0.5$.}
\label{1dpatch}
\end{figure}
\quad \\ \noindent
\textbf{Discontinuous initial data:}  
Figure \ref{1dpatch} compares exact and blob method solutions to the one dimensional aggregation equation with discontinuous initial data\begin{align} \label{patchfnrho0}
\rho_0(x) = 1_{[-1,1]}(x) = \begin{cases} 1 &\mbox{if } |x|\leq 1 \ , \\ 
0  & \mbox{otherwise} \ . \end{cases}
\end{align} 
Though our convergence results only apply to sufficiently regular solutions to the aggregation equation, the definition of the blob method (Definition \ref{blob method def}) merely requires that initial data to be a compactly supported function.

We discretize the domain $[-1,1]$ using $h = 0.04$. At this level of resolution,  the approximate and exact particle trajectories are visually indistinguishable (A), though the approximate density becomes rounded at $t= 0.8$ (B), with a similar oscillating profile as the mollifier (\ref{both 1d mollifiers}).

Considering the particle trajectories on a smaller spatial scale (C), we again observe the trajectories bending away to avoid collision. As expected, a log-log plot of the $L^1_h$ error (D) reveals a slower rate of convergence than in Figure \ref{1dsmooth}. Unlike in the previous example, for which the slower rate of convergence was due to using a particle method instead of the blob method, in this example the slower rate of convergence is due to the lower regularity of the initial data.

\subsection{One dimension, Various potentials}

\begin{table}[h] 
 \centering
 \begin{tabular}{c|c}
&\hspace{.8cm} particle \hspace{2.4cm} $m=4$  \hspace{2.4cm} $m=6$ \\
 \hline
\rotatebox{90}{\makebox(0,0)[lc]{~\parbox{3.2cm}{
  $K(x) = -\log |x|/2\pi$}}} & \includegraphics[trim={.9cm 1.3cm 1cm .5cm},clip,width=.35\textwidth]{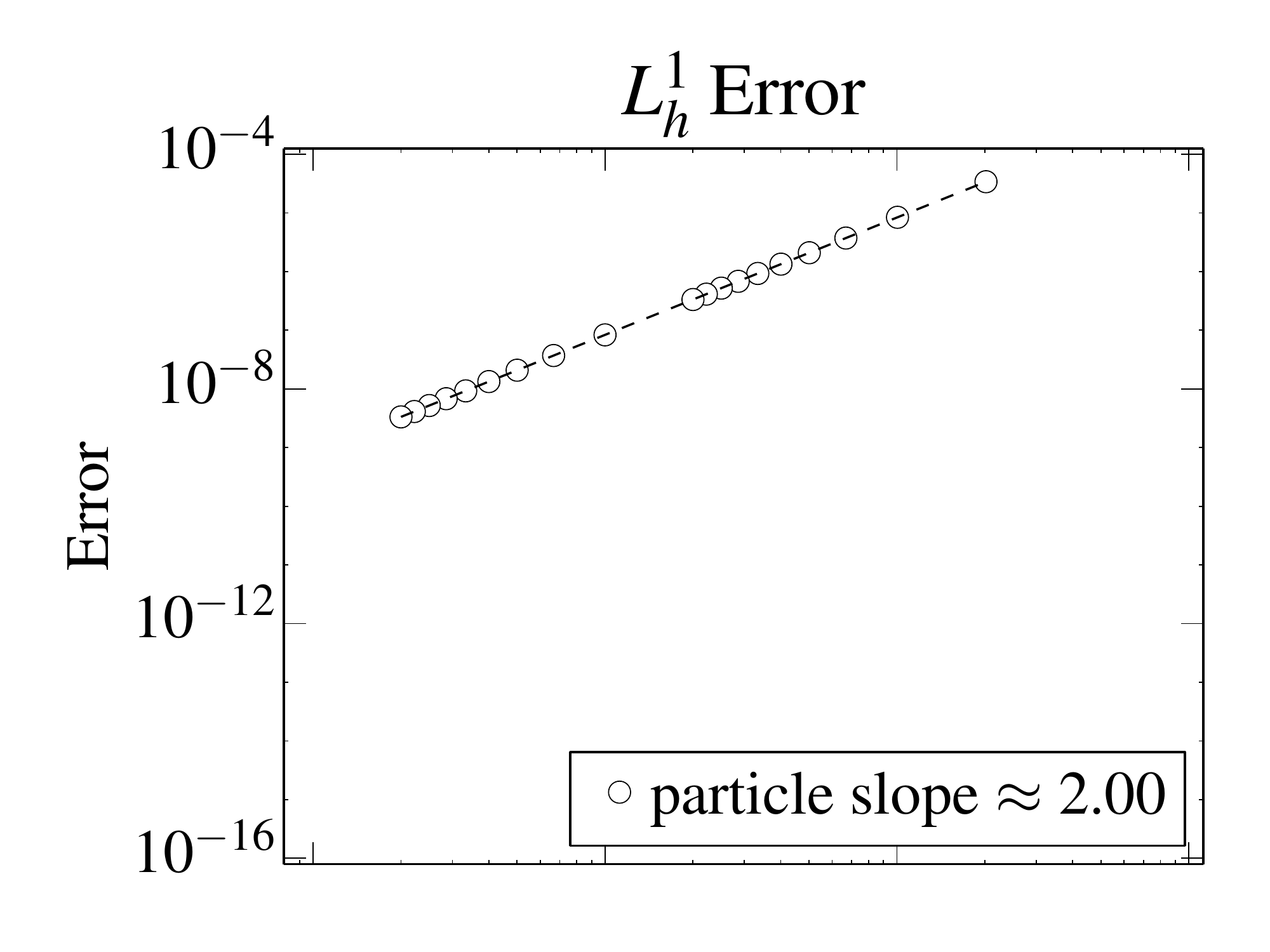} \hspace{-2.7mm} \includegraphics[trim={4.6cm .8cm 1cm .7cm},clip,width=.283\textwidth]{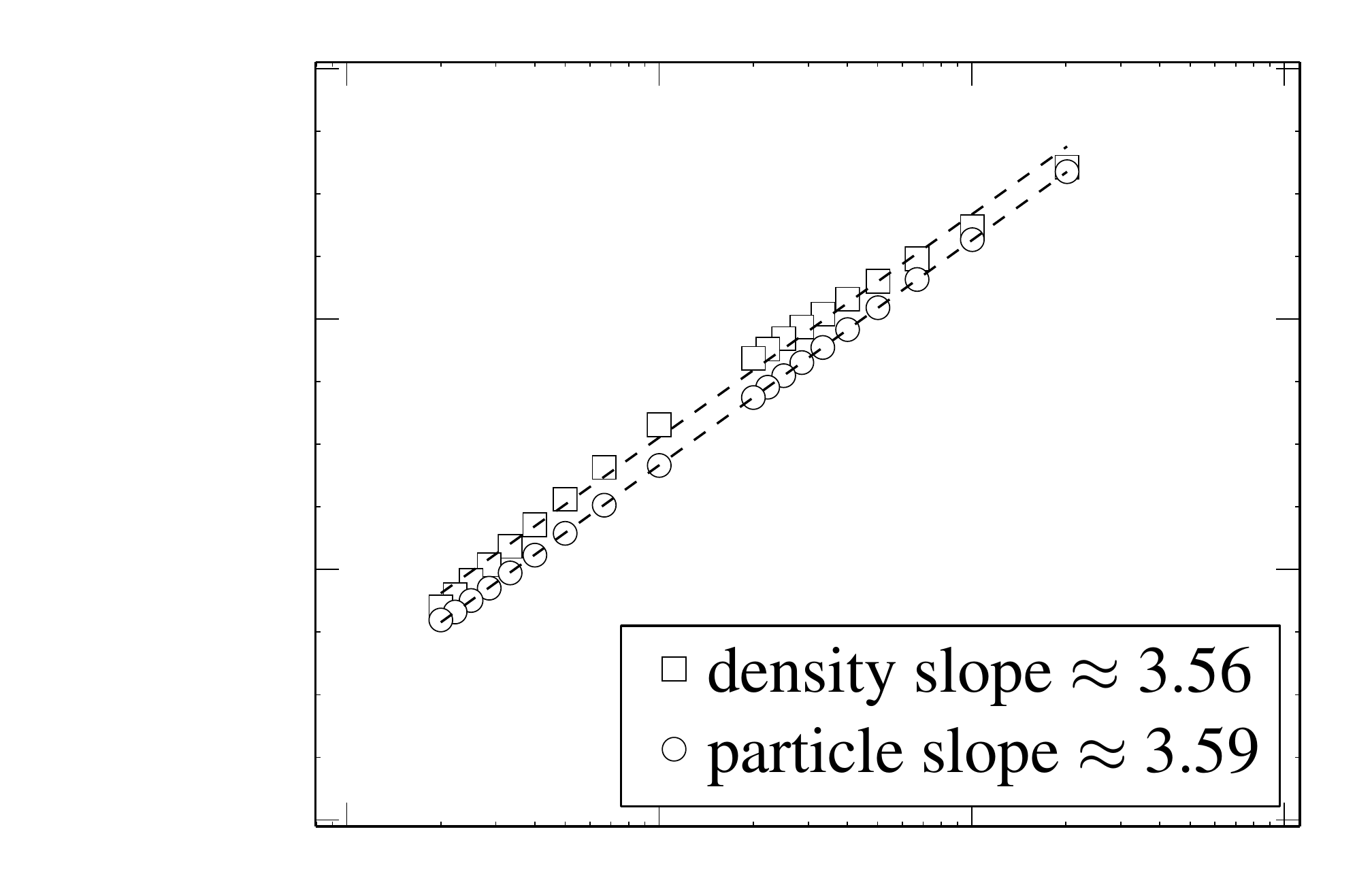} 
 \hspace{-2.8mm}  \includegraphics[trim={4.15cm .8cm 1cm .7cm},clip,width=.282\textwidth]{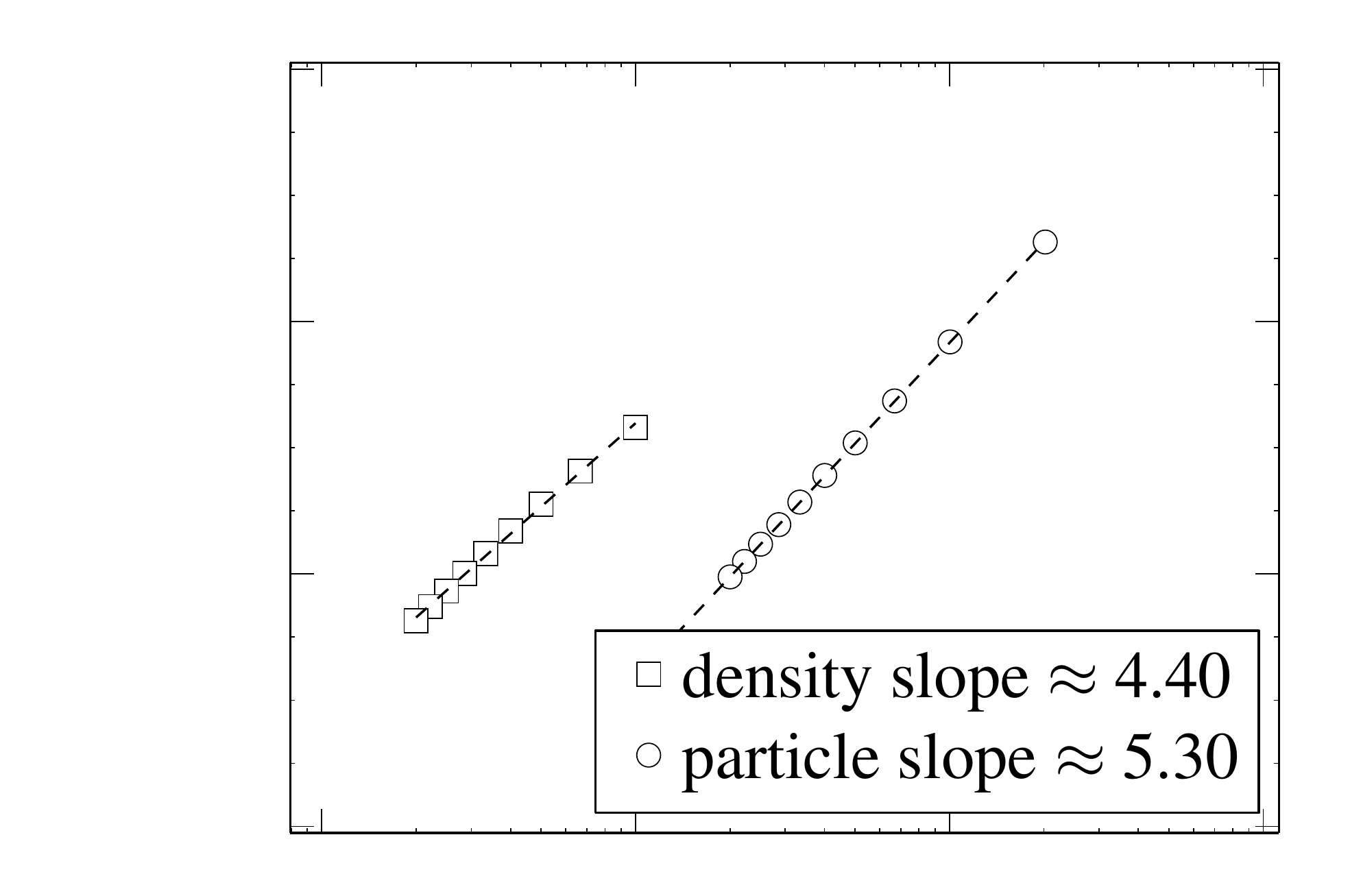}  \\
\hline
\parbox[t]{3mm}{\multirow{1}{*}{\rotatebox[origin=c]{90}{$K(x) =|x|^3/3$ \ }}}& \vspace{-3mm} \\ 
& \includegraphics[trim={-2.1cm 1.2cm 1cm .8cm},clip,width=.355\textwidth]{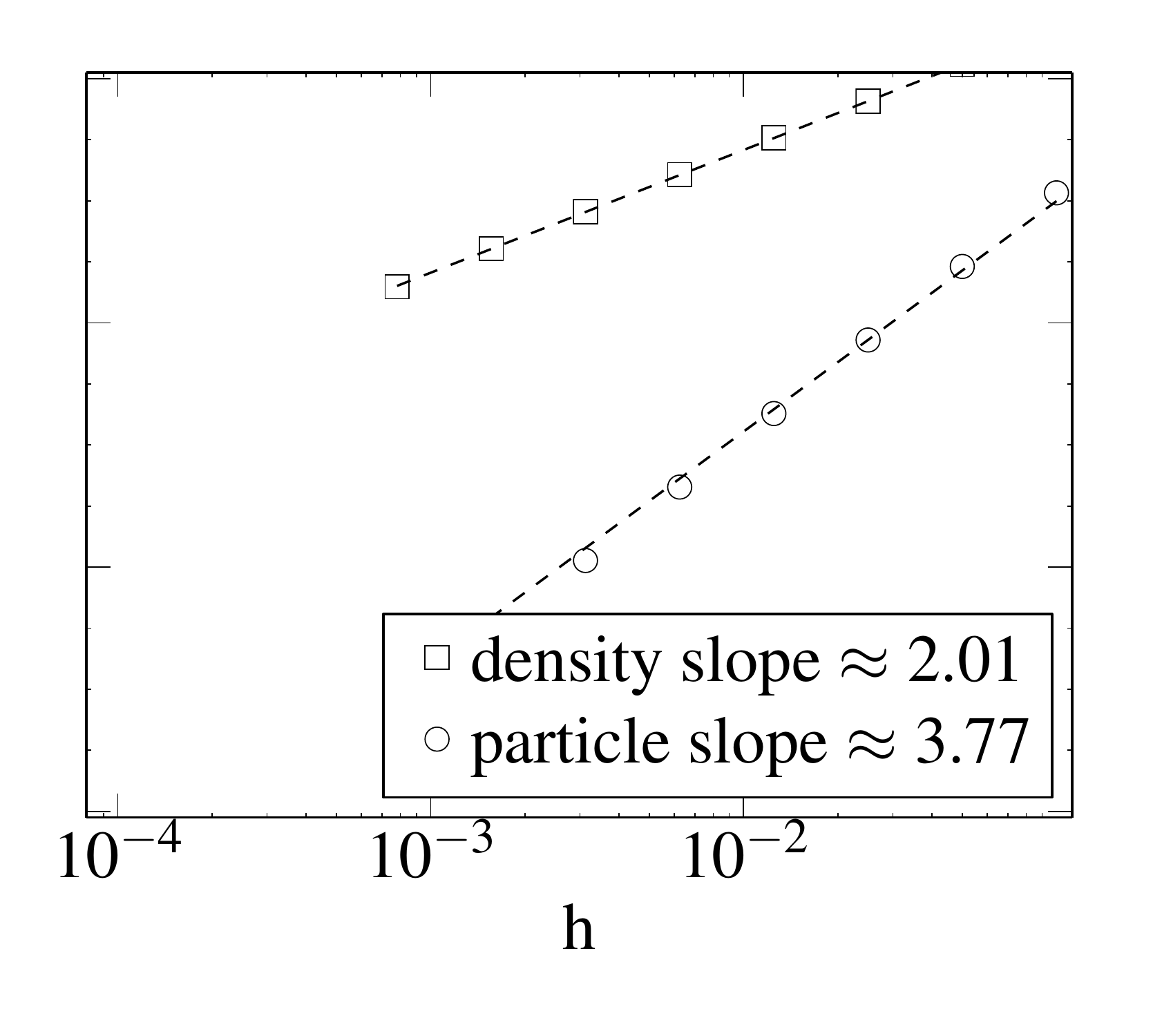} \hspace{-2.9mm} \includegraphics[trim={4.6cm -1.15cm 1cm .8cm},clip,width=.277\textwidth]{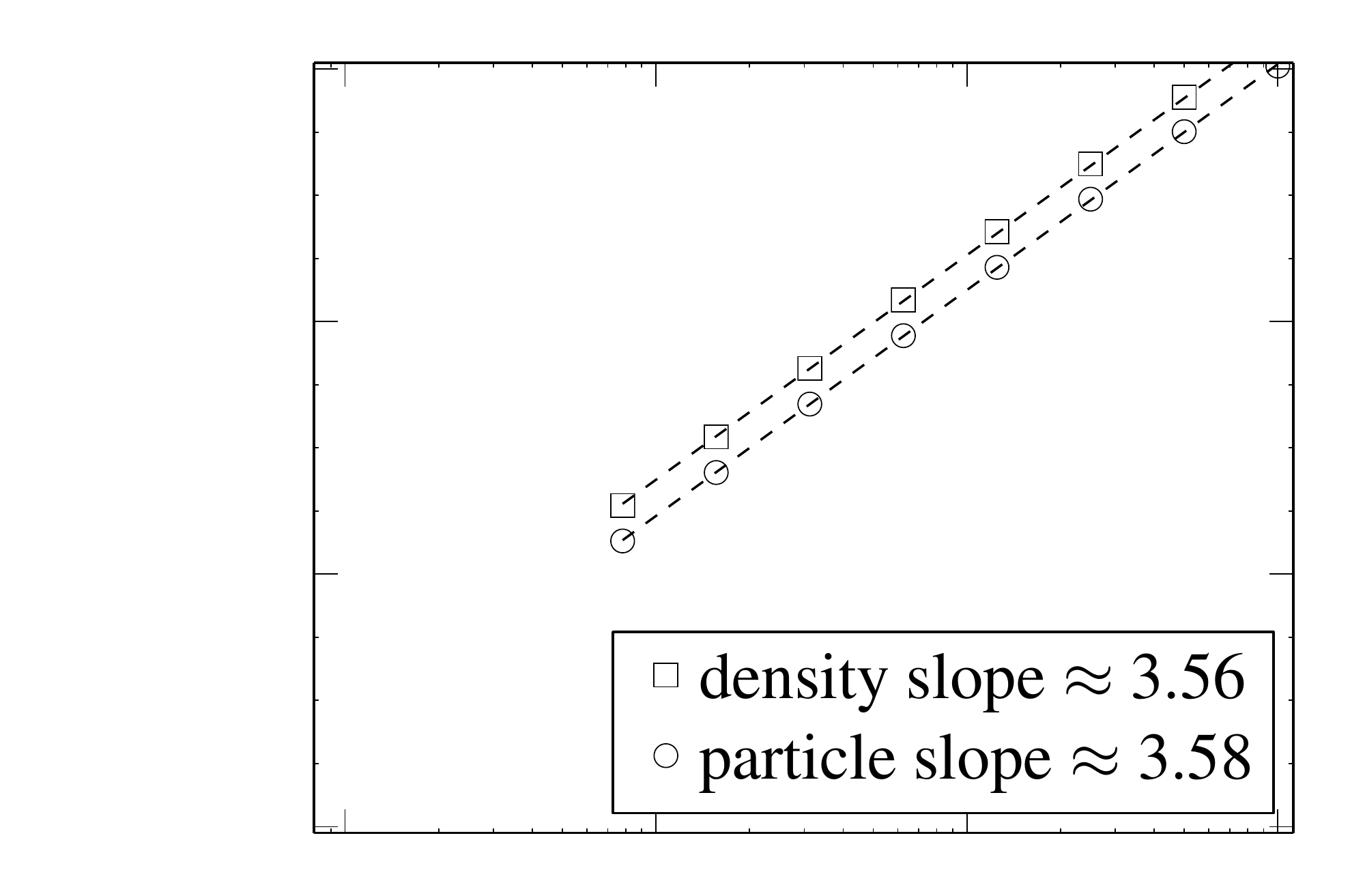} 
 \hspace{-2.7mm}  \includegraphics[trim={4.15cm -1.15cm 1cm .8cm},clip,width=.281\textwidth]{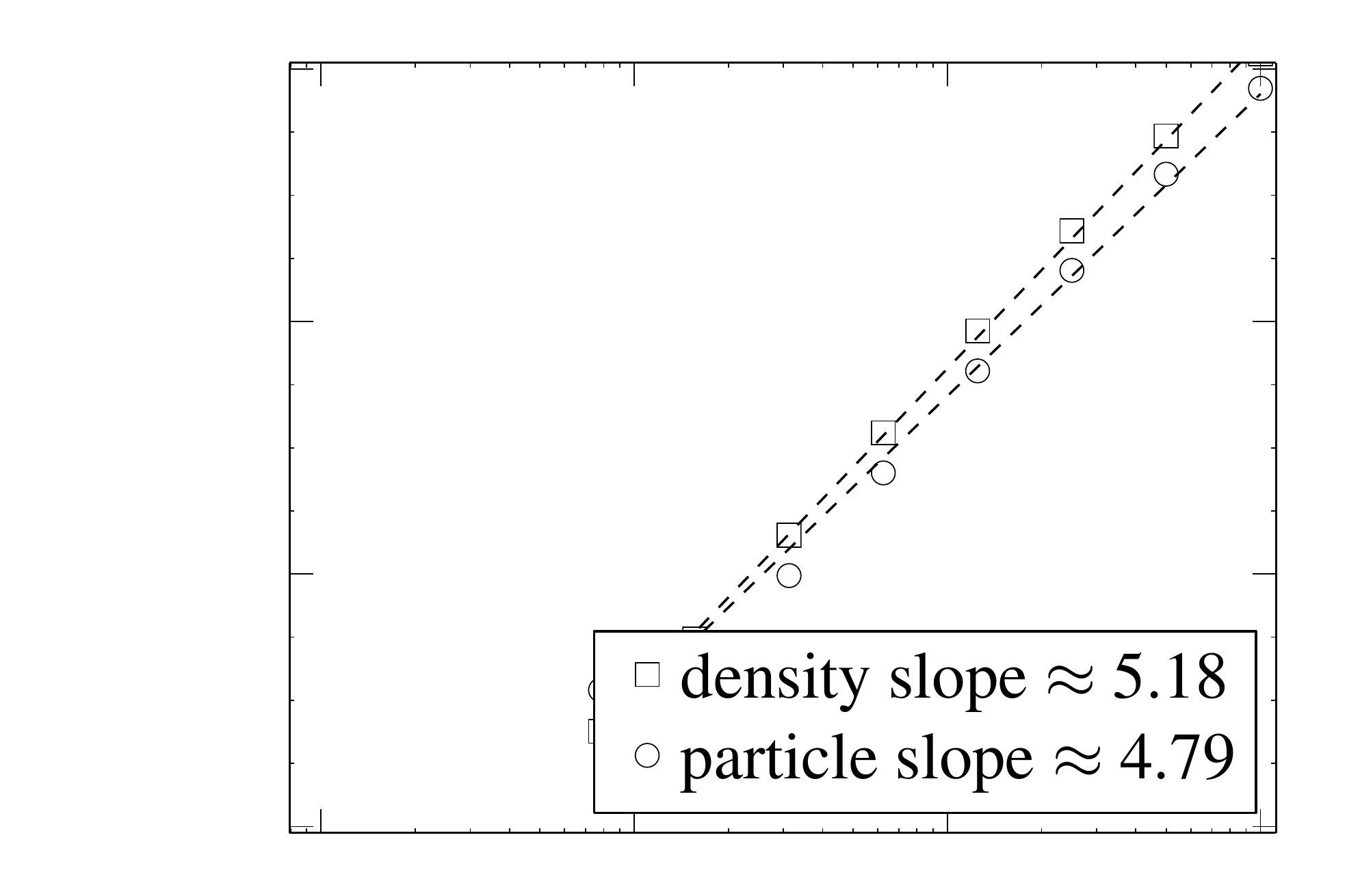}
\end{tabular} 
\captionof{figure}{A comparison of the $L^1_h$ error at $t=0.5$ for various kernels and numerical methods. For $K = (-\Delta)^{-1}$, we use analytic expressions for $\grad K *\psi_\delta$ and $\Delta K * \psi_\delta$. For the cubic potential, we compute the convolutions using a fast Fourier transform in radial coordinates on a ball of radius 2.5 with $5 \times 10^5$ grid points (for $m=4$) and $2 \times 10^6$ grid points (for $m=6$).}
\label{variouspotentialsloglog}
\end{table}

Figure \ref{variouspotentialsloglog} compares the rate of convergence of numerical solutions when $K = (-\Delta)^{-1}$ or $K(x) = |x|^3/3$. The mollifiers (\ref{both 1d mollifiers}) are of order $m=4$ and $m =6$. The scaling of the regularization is $\delta = h^q$ for $q = 0.9$, and the initial data is
\begin{align} \label{pfn10rho0}
\rho_0(x) = \begin{cases} (1-x^2)^{10} &\mbox{if } |x|\leq 1 \ , \\ 
0  & \mbox{otherwise} \ . \end{cases}
\end{align}

The more singular the kernel, the greater improvement we see in using the blob method over a particle method.
For the negative Newtonian potential, the $m=4$ blob method improves upon the particle method, and the $m=6$ blob method shows even greater improvement. For the cubic potential, the particle method is better than the $m=4$ blob method for the  trajectories, but not for the density. The $m=6$ blob method is best for both trajectories and density. 

For both kernels, the rates of convergence for the $m=4$ blob method are very close to the theoretically predicted rate of $mq = 3.6$, while the rates of convergence for the $m=6$ blob method are not as good as the theoretically predicted rate of $mq = 5.4$, due to other sources of error in the numerical implementation.

\subsection{Two Dimensions}
 \begin{table}[h] 
 \centering
 \begin{tabular}{cc|c}
 $\rho_0(x) = (1-x^2)_+^{2}$&  & $\rho_0(x) = 1_{[-1,1]}$ \\
\vspace{-3mm}& & \\
 $m=4$ & particle& $m=4$  \\
 \hline
\hspace{-5mm}
\includegraphics[trim={.5cm .7cm 1cm .5cm},clip,width=.35\textwidth]{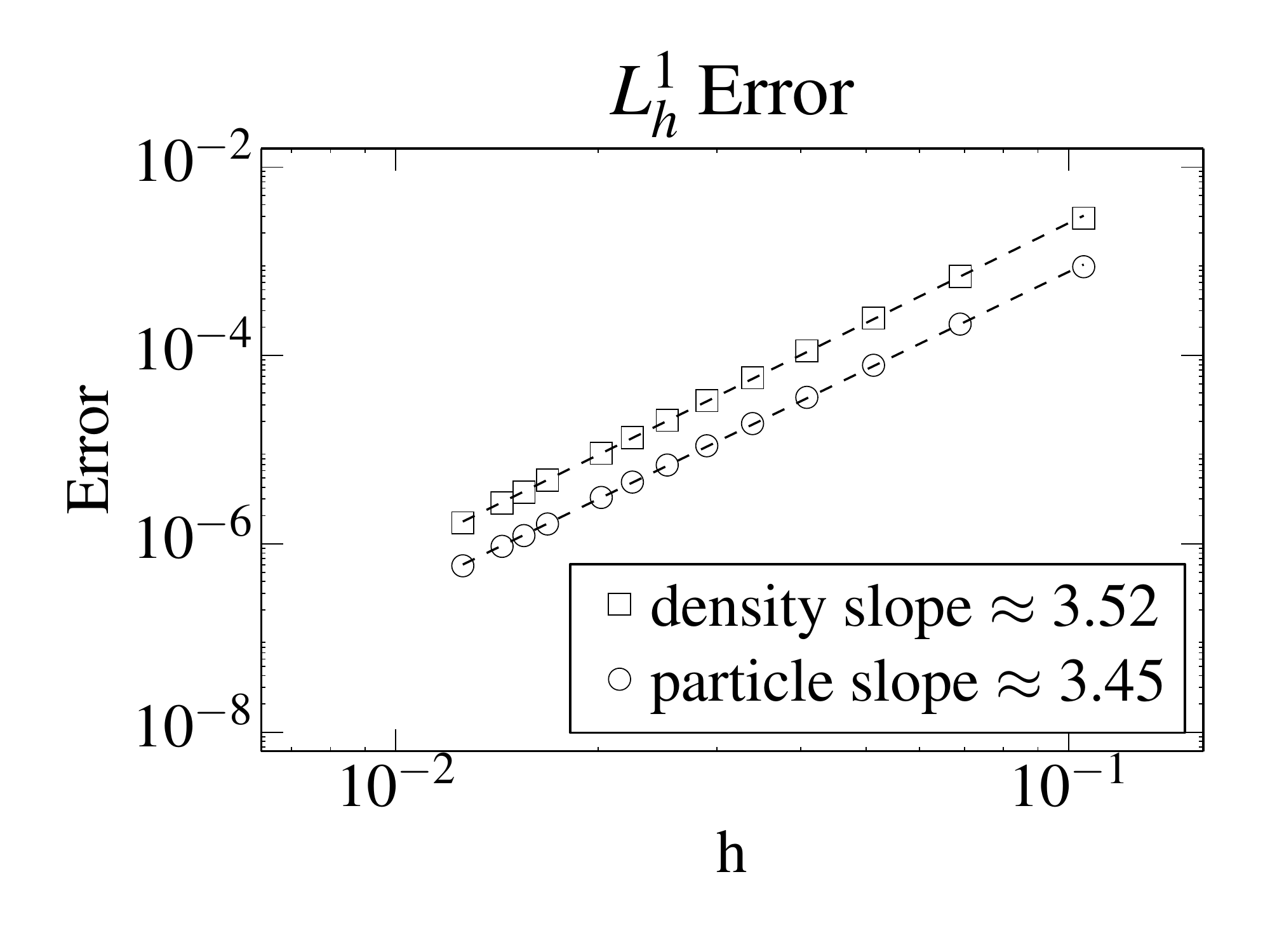} & \hspace{-4mm}
\includegraphics[trim={1cm -1.5cm .9cm .5cm},clip,width=.31\textwidth]{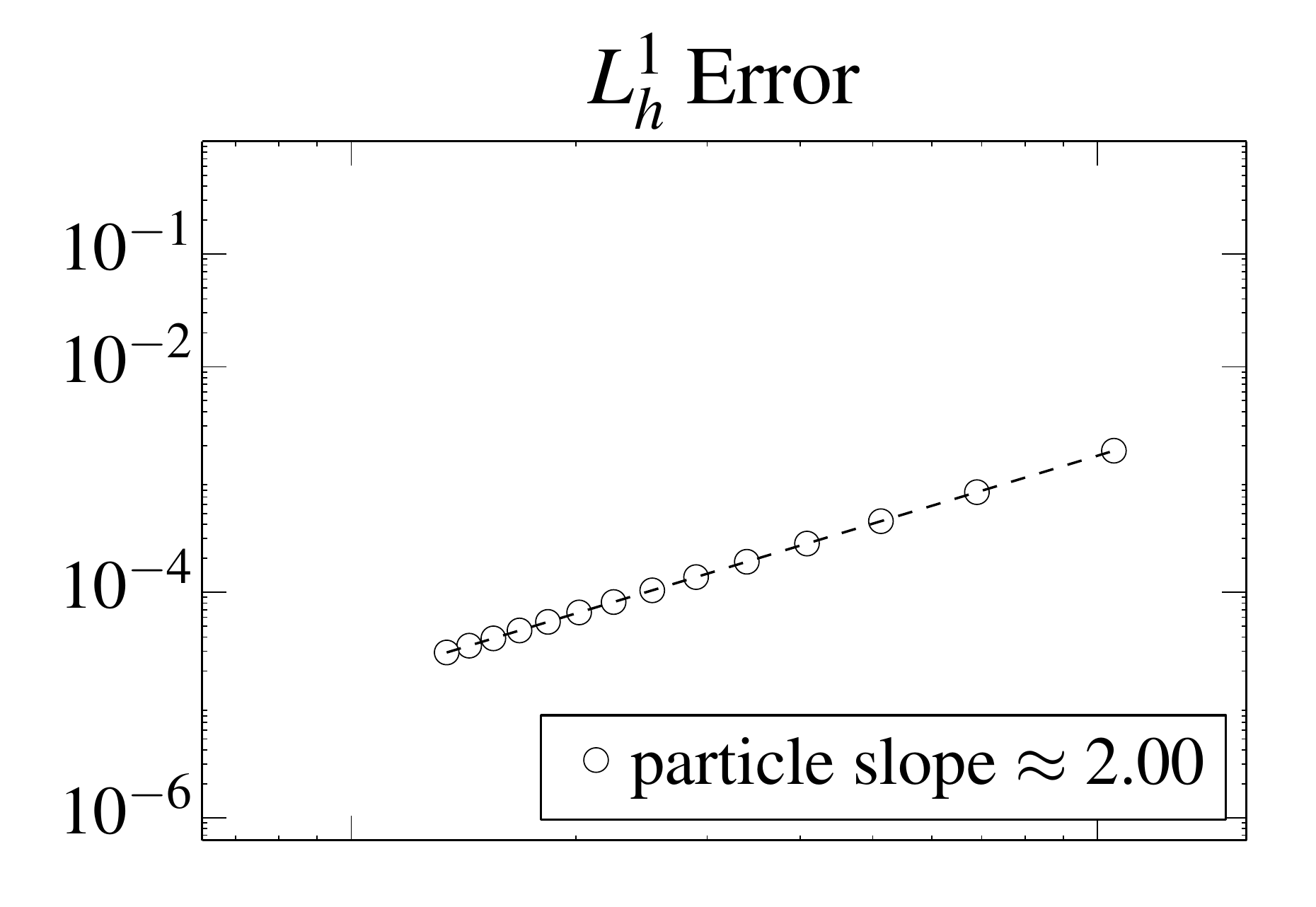} 
& \hspace{-2mm} \includegraphics[trim={2.1cm -1.8cm .9cm .7cm},clip,width=.303\textwidth]{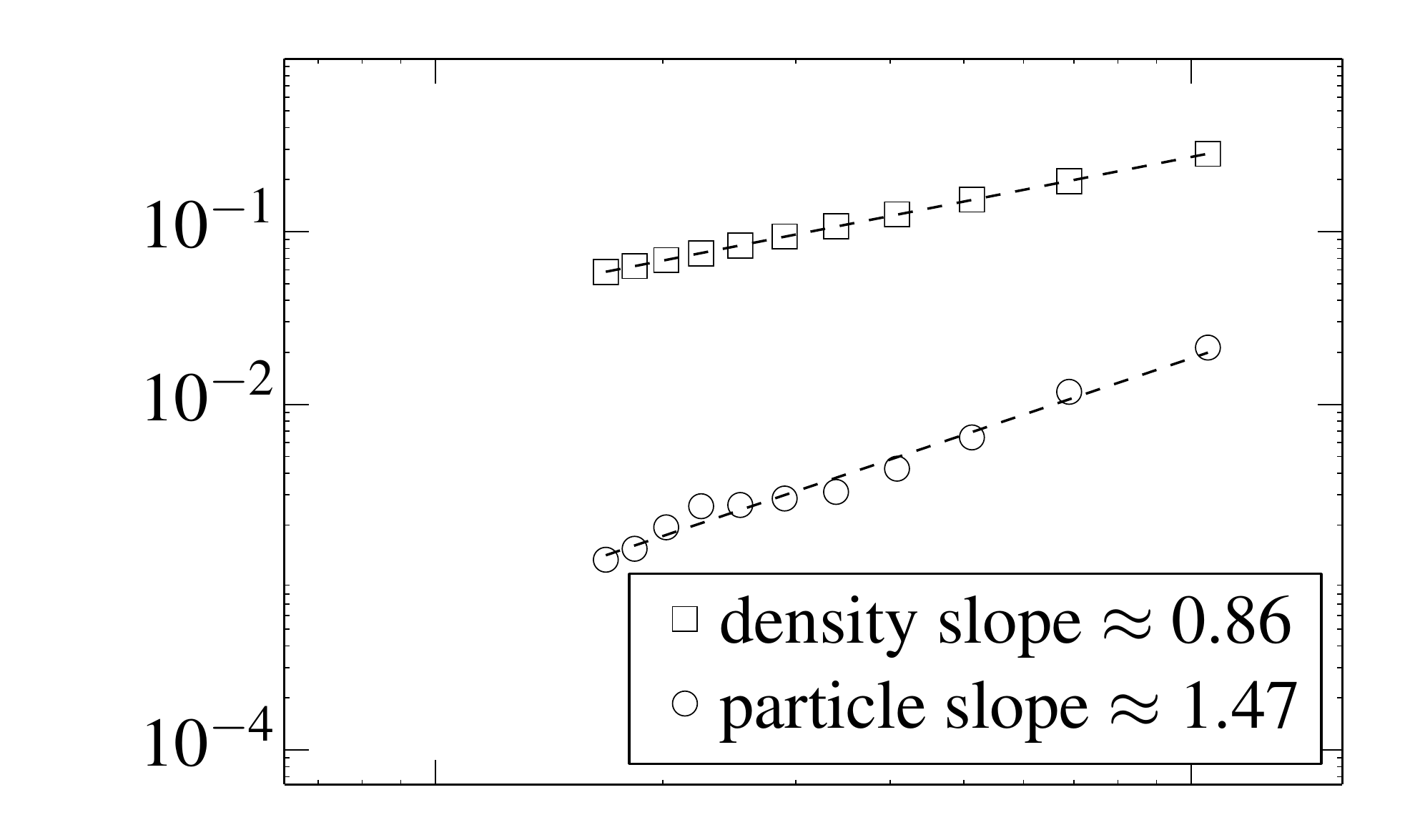} 
\end{tabular}
\captionof{figure}{A comparison of $L^1_h$ error at $t = 0.5$ for the two dimensional aggregation equation when $K = (\Delta)^{-1}$.} \label{2dNewtloglog}
\end{table}

\quad \\ \noindent
\textbf{Newtonian potential:}
Figure \ref{2dNewtloglog} compares the rate of convergence of numerical solutions when $K = (\Delta)^{-1}$ for various choices of initial data. In the first and third plot, the numerical solution is computed via the blob method with a mollifier (\ref{2d mollifier}) of order $m =4$ and $\delta = h^q$ for $q = 0.9$.  In the second plot, the numerical solution is computed via a particle method.

As in the one dimensional case, we see the best rates of convergence for the blob method applied to regular initial data (left). This agrees with our theoretically predicted rate of $mq = 3.6$. The rates of convergence for a particle method applied to regular initial data (middle) and the blob method applied to discontinuous initial data (right) are slower.

 \begin{table}[h]
 \centering
 \begin{tabular}{c|c}
   &\hspace{-.8cm} $K(x) = \log |x|/2\pi \hspace{1.7cm} K(x) = |x|^2/2 \hspace{2cm} K(x) = |x|^3/3 $\\
 \hline 
 &\\
  \parbox[t]{1mm}{\multirow{1}{*}{\rotatebox[origin=c]{90}{regular initial data}}}
 & \hspace{-2mm}
 \includegraphics[trim={4cm 1.75cm 3.2cm 2.9cm},clip,width=.31\textwidth]{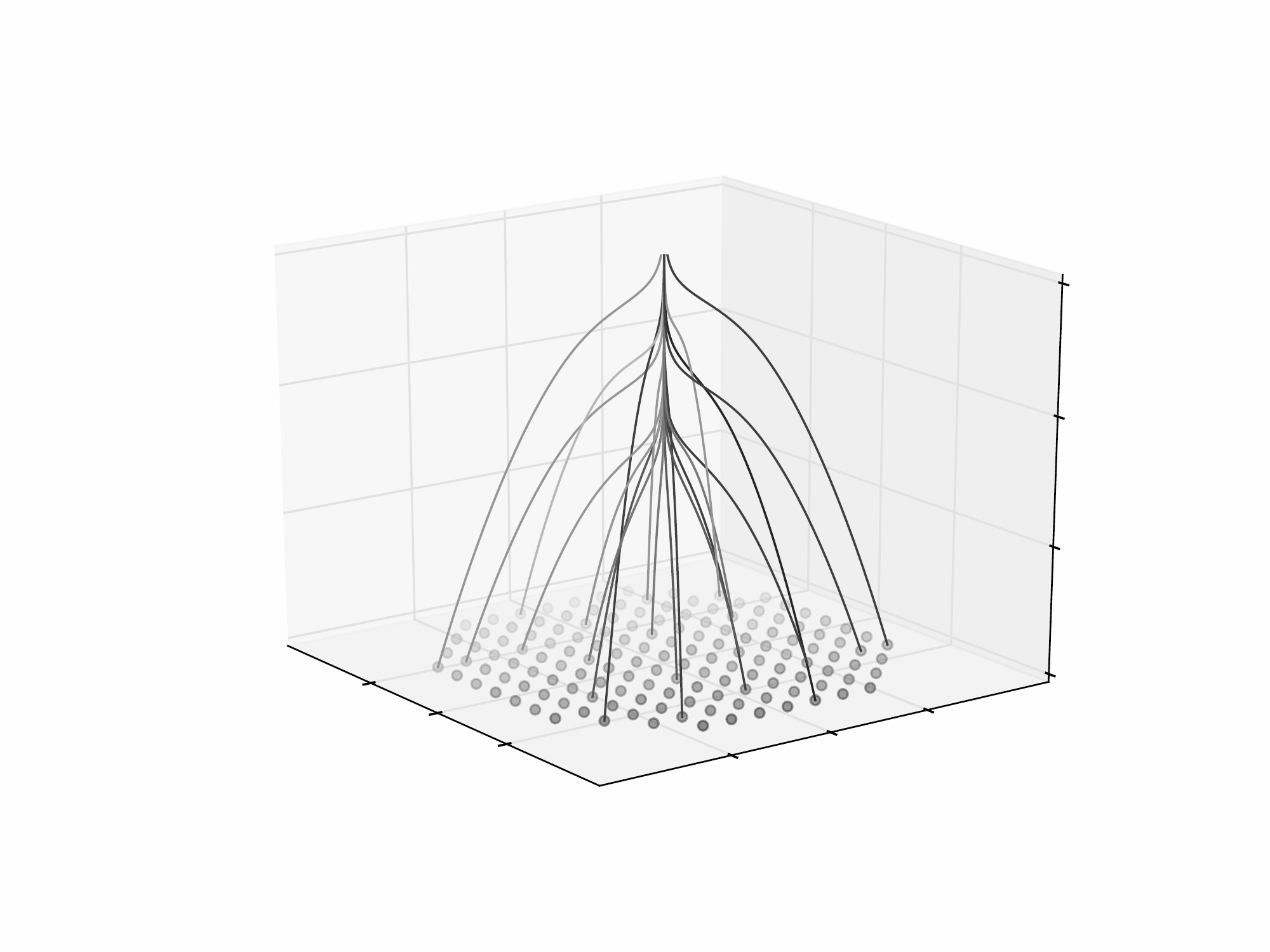}\hspace{-3mm}
\includegraphics[trim={4cm 1.75cm 3.2cm 2.9cm},clip,width=.31\textwidth]{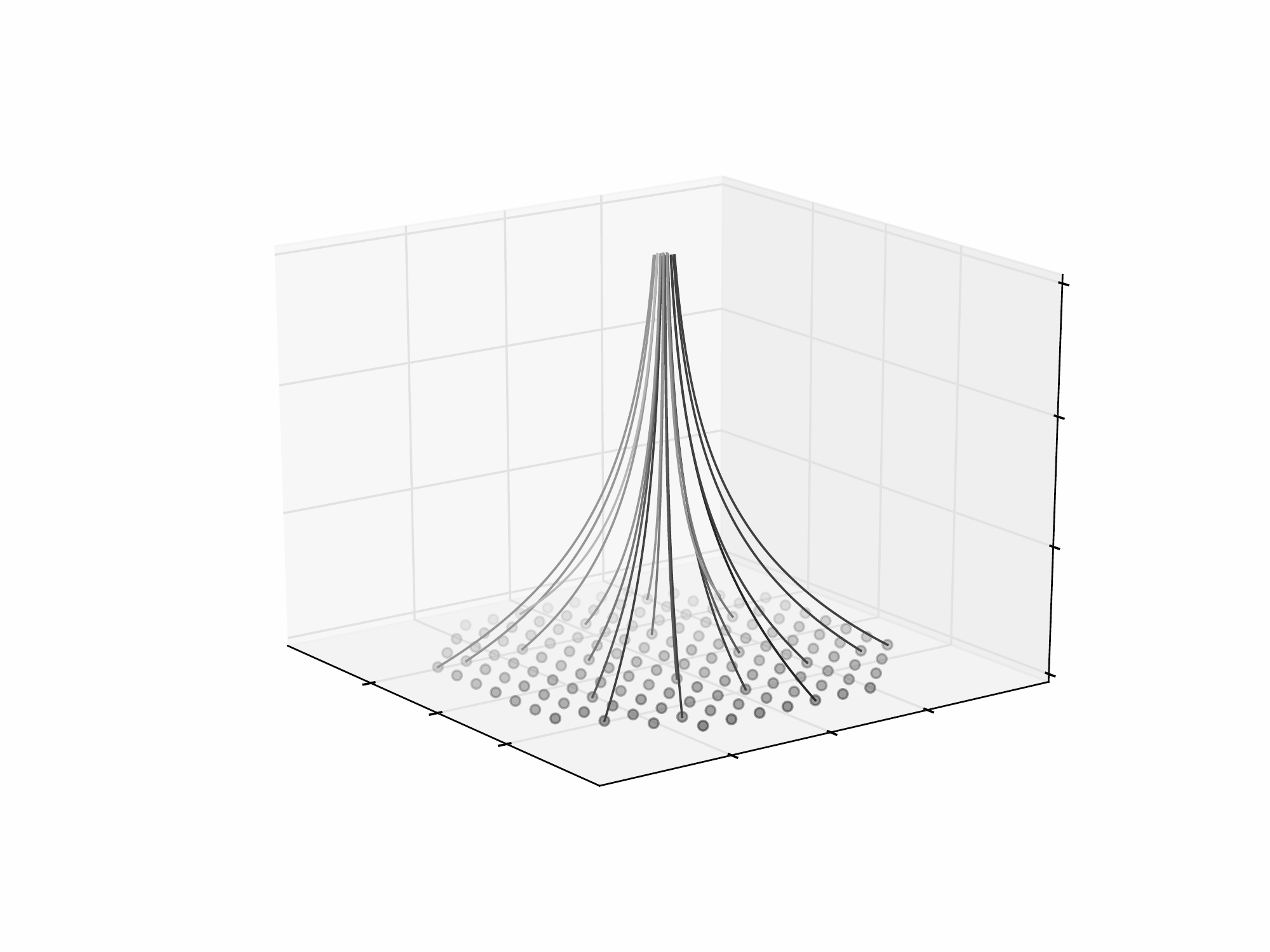}\hspace{-3mm}
\includegraphics[trim={4cm 1.75cm 1.5cm 2.75cm},clip,width=.34\textwidth]{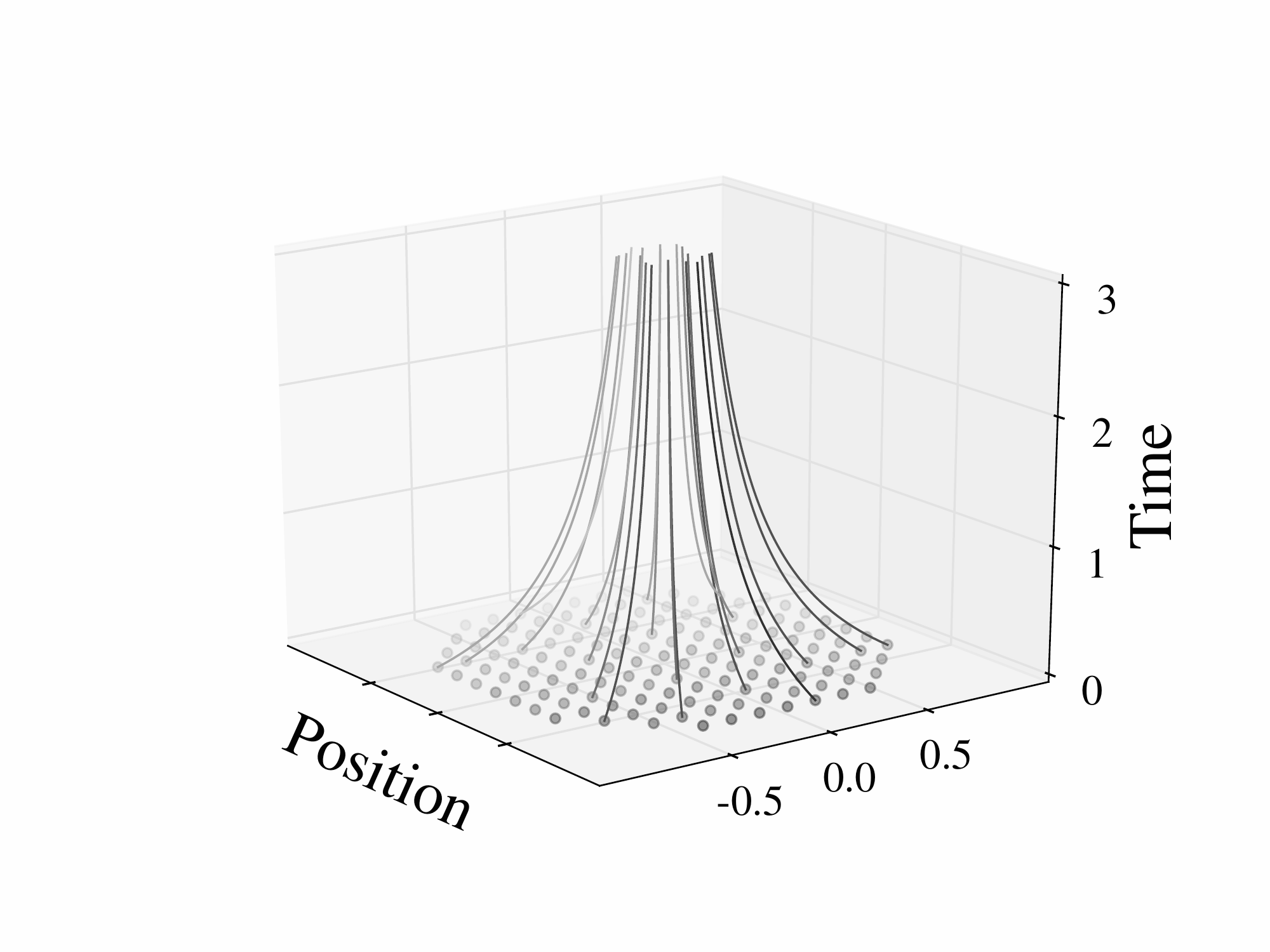}\\
 & \hspace{-2mm}
\includegraphics[trim={4cm 1.75cm 3.2cm 2.9cm},clip,width=.31\textwidth]{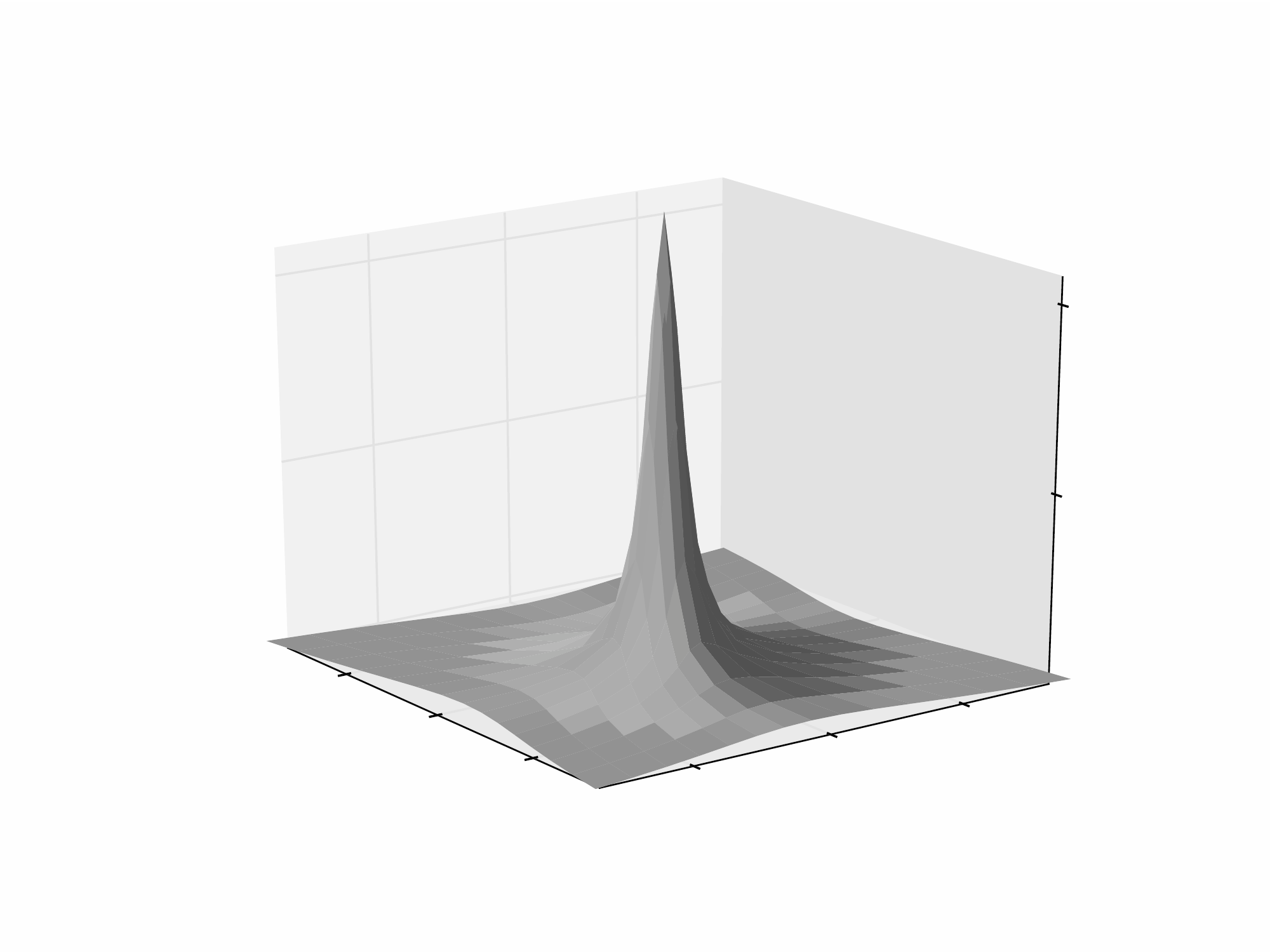}
\hspace{-3mm}
\includegraphics[trim={4cm 1.75cm 3.2cm 2.9cm},clip,width=.31\textwidth]{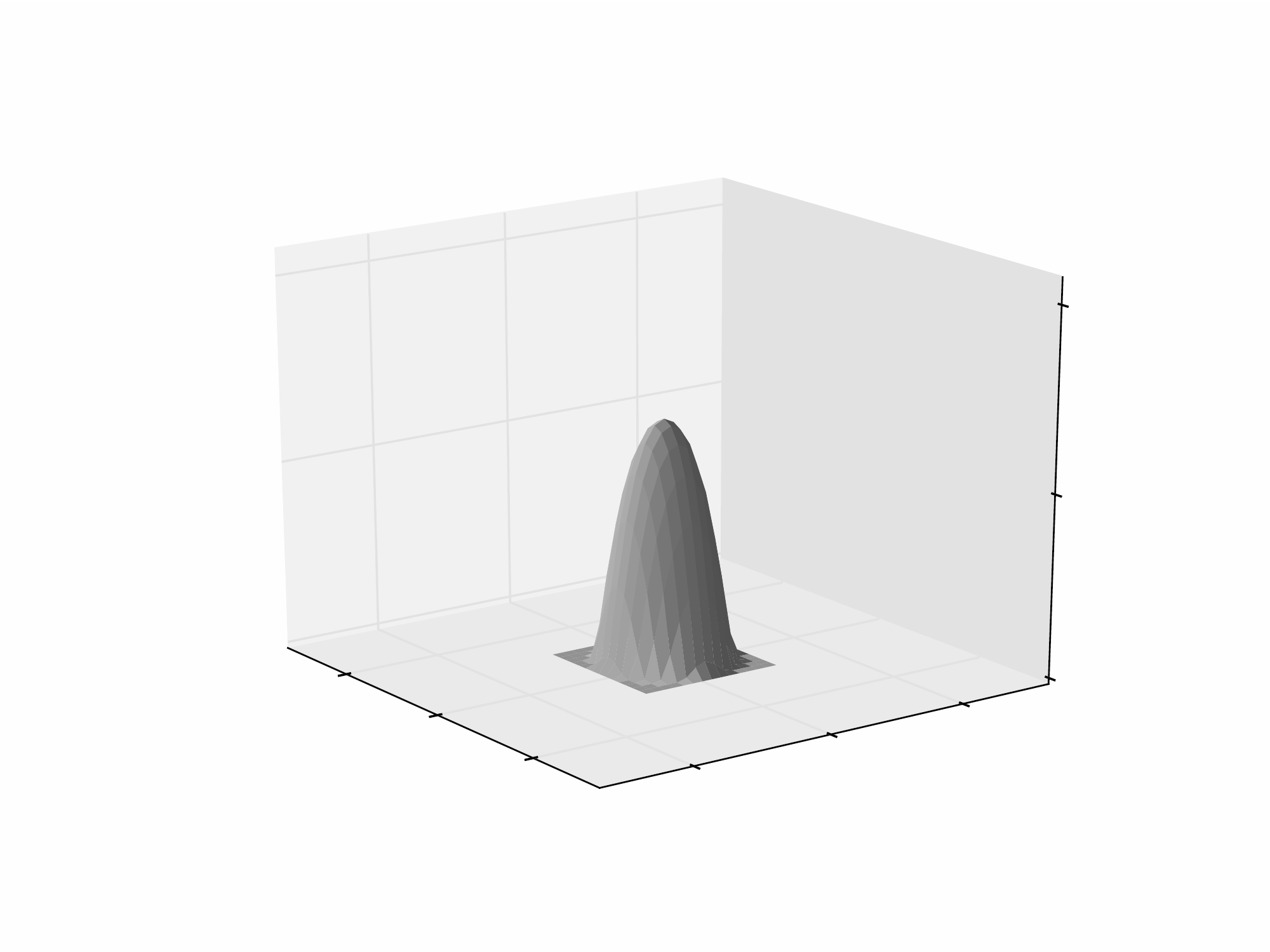} \hspace{-3mm}
\includegraphics[trim={4cm 1.75cm 1.5cm 2.75cm},clip,width=.34\textwidth]{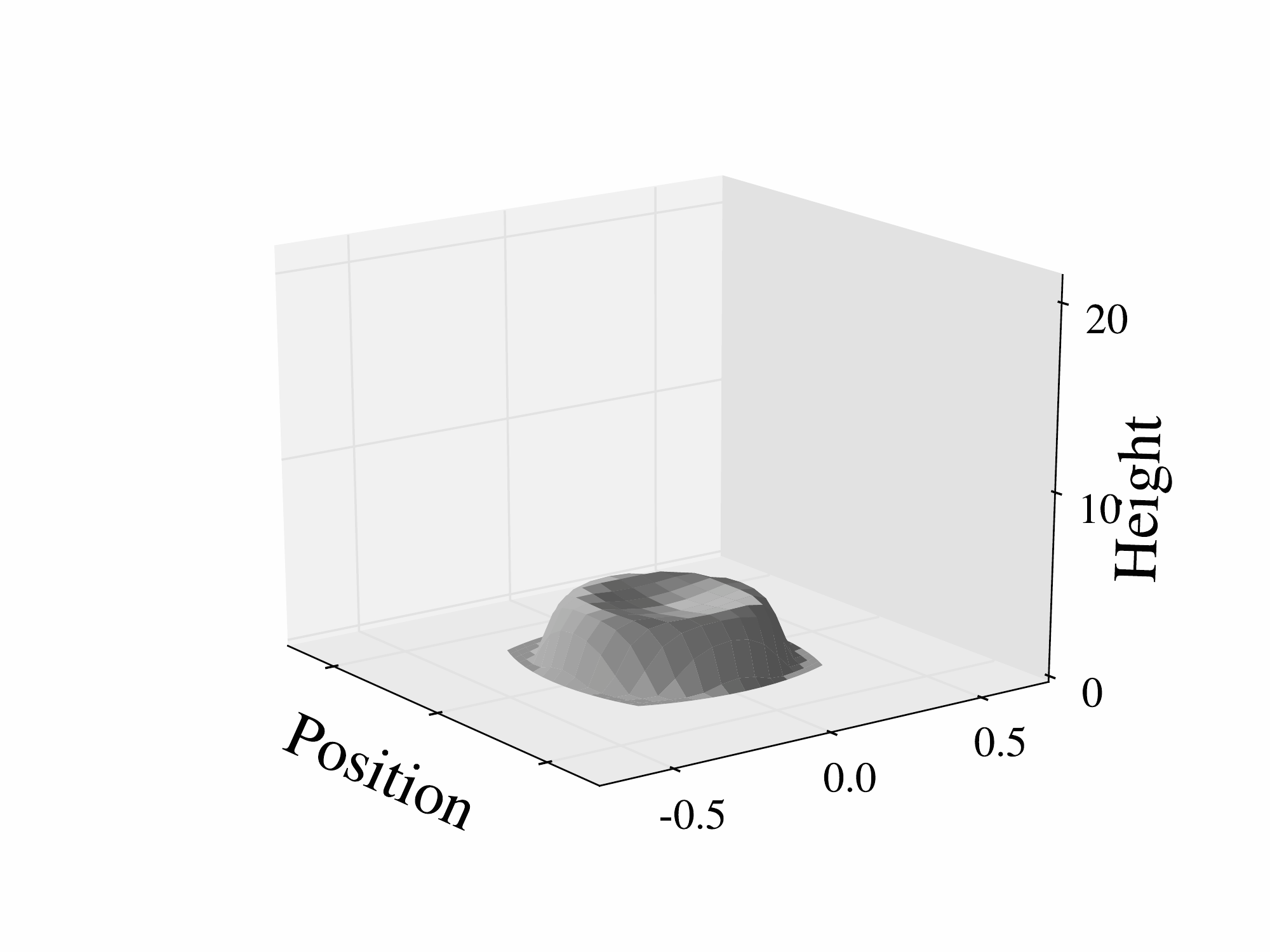}\\
\hline
& \\
 \parbox[t]{1mm}{\multirow{1}{*}{\rotatebox[origin=c]{90}{discontinuous initial data}}} & \hspace{-2mm}
\includegraphics[trim={4cm 1.75cm 3.2cm 2.9cm},clip,width=.31\textwidth]{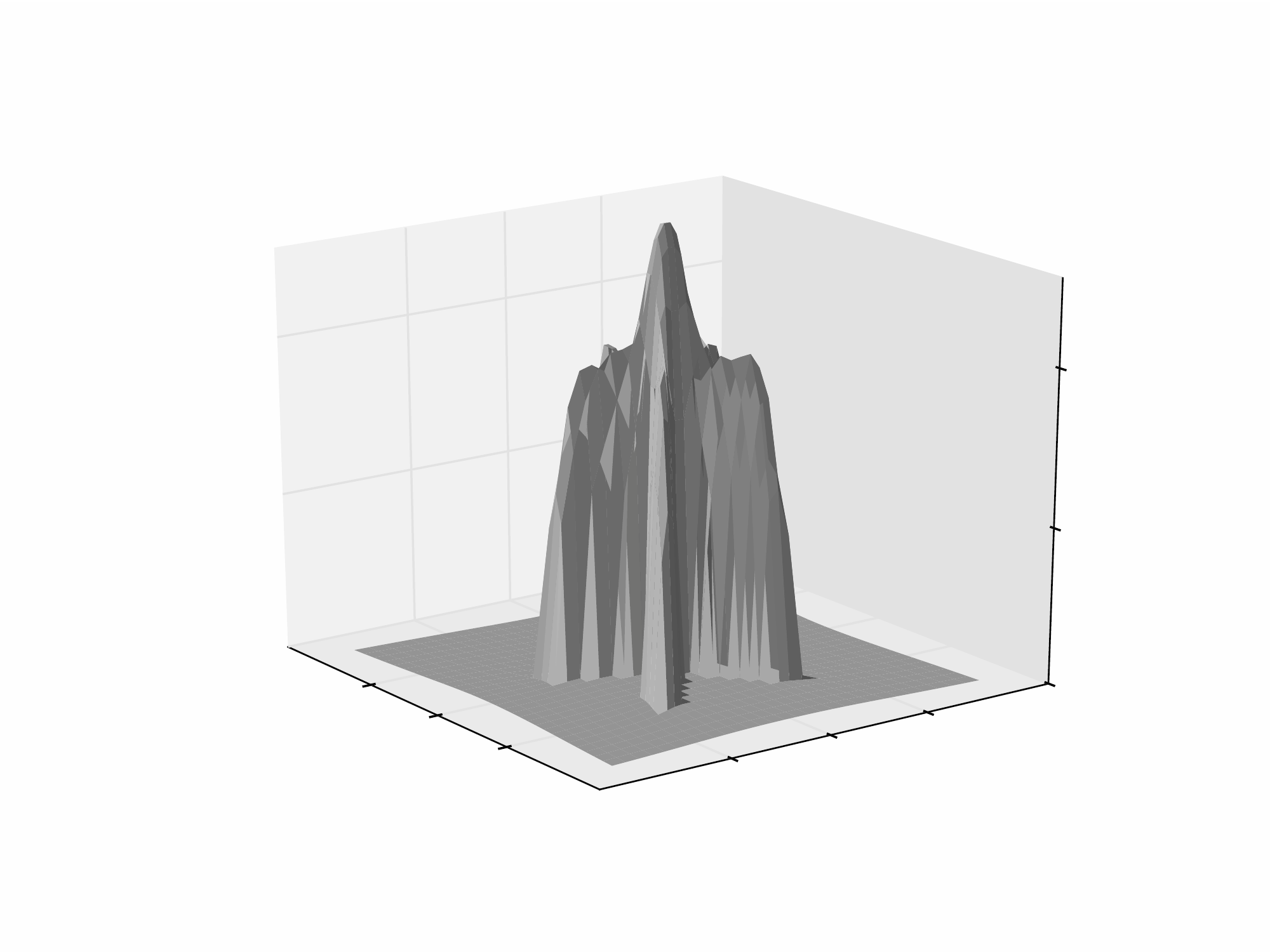}\hspace{-3mm}
\includegraphics[trim={4cm 1.75cm 3.2cm 2.9cm},clip,width=.31\textwidth]{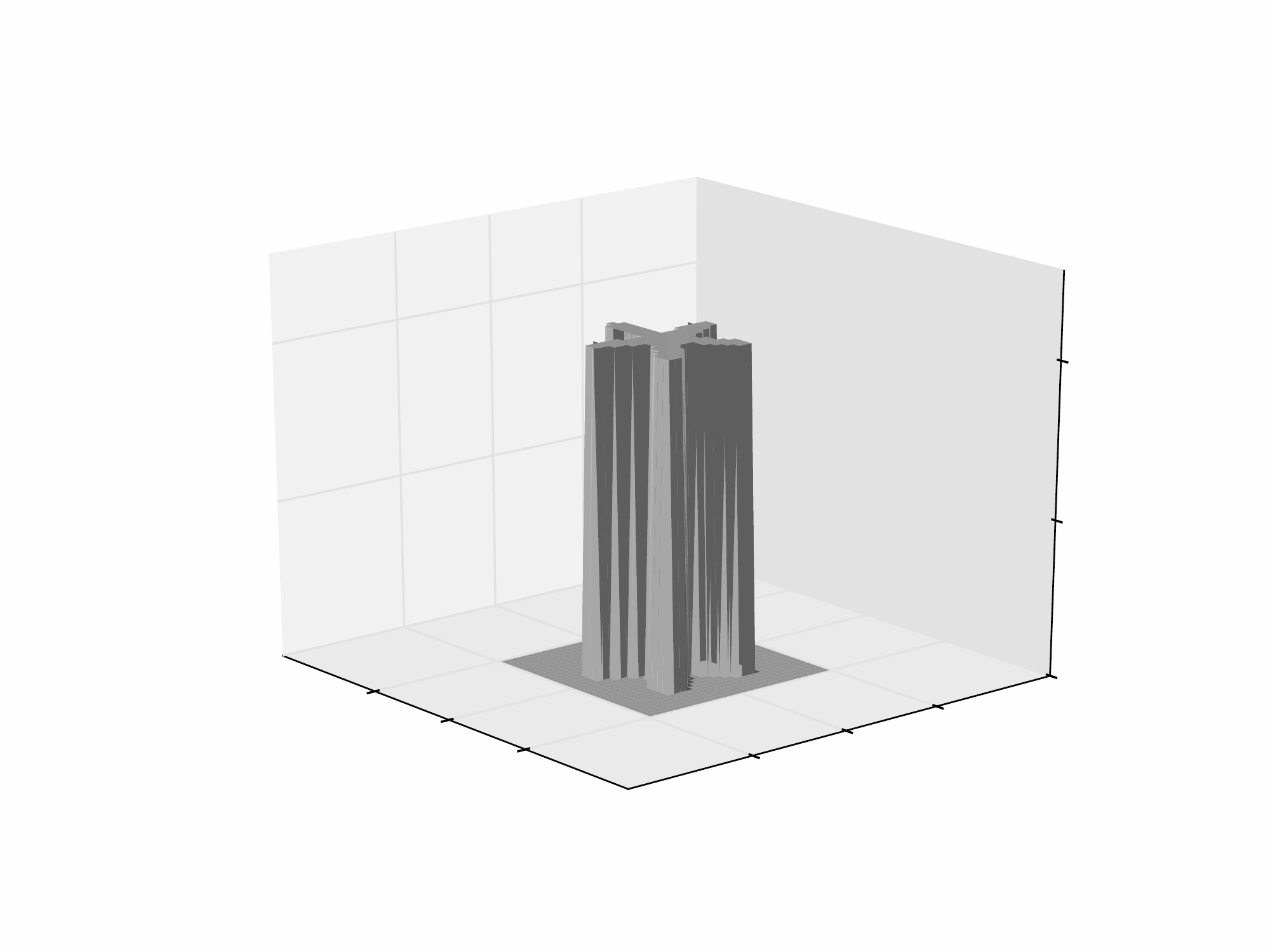} \hspace{-3mm}
\includegraphics[trim={4cm 1.75cm 1.5cm 2.75cm},clip,width=.34\textwidth]{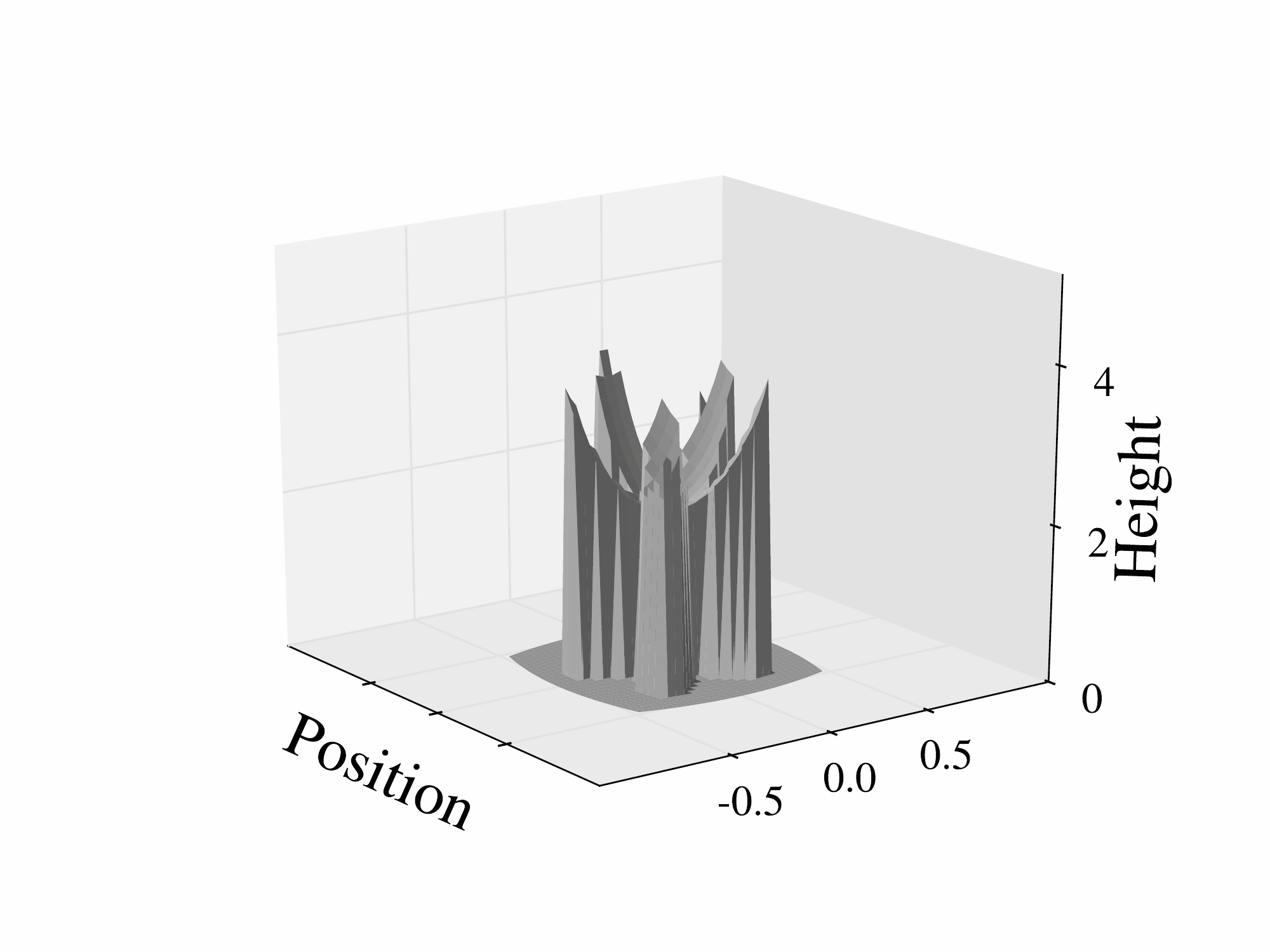} \\
  & \hspace{-2mm}
\includegraphics[trim={6.5cm 4.25cm 5.25cm 4.25cm},clip,width=.31\textwidth]{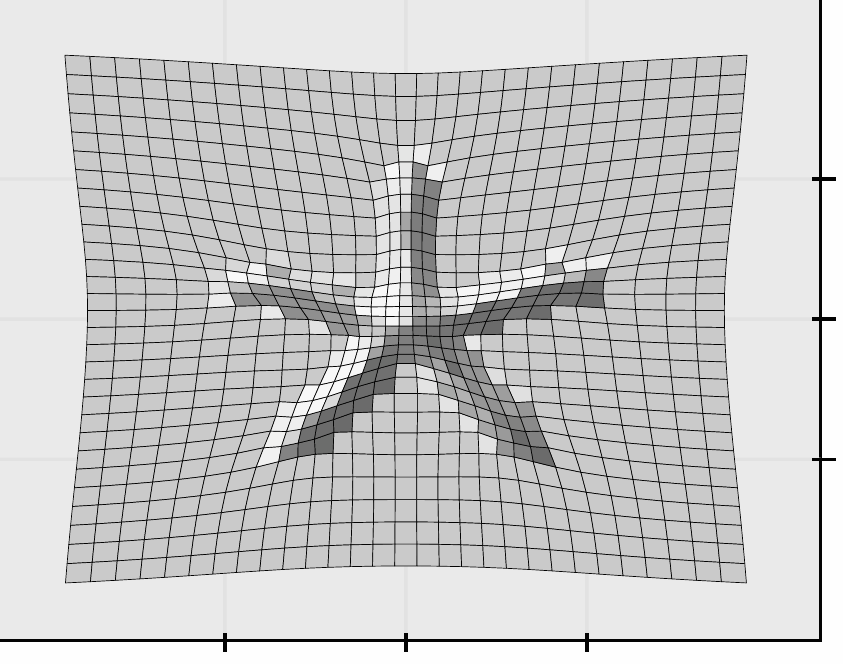}\hspace{-3mm}
\includegraphics[trim={6.5cm 4.25cm 5.25cm 4.25cm},clip,width=.31\textwidth]{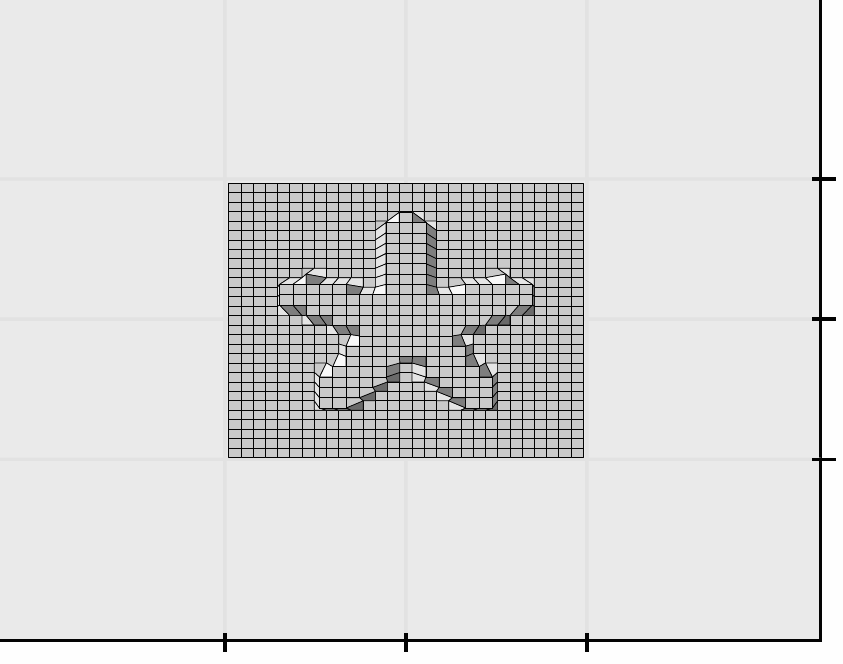}\hspace{-3mm}
\includegraphics[trim={6.5cm 4.25cm 5.25cm 4.25cm},clip,width=.31\textwidth]{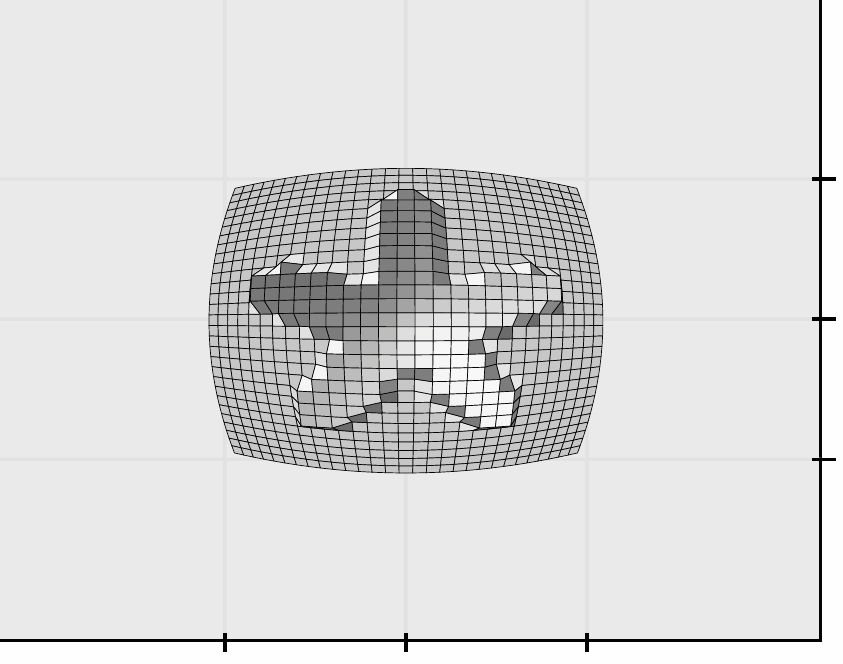} \\
& \\
\hline
 \end{tabular}
\captionof{figure}{A comparison of blob method solutions for the two dimensional aggregation equation. The scaling of the regularization is $\delta = h^q$ for $q=0.9$, and the mollifier (\ref{2d mollifier}) satisfies $m =4$.  In the top six plots, the unit square is discretized using $h \approx 0.13$. In the bottom six plots, the unit square is discretized using $h \approx 0.07$.} \label{2D}
 \end{table}
\quad \\ \noindent
\textbf{Newtonian, quadratic, and cubic potentials:} Figure \ref{2D} illustrates various phenomena of blob method solutions for three choices of kernel (Newtonian, quadratic, and cubic) as well as two choices of initial data (regular and discontinuous).

The top six plots illustrate the behavior of solutions with smooth, compactly supported initial data,
\begin{align*}
\rho_0(x) = \begin{cases} e^{1/(x^2-1)} &\mbox{if } |x|\leq 1 \ , \\ 
0  & \mbox{otherwise} \ . \end{cases}
\end{align*} 
The first row shows the space-time trajectories of twenty particles, and the second row shows the density at time $t= 1.4$. For $K= (\Delta)^{-1}$, finite time blowup of the classical solution occurs at $t = 1/\rho_0(0) \approx 1.2$, and the particle trajectories become very close at this time, bending to avoid collision. Two subsequent near-collisions occur before $t =3.2$. For $K(x) = |x|^2/2$, the particles converge to a point in infinite time, while for $K(x) = |x|^3/3$, the particles converge to a ring in infinite time. Each of these phenomena is reflected in the evolution of the density.

The second two rows consider discontinuous initial data given by the characteristic function on a star shaped patch,
\begin{align} \label{starrho0}
\rho_0(r, \theta) = \begin{cases} 1 &\mbox{if } r< \left(\sin^2(\frac{5\theta}{2})+\frac{1}{2}\right)/4 \ , \\ 
0  & \mbox{otherwise} \ . \end{cases}
\end{align} 
All six plots show the approximate density at $t = 0.8$, either from the side or above. For $K = (\Delta)^{-1}$, the density exhibits the same rounding due to regularization as in the one dimensional case. For $K(x)=|x|^2/2$, there is no rounding due to regularization since convolution with a mollifier of accuracy $m=4$ preserves polynomials of degree less than 4, including $\grad K$ and $\Delta K$. For $K(x) = |x|^3$, the density again converges to a ring in infinite time.

\quad  \\ \noindent
 \begin{table}[h]
 \centering
 \begin{tabular}{c|c|c}
 \hspace{-3mm}  $K(x) = \frac{|x|^4}{4} - \frac{\log |x|}{2\pi}$ & \hspace{-2mm} $K(x) =\frac{|x|^4}{4} - \frac{|x|^{3/2}}{3/2}$ &\hspace{-2mm}  $K(x) = \frac{|x|^7}{7} - \frac{|x|^{3/2}}{3/2}$\\
 \hline 
 & & \\\hspace{-2mm} 
 \includegraphics[trim={.3cm .7cm .6cm 1.3cm},clip,height=.34\textheight]{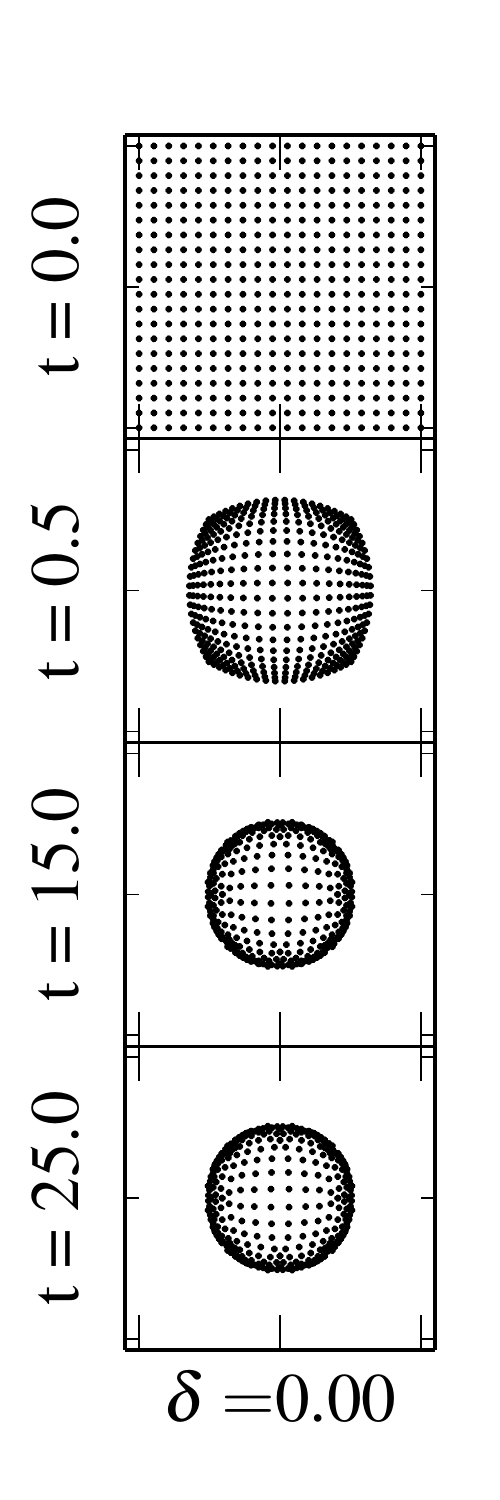} \hspace{-2mm} 
\includegraphics[trim={1.26cm .7cm .6cm 1.3cm},clip,height=.34\textheight]{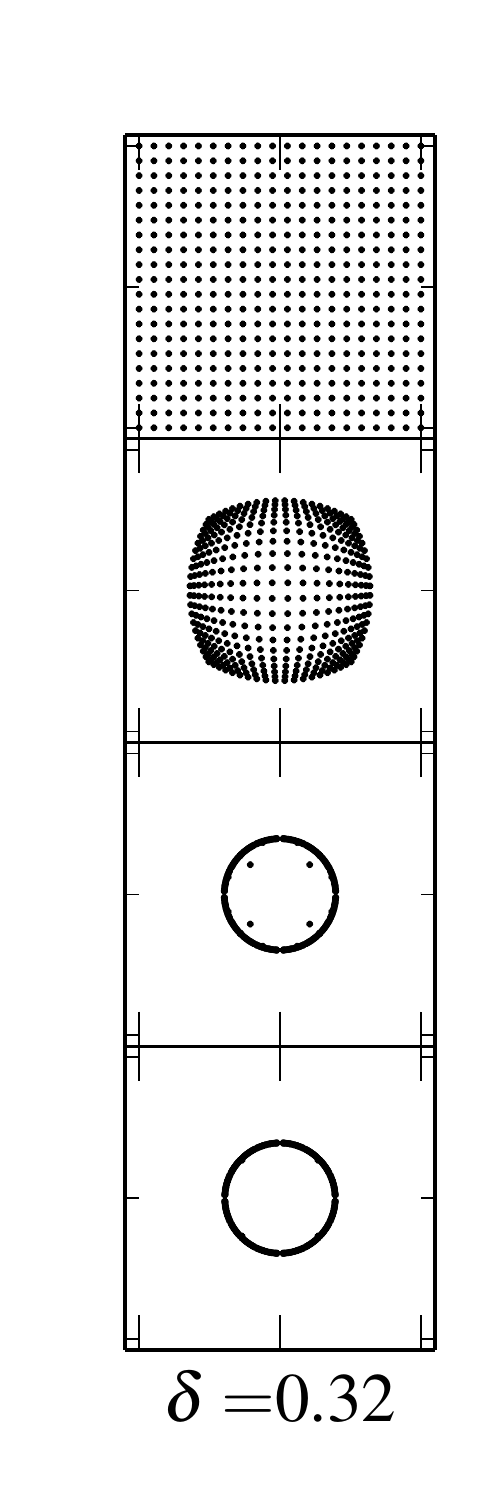}
& \hspace{-2mm} 
\includegraphics[trim={.3cm .7cm .6cm 1.3cm},clip,height=.34\textheight]{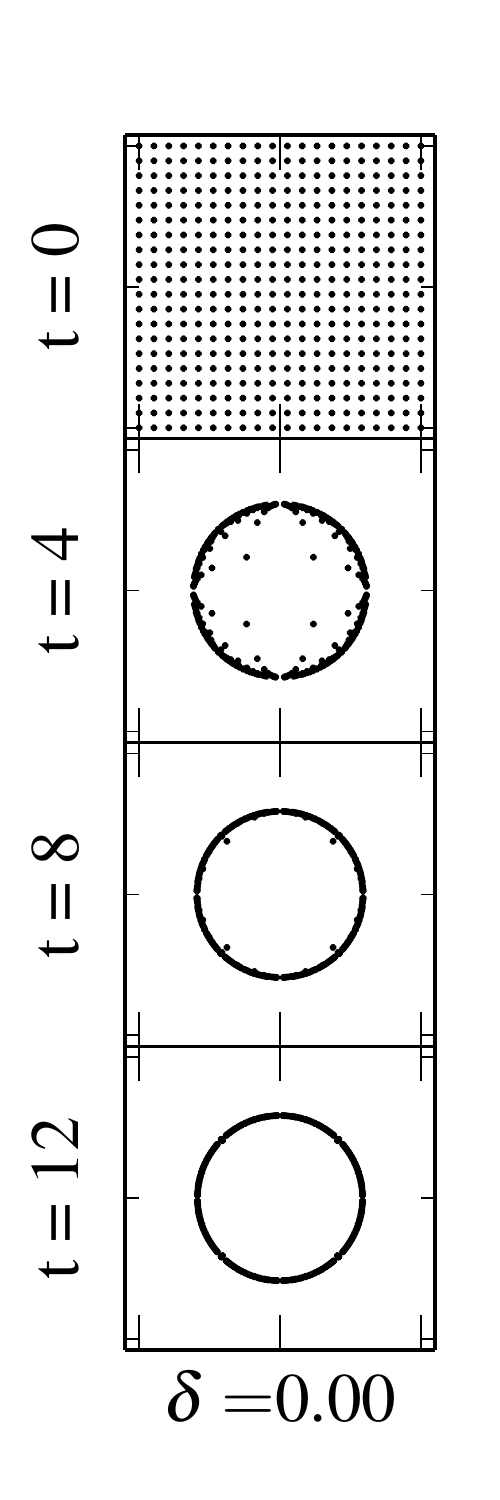} \hspace{-2mm} 
\includegraphics[trim={1.26cm .7cm .6cm 1.3cm},clip,height=.34\textheight]{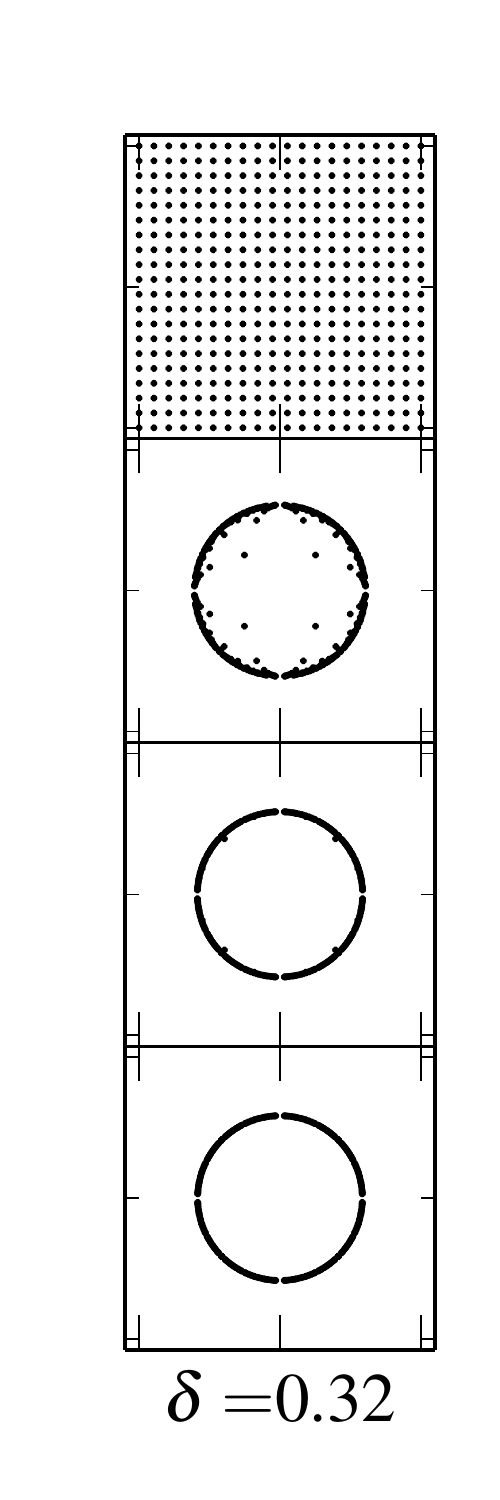}
& \hspace{-2mm} 
\includegraphics[trim={.3cm .7cm .6cm 1.3cm},clip,height=.34\textheight]{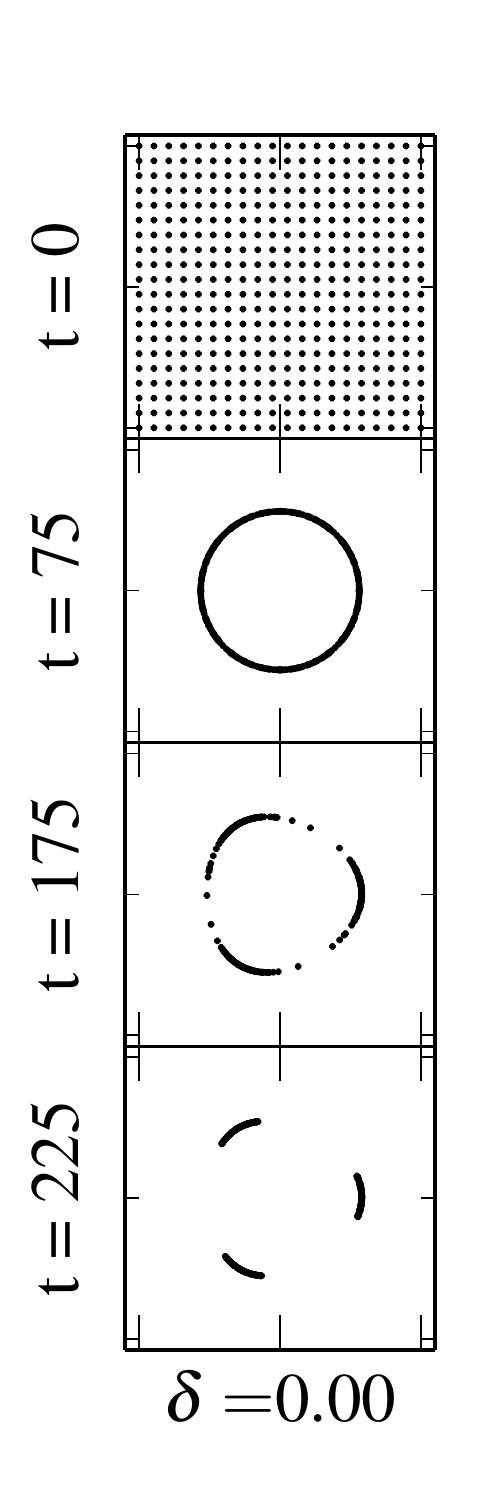} \hspace{-2mm} 
\includegraphics[trim={1.26cm .7cm .6cm 1.3cm},clip,height=.34\textheight]{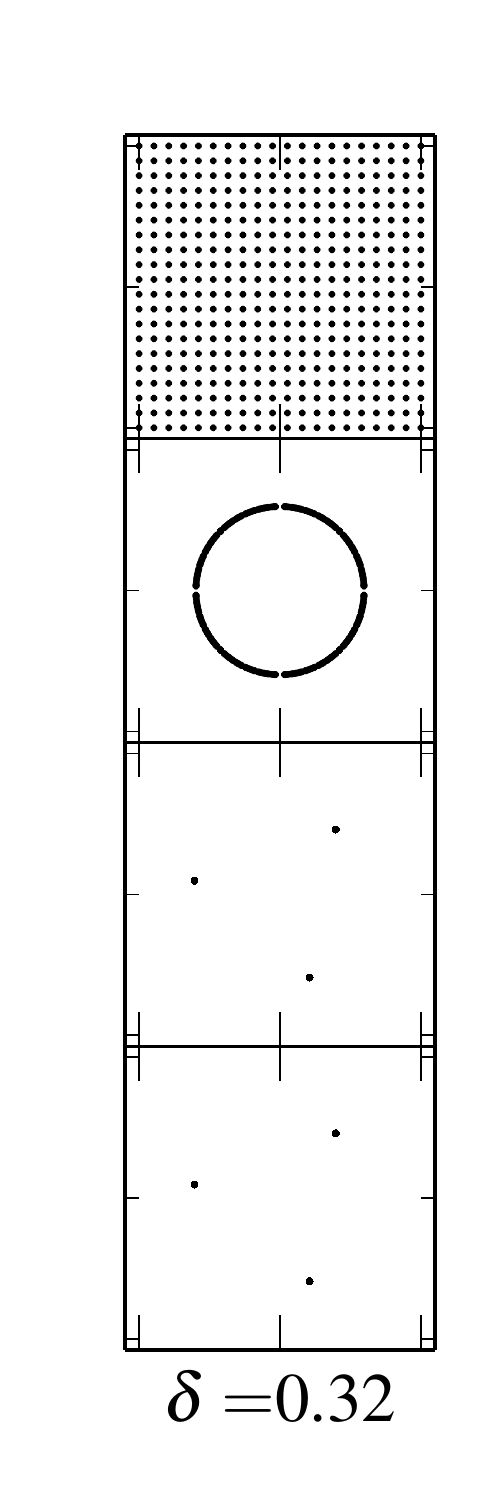}
\\
& & \\
\hline
 \end{tabular}
\captionof{figure}{A comparison of numerical solutions to the aggregation equation. The mollifier (\ref{2d mollifier}) satisfies $m =4$, and the unit square is discretized with $h \approx .11$. For the first kernel, we use an analytic expression for $\grad K_\delta$. For the second two kernels, we numerically compute the convolution $\grad K * \psi_\delta$ in radial coordinates on a ball of radius 2.5 with 100 grid points.} \label{repulsiveattractivetable}
 \end{table}
 
\noindent
\textbf{Repulsive-attractive power law kernels:} Figure \ref{repulsiveattractivetable} displays the evolution of particle trajectories for numerical solutions to the two dimensional aggregation equation with repulsive-attractive power law kernels. In the past few years, there has been significant interest in such kernels, due to the  stationary patters which develop \cite{BalagueCarrilloLaurentRaoul,BalagueCarrilloLaurentRaoul_Dimensionality,Bertozzietal_RingPatterns,BalagueCarrilloYao, BertozziLaurentLeger, BertozziGarnettLaurent, BertozziBrandman, BertozziCarrilloLaurent, 5person2, CarrilloChipotHuang, Dong, FellnerRaoul, FellnerRaoul2, FetecauHuangKolokolnikov, FetecauHuang, HuangBertozzi, HuangBertozzi2, Kolokolnikovetal_StabilityRingPatterns, Poupaud,Raoul, SunUminskyBertozzi, TopazBertozzi1,HuangWitelskiBertozzi}.

Figure \ref{repulsiveattractivetable} compares the results of a particle method with the blob method. In all six plots, the initial data is
\begin{align} \label{pfn2rho0}
\rho_0(x) = \begin{cases} C(1-x^2)^{2} &\mbox{if } |x|\leq 1 \ , \\ 
0  & \mbox{otherwise} \ , \end{cases}
\end{align} 
with $C$ chosen so that $\int \rho_0 = 1$. 

For each of the three choices of kernel, the corresponding plots demonstrate that a large regularization parameter---in this case $\delta \approx 0.32$---can affect the steady states of blob method solutions. This effect vanishes as $\delta$ becomes small, but it still of interest since it illustrates the important role of the kernel's regularity in the dimensionality of steady states. 
Balagu\'e, Carrillo, Laurent, and Raoul proved that the dimensionality of the support of steady state solutions depends on the strength of the repulsive forces at the origin \cite{BalagueCarrilloLaurentRaoul_Dimensionality}. For a  repulsive-attractive power law kernel,
\[ K(x) = |x|^a/a - |x|^b/b \ , \quad a>b \ , \]
the repulsive part is more singular than the attractive part, so regularizing the kernel by convolution with a mollifier has a greater effect for the repulsive part, and we expect this to dampen the repulsive forces.

For $K(x) = |x|^4/4 - \log |x|/2\pi$, we recover the radial, integrable, compactly supported steady states found by Fetecau, Huang, and Kolokolnikov \cite{FetecauHuangKolokolnikov}. A large regularization parameter causes the steady states to collapse to a ring.
For $K(x) =|x|^4/4 - |x|^{3/2}/(3/2)$, we recover the stable delta ring found by Kolokonikov, Sun, Uminsky, and Bertozzi \cite{Kolokolnikovetal_StabilityRingPatterns}. A comparison of the particle and blob methods at $t=8$ indicates that the blob method solution converges to the ring more quickly.
For $K(x) = |x|^7/7 - |x|^{3/2}/(3/2)$, we recover the ring formation and breakup found by Bertozzi, Sun, Kolokolnikov, Uminsky, and Von Brecht \cite{Bertozzietal_RingPatterns}. A large regularization parameter lowers the dimension of the steady state to three point masses.

\begin{table}[h]
 \centering
 \begin{tabular}{c|cc}
 \hspace{-3mm}  $K(x) = e^{-|x|} - 2 e^{-|x|/2}$ & \hspace{-2mm} $K(x) = 2e^{-|x|} - 2 e^{-|x|/2}$ &\hspace{-2mm}  \\
\vspace{-3mm} & & \\
regular initial data & regular initial data & discontinuous initial data \\
 \hline 
 & & \\\hspace{-3mm} 
 \includegraphics[trim={.3cm .7cm .6cm 1.3cm},clip,height=.34\textheight]{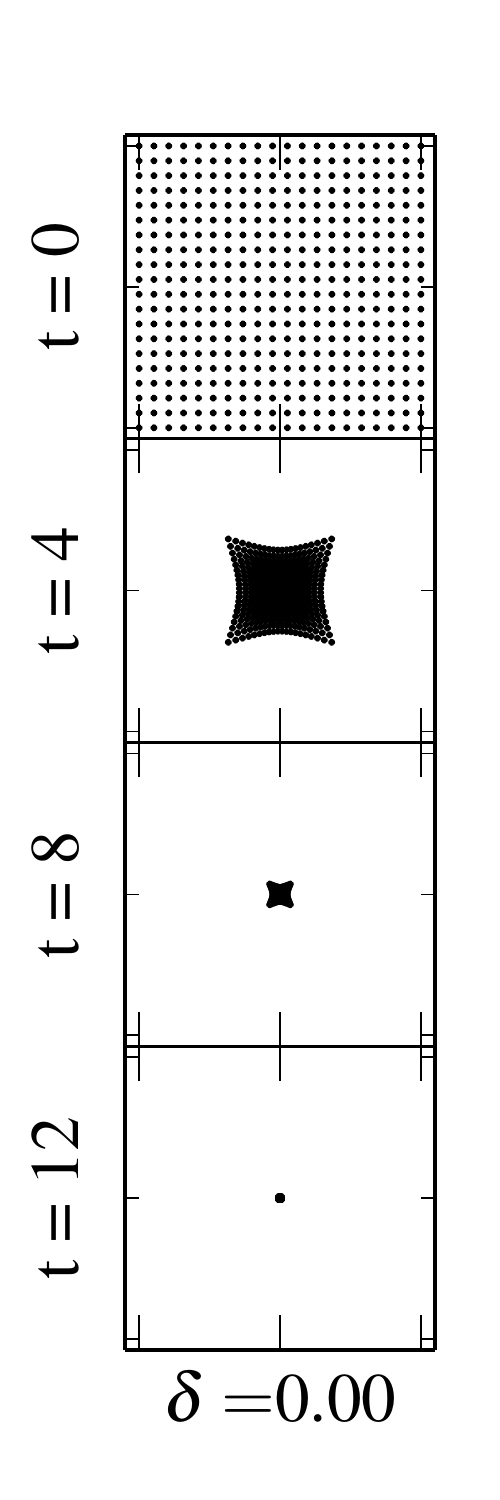} \hspace{-3mm} 
\includegraphics[trim={1.26cm .7cm .6cm 1.3cm},clip,height=.34\textheight]{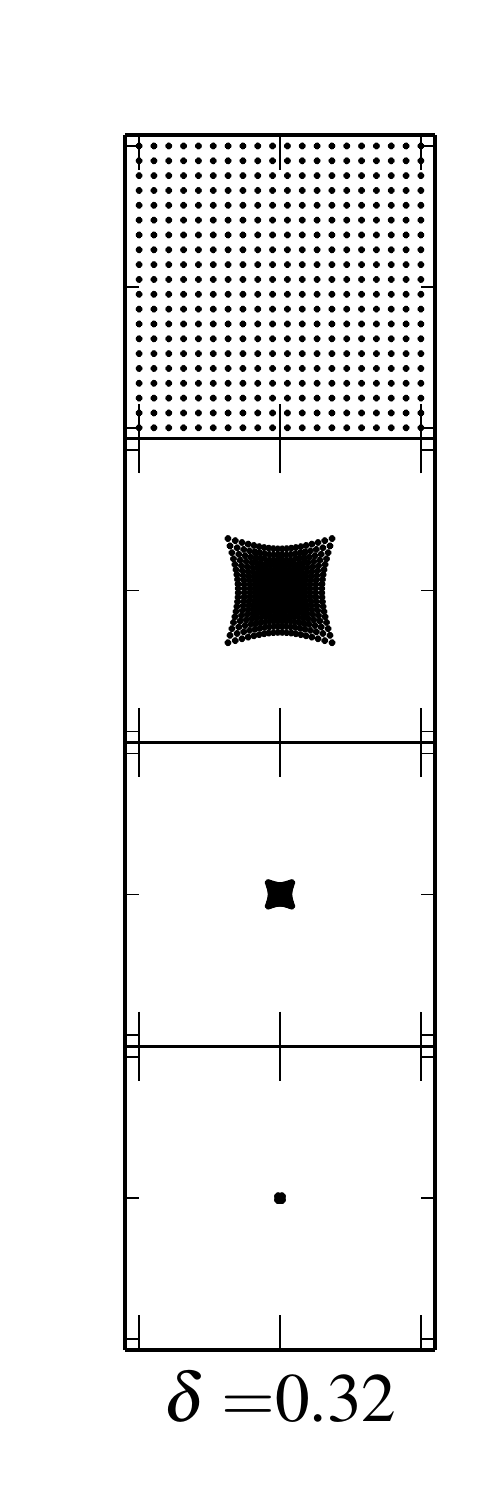}
& \hspace{-2mm} 
\includegraphics[trim={.3cm .7cm .6cm 1.3cm},clip,height=.34\textheight]{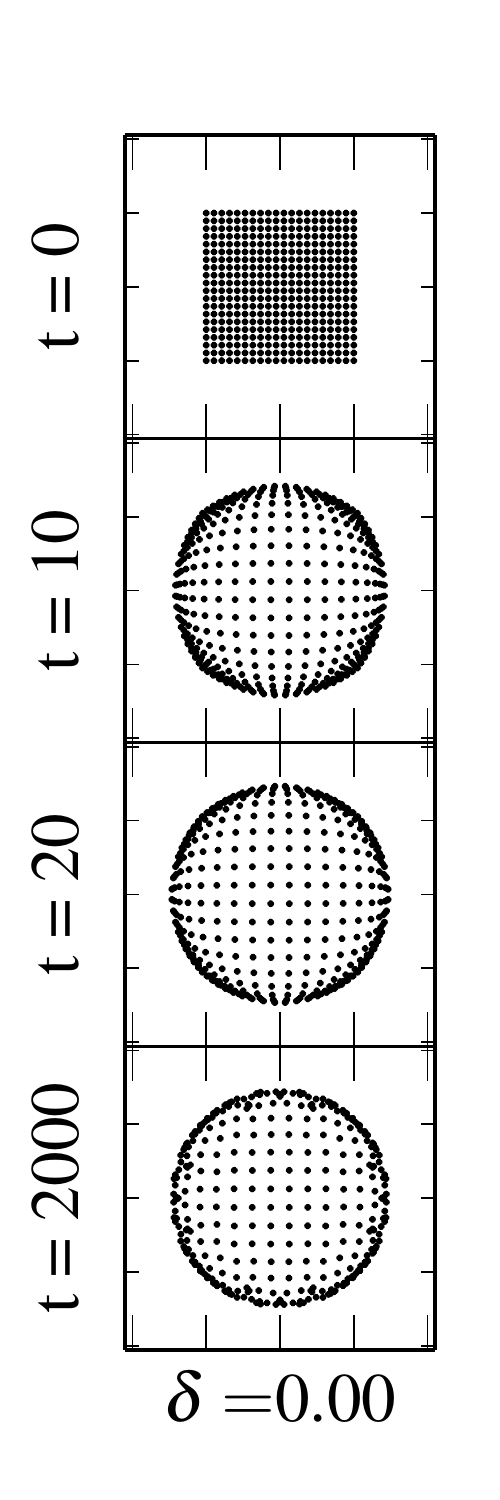} \hspace{-3mm} 
\includegraphics[trim={1.26cm .7cm .6cm 1.3cm},clip,height=.34\textheight]{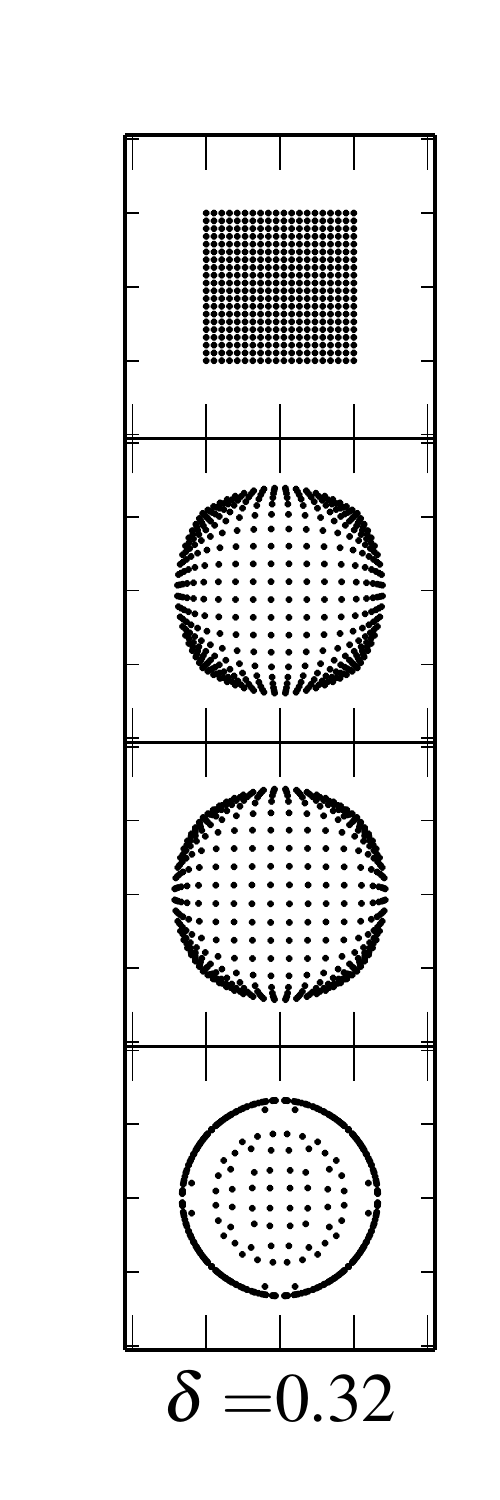}
& \hspace{-3mm} 
\includegraphics[trim={.3cm .7cm .6cm 1.3cm},clip,height=.34\textheight]{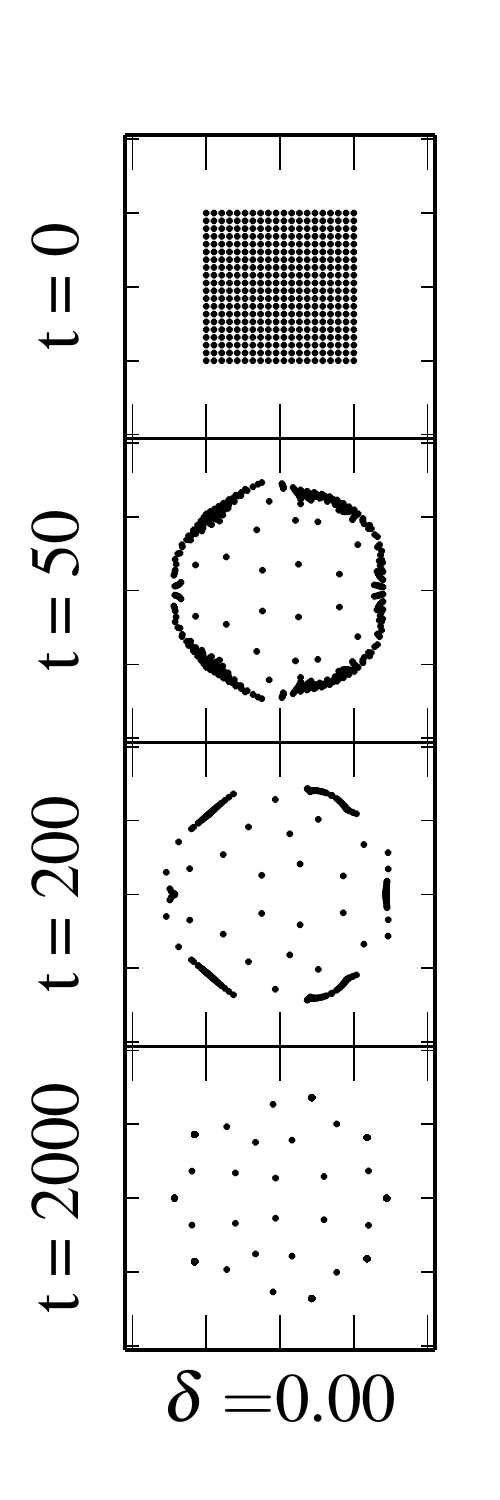} \hspace{-3	mm} 
\includegraphics[trim={1.26cm .7cm .6cm 1.3cm},clip,height=.34\textheight]{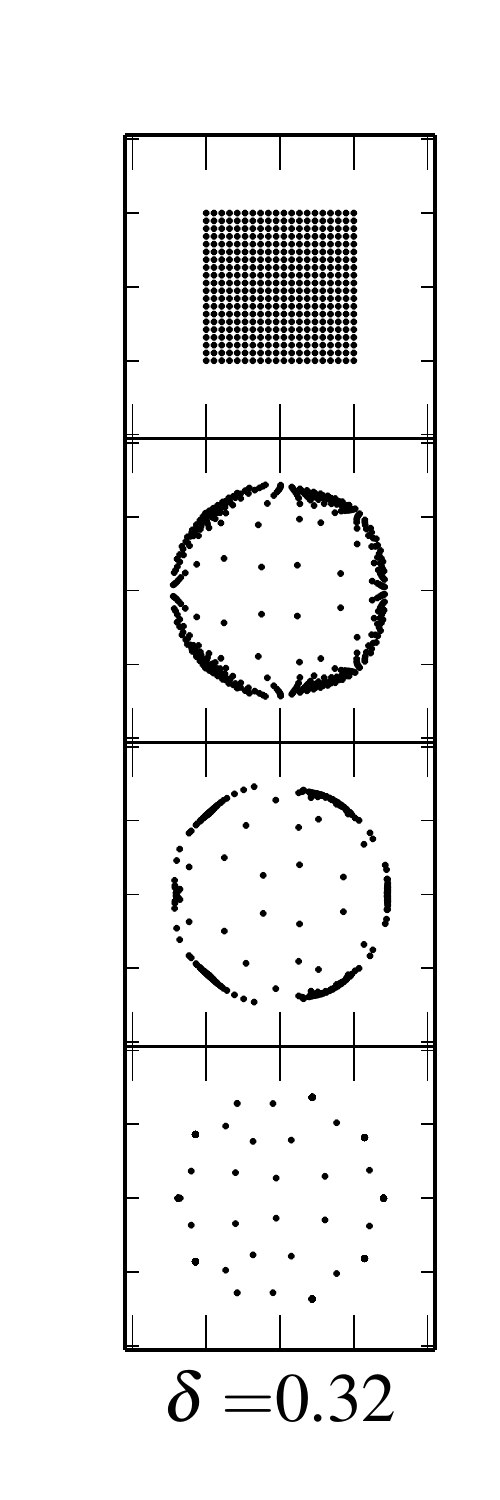}
\\
& & \\
\hline
 \end{tabular} 
\captionof{figure}{A comparison of numerical solutions to the aggregation equation computed by particle and blob methods. The mollifier (\ref{m4d1mollifier}) is chosen so that $m =4$, and the unit square is discretized with $h \approx 0.11$. The spatial scale in the first column is the unit square, and the spatial scale is doubled in the second and third columns. For all three kernels, we numerically compute the convolution $\grad K * \psi_\delta$ and $\Delta K * \psi_\delta$ in radial coordinates on a ball of radius 5 with 200 grid points.}
\label{morsetable}
 \end{table}
 
\quad \\  \noindent
 \textbf{Repulsive-attractive Morse potential:} Figure \ref{morsetable} shows the evolution of particle trajectories for numerical solutions to the aggregation equation when the kernel is given by a repulsive-attractive Morse potential, comparing the results of a particle method with the blob method. In the first two columns, the initial data is given by (\ref{pfn2rho0}), with $C$ chosen so that $\int \rho_0 =1$. In the third column, the initial data is given by a star shaped patch (\ref{starrho0}).
 
We only observe the effect of a large regularization parameter for the kernel $K(x) = 2e^{-|x|} - 2 e^{-|x|/2}$ with regular initial data. This is likely due to the fact that the repulsive and attractive components of the kernel have the same regularity, so regularization does not disproportionately affect one without the other. 

\section{Conclusion} \label{ConclusionSection}

We develop a new numerical method for the aggregation equation for a range of  kernels, including singular kernels, kernels which grow at infinity, and repulsive-attractive kernels. We prove that our blob method solutions converge to classical solutions of the aggregation equation with arbitrarily high rates of convergence, depending on the choice of blobs. We also provide several numerical examples which confirm our theoretically predicted rates of convergence and illustrate key properties of the method, including long-time existence of particle trajectories.

As analysis of numerical methods for the aggregation equation is a relatively new area of interest, there are several directions for future work. First, our estimates do not differentiate between purely repulsive, purely attractive, or repulsive-attractive kernels, and our numerical method may be improved by leveraging the different dynamics in each of these cases. In particular, a numerical method for repulsive or repulsive-attractive kernels might take advantage of the formation of steady states to obtain global in time bounds on the error. Preliminary analysis in this direction indicates that it will be necessary to measure error in different norms. Our current discrete $L^p$ norms measure the distance between exact and approximate solutions along particle trajectories that begin at the same grid point. However, small perturbations of the functional can change where particles settle in a steady state. This causes the discrete $L^p_h$ norm of the error to be large, even if the overall structure of the steady state is the same.

Another direction for future work is to adapt the blob method to the Keller-Segel equation by the addition of (possibly degenerate) diffusion. From the perspective of classical vortex methods, the addition of diffusion corresponds to passing from the Euler equations to the Navier Stokes equations, so it may be possible to adapt random vortex or core spreading methods from the Navier Stokes equations to the Keller-Segel equation \cite{Rossi, Chorin,MarchioroPulvirenti,Goodman,Long}. 
Our method might also be extended to the Keller-Segel equation by separately simulating the effects of aggregation and (degenerate) diffusion at each time step. Yao and Bertozzi developed such a method by using a radial particle method to simulate aggregation and an implicit finite volume scheme to simulate degenerate diffusion \cite{YaoBertozzi}.
 The blob method would allow one to perform the aggregation step for non-radial solutions with a high degree of accuracy.

Finally, our result on the convergence of the blob method to classical solutions  might be extended to weak measure solutions. Lin and Zhang \cite{LinZhang} used a blob method to prove existence of weak measure solutions for the two dimensional aggregation equation when $K = (-\Delta)^{-1}$. The blob method might also be used to prove existence of weak measure solutions for the multidimensional aggregation equation for a range of singular kernels. This would build on the work of Bertozzi, Garnett, and Laurent for radial initial data and kernels of the form $K(x) = |x|^a$, $2-d \leq a <2$ \cite{BertozziGarnettLaurent}. Such a result would be of particular interest for kernels in the range $K(x) = |x|^{-1}$ to $K(x) = |x|$, for which uniqueness may also hold.

\section{Appendix} \label{Appendix}

\subsection{Discrete $L^p$ and Sobolev Norms} \label{discrete norms section}
In this section, we provide references and proofs for the discrete norm inequalities given in Proposition \ref{discrete norm prop}. We begin with the following lemma regarding extensions for discrete Sobolev spaces.
\begin{lemma}[extensions from $W^{1,p}_h(Q_R)$ to $W^{1,p}_h(\Rd)$] \label{extlemma}
Let $Q_R = [-R,R]^d$ for $R \geq 3$. For all $0<h<1$, there exists $P: W^{1,p}_h(Q_R) \to W^{1,p}_h(\Rd)$ satisfying $Pg|_{Q_R} = g$ and $\|Pg\|_{W^{1,p}_h(\Rd)} \leq C_{d,p} \|g\|_{W^{1,p}_h(Q_R)}$.
\end{lemma}
\begin{proof}
Denote the vertices of $Q_R$ by $\vec{v} \in \{ (v_1, \dots, v_d) : v_i = \pm R \}$. Define a partition of unity $\{\eta_{\vec{v}}\}$ on $Q_R$ so that $0 \leq \eta_{\vec{v}} \leq 1$, $\sum_{\vec{v}} \eta_{\vec{v}} =1$, and $\eta_{\vec{v}}(x)$ vanishes on the edges of the cube opposite $\vec{v}$, i.e. whenever $x = (x_1, \dots, x_d)$ satisfies $|x_j - (-v_j)|< 1$.

For any $g \in W^{1,p}_h(Q_R)$, define $g_{\vec{v}} = \eta_{\vec{v}} g $. We claim that for all $g_{\vec{v}}$, there exists an extension $P g_{\vec{v}}$ so that $Pg_{\vec{v}}|_{Q_R} = g_{\vec{v}}$ and  $ \|Pg_{\vec{v}}\|_{W^{1,p}_h(\Rd)} \leq C_{d,p} \|g_{\vec{v}}\|_{W^{1,p}_h(Q_R)} $. Supposing this claim holds, define the extension $Pg$ of $g$ by $Pg = \sum_{\vec{v}} Pg_{\vec{v}}$. This satisfies $P g |_{Q_R} = \sum_{\vec{v}} Pg_{\vec{v}} |_{Q_R}  = \sum_{\vec{v}} g_{\vec{v}} = g$ and
\[ \|Pg\|_{W^{1,p}_h(\Rd)} \leq \sum_{\vec{v}} \|Pg_{\vec{v}}\|_{W^{1,p}_h(\Rd)} \leq \sum_{\vec{v}}C_{d,p} \|g_{\vec{v}}\|_{W^{1,p}_h(Q_R)}  \leq C_{d,p} \|g\|_{W^{1,p}_h(Q_R)} \ . \]

Thus it remains to construct the extension $Pg_{\vec{v}}$. By translation invariance, we suppose that $Q_R = [0, 2R]^d$, and by rotational symmetry, we suppose that $\vec{v} = (0,0, \dots, 0)$. Then $g_{\vec{v}} \in W^{1,p}(Q_R)$ and $g_{\vec{v}}(x) = 0$ whenever $x = (x_1, \dots, x_d)$ satisfies $|x_j - 2R|< 1$. To extend $g_{\vec{v}}$ to all of $\Rd$, first
reflect the function across each coordinate axis, giving a function defined on $[-2R, 2R]^d$, and then extend it to be zero outside of $[-2R, 2R]^d$. The resulting extension satisfies $Pg_{\vec{v}}|_{Q_R} = g_{\vec{v}}$ and $ \|Pg_{\vec{v}}\|_{W^{1,p}_h(\Rd)} \leq C_{d,p} \|g_{\vec{v}}\|_{W^{1,p}_h(Q_R)}$.
\end{proof}

We now turn to the proof of Proposition \ref{discrete norm prop}.
\begin{proof}[Proof of Proposition \ref{discrete norm prop}] \ \\
\begin{enumerate}[(a)]

\item The result follows from H\"older's inequality on $\ell^p$ and the definition of $W^{-1,p}_h$.
\[  \|u\|_{W^{-1,p}_h} = \sup_{g \in W^{1,p'}_h} \frac{|(u,g)_h|}{\|g\|_{W^{1,p'}_h}} \leq \sup_{g \in W^{1,p'}_h} \frac{\|u\|_{L^p_h} \|g\|_{L^{p'}_h}}{\|g\|_{W^{1,p'}_h}}  \leq \|u\|_{L^p_h} \]
(See also \cite[Equation 2.9]{Beale}.)

\item By definition of the ${W^{-1,p}_h}$ norm,
\begin{align} \label{eineqpf}
|(u,g)_h| \leq \|u\|_{W^{-1,p}_h} \|g\|_{W^{1,p'}_h} \leq (1+ 2d/h) \|u\|_{W^{-1,p}_h} \|g\|_{L^{p'}_h} \ .
\end{align}
Suppose $p< +\infty$. If we define $g= \{ |u_i|^{p-2} u_i \}$ and use that $(p-1)p' = p$,
\[ (u,g)_h = \sum |u_i|^{p}h^d = \|u\|_{L^p_h}^p \  \text{ and } \  \|g\|_{L^{p'}_h}^{p'} =  \sum |u_i|^{(p-1)p'} h^d = \|u\|_{L^p_h}^{p} \ . \]
Thus, $ \|u\|_{L^p_h} \leq (1+ 2d/h) \|u\|_{W^{-1,p}_h}$.

If $p = +\infty$, define $g^j_i$ to be $1$ if $ i = j$ and $0$ otherwise. Then (\ref{eineqpf}) implies
\[ |u_j|h^d \leq (1+2/h)\|u\|_{W^{-1,\infty}_h} h^d \ , \]
so $\|u\|_{L^\infty_h} \leq (1+2d/h)\|u\|_{W^{-1,\infty}_h}$. (See also \cite[Equation 2.10]{Beale}.)

\item If $\Omega \subseteq B_R$, the number of grid points in $\Omega$ is bounded by $C_R h^{-d}$. By an elementary inequality for $\ell^p$ norms on finite dimensional vector spaces,
\[ \|u\|_{L^q_h(\Omega)} = h^{d/q} \|u\|_{\ell^q}  \leq (C_R h^{-d})^{\frac{1}{q}-\frac{1}{p}} h^{d/q} \|u\|_{\ell^p} \leq C_{R,q,p}  \|u\|_{L^p_h(\Omega)} \ .\]

\item Fix $Q_R = [-R,R]^d$, so by $(c)$,$ \|g\|_{W^{1,p'}_h(Q_R)} \leq C_{p,q,R} \|g\|_{W^{1,q'}_h(Q_R)} \leq C_{p,q,R} \|g\|_{W^{1,q'}_h}$.
 By Lemma \ref{extlemma}, there exists an extension operator $P: W^{1,p'}_h(Q_R) \to W^{1,p'}_h(\Rd)$.
Since $u$ is supported in $B_R \subseteq Q_R$,
\begin{align*}
|(u,g)_h| &= |(u,Pg)_h| \leq \|u\|_{W^{-1,p}_h}\|Pg\|_{W^{1,p'}} \leq \|u\|_{W^{-1,p}_h}\|g\|_{W^{1,p'}_h(Q_R)} \\
&\leq C_{p,q,R} \|u\|_{W^{-1,p}_h}\|g\|_{W^{1,q'}_h} \ . 
\end{align*}

\item See \cite[Proposition 2.1]{BealeMajda1}.

\end{enumerate}
\end{proof}

\subsection{Proof of Regularity of Velocity Field and Particle Trajectories} \label{reg velocity field section}
In this section we prove Lemma \ref{reg velocity field lemma} on the regularity of the velocity field $\vec{v} = \grad K * \rho$, the divergence of the velocity field $\grad \cdot \vec{v} = \Delta K *\rho$ and Lemma \ref{particlelemma} on the regularity of the particle trajectories $X^t(\alpha)$ and the Jacobian determinants $J^t(\alpha)$.

\begin{proof}[Proof of Lemma \ref{reg velocity field lemma}] 
By the linearity of differentiation and convolution, it is enough to show the result in the specific case that that $K = K_1$ with $s = S \geq 1-d$.
 
 When $s = 1-d$, $K$ is a constant multiple of the Newtonian potential. By Assumption \ref{exact sol reg as}, $\rho \in C^1([0,T],C^{r}_c(\Rd))$ for $r \geq m,L$ and has compact support. It is a classical result that $\Delta K * \rho = C \rho$. Thus $\grad \cdot \vec{v} = \Delta K * \rho$ belongs to $C^{L}(\Rd) \cap C^{m}(\Rd)$ and has bounded derivatives up to order $m$.

Now, we may assume that both $\grad K$ and $\Delta K$ are homogeneous of order at least $1-d$. We  treat both simultaneously by proving the result for a kernel $\K$, which for fixed $l\geq 1-d$ satisfies
 \[ |\partial^\beta \K(x)| \leq C|x|^{l-|\beta|} , \ \forall x \in \Rd \setminus \{0\}  , 0 \leq |\beta| \leq l + d-1 \ . \] 
 Note that this implies $\partial^\beta \K \in L^1_\loc(\Rd)$ for all $|\beta| \leq l + d - 1$.
When $\K(x) = \grad K(x)$, $l=s$, and when $\K(x) = \Delta K(x)$, $l = s-1$. The fact that $ \K * \rho \in C^L(\Rd)$ is immediate, since $\rho \in C^L_c(\Rd)$ and $\K \in L^1_\loc (\Rd)$. We now turn to the estimates that imply $\K * \rho \in C^m(\Rd)$ and the bound on its derivatives.

We prove by induction that for ${0\leq |\gamma| \leq l + d-1}$, 
\[\partial^\gamma \K * \rho (x) = \int \partial^\gamma \K(x-y) \rho(y) dy \ .\]
 If $|\gamma|=0$, equality holds. Suppose $|\gamma| \geq 1$ and $\partial^\beta \K * \rho (x) = \int \partial^\beta \K(x-y) \rho(y) dy$ for all $|\beta| = |\gamma|-1$. If $|\eta|=1$ satisfies $\partial^\eta \partial^\beta = \partial^\gamma$,
\[ \partial^\gamma (K*\rho) = \partial^\eta ((\partial^\beta \K )* \rho) = (\partial^\beta \K )*( \partial^\eta \rho) \ . \]
We now move $\partial^\eta$ onto $\K$. Fixing $\epsilon \in (0,1)$,
\begin{align*}
(\partial^\beta \K) * (\partial^\eta \rho)(x) &= \int_{|y|< \epsilon} \partial^\beta \K(y) \partial^\eta_x \rho(x-y) dy + \int_{|y|\geq \epsilon} \partial^\beta \K(y) \partial^\eta_x \rho(x-y) dy \\
&=\mathcal{O}(\epsilon^{d+l - |\beta|}) + \int_{|y|\geq \epsilon} \partial^\gamma \K(y) \rho(x-y) dy + \mathcal{O}(\epsilon^{d-1+l - |\beta|}) \ .
\end{align*}
Since $|\beta| = |\gamma|-1 \leq l + d-2$, sending $\epsilon \to 0$ gives
\[ ( \partial^\beta \K )* (\partial^\eta \rho)(x)  = ( \partial^\gamma \K )* \rho(x)  \ . \]

This shows $\K*\rho \in C^{l + d-1}(\Rd)$. We now show $\K*\rho \in C^m(\Rd)$. If $m \leq l + d-1$, this is immediate, so suppose $m > l + d-1$. By Assumption \ref{exact sol reg as}, $\rho \in C^r_c(\Rd)$ for $r \geq m -( s + d-2)$. Since $\partial^\gamma (\K*\rho) = (\partial^\gamma \K)*\rho$ and $\partial^\gamma \K \in L^1_\loc$ for all $|\gamma| \leq l+ d-1$, for any $|\beta| = m$, there exists $|\eta| = m - (l + d-1) \leq r$ so that 
\[ \partial^\beta \K*\rho = (\partial^\gamma \K)*(\partial^\eta \rho) \ . \]
Thus, $\K*\rho \in C^m(\Rd)$.

Finally, we show that $|\partial^\beta \K*\rho(x,t)| \leq C(1+ |x|^{(l-|\beta|)_+})$ for $|\beta| \leq m$. If ${|\beta| \geq l + d-1 }$, let $|\eta| = |\beta| - (l + d-1)$ and $|\gamma| =l+ d-1$ so that $\partial^\beta = \partial^\eta \partial^\gamma$. Otherwise, let $\gamma = \beta$ and $\eta =0$. Since $\partial^\beta \K * \rho(x)$ is continuous, it is bounded for $|x| \leq 2R_0$. If $|x| > 2R_0$,
\begin{align*}
|\partial^\beta \K * \rho (x) | &= |(\partial^\gamma \K)*(\partial^\eta \rho)(x) |\leq C\int |x-y|^{l - \gamma} |\partial^\eta \rho(y) |dy
\leq C \int_{B_{R_0}} |x-y|^{l - \gamma}dy \\
& \leq \begin{cases} C |x|^{l - \gamma} &\text{ if } l > \gamma  \ , \\
C &\text{ if } l \leq \gamma \ . \end{cases}
\end{align*}
Thus, $|\partial^\beta \K*\rho(x)| \leq C(1+ |x|^{(l- |\beta|)_+})$. The constant depends on the exact solution, the kernel, the dimension, $\beta$, $T$, and $R_0$.
\end{proof}

We now prove Lemma \ref{particlelemma} on the regularity of the particle trajectories.
\begin{proof}[Proof of Lemma \ref{particlelemma}]
Assumption \ref{exact sol reg as} and Lemma \ref{reg velocity field lemma} ensure sufficient regularity on the velocity field so that there exists a time interval $[0,T_0]$ on which, for $\alpha \in B_{R_0+2}$, the particle trajectories $X^t(\alpha)$ and their inverses $X^{-t}(\alpha)$ uniquely exist, are continuously differentiable in time, and $C^L$ in space. Likewise $J^t(\alpha)$ and $J^{-t}(\alpha)$ are $C^{L-1}$ in space and satisfy equation (\ref{conservationofmassformula}) .

Suppose $T_0 < T$. Then there exists $\alpha$ so that $|X^t(\alpha)| \to+ \infty$ as $t \to T_0$. This contradicts Assumption \ref{exact sol reg as}. Therefore, $T_0 = T$. 

If $s=1-d$, we may replace $B_{R_0+2}$ with $\Rd$.
\end{proof}

\subsection{Proof of Regularized Kernel Estimates} \label{reg kernel estimates}

We now prove the regularized kernel estimates from Section \ref{ConvergenceSection}.
Throughout, we use that
\begin{align} \label{alphaintegrabilityfact}
\begin{cases} |y|^s \in L^1(B_1(0)) \text{ and }|y|^s \in L^\infty(\Rd \setminus B_1(0)) & \text{ if } 1-d \leq s \leq 0 \ , \\
|y|^s \in L^\infty(B_1(0)) &\text{ if }0 < s \ . \end{cases}
\end{align}

\begin{proof}[Proof of Lemma \ref{regofgradKdelta}] 
By the linearity of differentiation and convolution, it is enough to show the result in the specific case that that $K = K_1$ with $s = S \geq 1-d$.

If $s = 1-d$, $K$ is a constant multiple of the Newtonian potential. In this case, since $\psi \in C^L(\Rd)$ for $L \geq d+2$, it is a classical result that $\Delta K_\delta = \Delta K * \psi_\delta = C\psi_\delta$. Assumption \ref{mollifier assumption} ensures $\psi_\delta \in C^L(\Rd)$, hence $\Delta K_\delta \in C^L(\Rd)$.

We may now treat the cases of $\grad K_\delta$ and $\Delta K_\delta$ simultaneously by proving the result for a kernel $\K_\delta = \K * \psi_\delta$, which for fixed $l\geq 1-d$ satisfies
 \[ |\partial^\beta \K(x)| \leq C|x|^{l-|\beta|} , \ \forall x \in \Rd \setminus \{0\}  , 0 \leq |\beta| \leq l + d-1 \ . \] 
When $\K(x) = \grad K(x)$, $l=s$ and when $\K(x) = \Delta K(x)$, $l = s-1$.

It is enough to show that in a neighborhood around every $x$, there exists  $g(y) \in L^1(\Rd)$ which dominates $\K(y) \partial^\beta_x \psi_\delta(x-y)$. Then, the mean value theorem ensures that the difference quotient which converges to this derivative at $x$ is also dominated by $g(y)$, allowing us to conclude
\[\partial^\beta \int \K(y) \psi_\delta(x-y)dy = \int \K(y) \partial^\beta \psi_\delta(x-y)dy \ ,\]
and $\K_\delta \in C^L(\Rd)$. We use Assumption \ref{mollifier assumption} on the decay and regularity of $\psi$ to find dominating functions when $l \leq 0$ and $l >0$.

If $1-d \leq l \leq 0$, the decay and regularity assumptions ensure that there exists $\epsilon >0$ so that $|\partial^\beta \psi(x)| \leq C |x|^{-d-\epsilon}$ for all $|\beta| \leq L$. If $l > 0$ the regularity assumption ensures that $  | \partial^\beta \psi(x)| \leq C|x|^{-d-l-\epsilon}$ for all $|\beta| \leq L$. Since $\partial^\beta \psi$ is bounded near the origin, there exists $C' >0$ so that for all $x \in \Rd$,
\[ | \partial^\beta \psi (x)| \leq \begin{cases}
C' (|x|+1)^{-d - \epsilon} &\text{ if } 1-d \leq l \leq 0  \ , \\
C' (|x|+1)^{-d  - \epsilon- l} &\text{ if } 0< l \ .  \end{cases}\]
Therefore, 
\begin{align}\label{firstkernelbound}
 |\K(y) \partial^\beta \psi_\delta(x-y)| \leq 
\begin{cases} C_\delta |y|^l (|x-y|+1)^{-d  - \epsilon} &\text{ if } 1-d \leq l \leq 0  \ , \\
C_\delta |y|^l (|x-y|+1)^{-d - \epsilon -l}   &\text{ if } 0< l \ .  \end{cases}
\end{align}
If $|x|< R$ and $|y| > 2R$, $|x-y| +1> |y|-|x| >  \frac{1}{2} |y|$. This gives the following dominating functions when $|x| < R$:
\[ |\K(y) \partial^\beta \psi_\delta(x-y)| \leq \begin{cases} C_{\delta,R} |y|^l 1_{|y| \leq 2R} +   C_{\delta,R}   1_{|y| > 2R}&\text{ if } 1-d \leq l \leq 0  \ , \\
C_{\delta,R} &\text{ if } 0< l \ .  \end{cases} \]
Since $R$ was arbitrary, $\partial^\beta \K_\delta(x) = \K * \partial^\beta \psi_\delta(x)$ for all $x\in \Rd$.
\end{proof}

\begin{proof}[Proof of Lemma \ref{reg ker est lem}]
Our proof generalizes the approaches of Beale and Majda \cite{BealeMajda1} and Anderson and Greengard \cite{AndersonGreengard}.
By the linearity of differentiation and convolution, it is enough to show the result in the specific case that that $K = K_1$ with $s = S \geq 1-d$.

We first prove the following pointwise bounds for $|\beta| \leq L$ and $r>0$:
\begin{align*}
|\partial^\beta \grad K_\delta(x) | \leq \begin{cases}C |x|^{s - |\beta|} &\mbox{if } \delta \leq |x| \ , \\ 
C \delta^{s - |\beta|} & \mbox{if }  |x| \leq r \delta \ . \end{cases}
\end{align*}
We begin with $ |x| \leq r \delta$. By Lemma \ref{regofgradKdelta}, $\partial^\beta \grad K_\delta = \grad K* ( \partial^\beta \psi_\delta)$, so
\begin{align}
|\partial^\beta \grad K_\delta(x)| &\leq  \int |\grad K(y)| |\partial^\beta \psi_\delta(x-y) |dy \leq C \int |y|^s \delta^{-d-|\beta|} \left|\partial^\beta \psi\left(\frac{x-y}{\delta}\right) \right| dy \ , \nonumber\\
&\leq C \int |\delta y|^{s} \delta^{-|\beta|} \left|\partial^\beta \psi\left(\frac{x}{\delta} - y \right) \right| dy  = C \delta^{s - |\beta|}  \int |y|^s \left| \partial^\beta  \psi\left(\frac{x}{\delta} - y \right) \right|  dy \ . \label{pointwiseproof1}
\end{align}
Suppose $1-d \leq s \leq 0$.  By (\ref{alphaintegrabilityfact}),
\begin{align*}
|\partial^\beta \grad K_\delta(x)| \leq C \delta^{s - |\beta|} \left\{ \| \partial^\beta \psi \|_{L^\infty(\Rd)} \int_{|y| \leq 1} |y|^s  dy + \sup_{|y| \geq 1} |y|^s \ \|\partial^\beta \psi \|_{L^1} \right\} \leq C  \delta^{s - |\beta|} \ .
\end{align*}

If $ s>0$, we apply Assumption \ref{mollifier assumption}, $|x| \leq r \delta$, and (\ref{pointwiseproof1}) to conclude
\begin{align*}
|\partial^\beta \grad K_\delta(x)| &\leq C \delta^{s - |\beta|}  \int \left|y + \frac{x}{\delta} \right|^s \left| \partial^\beta  \psi\left(y \right) \right|  dy \ , \\
&\leq C \delta^{s - |\beta|} \int  \left(|y| + r \right)^s \left| \partial^\beta  \psi\left(y \right) \right|  dy \leq C \delta^{s - |\beta|} \ .
\end{align*}

We now consider the the pointwise estimate on $|\partial^\beta K_\delta (x)|$ for $|x| \geq \delta$.
First suppose $1-d \leq s\leq 0$.
Let $\phi_0:\R \to [0,1]$ satisfy $\phi_0(s) = 0$ for $s \leq 1/4$ and $\phi_0(s) = 1$ for $s \geq 1/2$.
Define $\phi_x(y) = \phi_0(|y|/|x|)$. By Lemma \ref{regofgradKdelta},
\begin{align*}
|\partial^\beta \grad K_\delta(x)| &= \left| \int \grad K(y) \partial^\beta \psi_\delta(x-y) dy \right|  \ , \\
&\leq \left| \int \grad K(y)\phi_x(y) \partial^\beta \psi_\delta(x-y) dy \right| + \left| \int \grad K(y)(1-\phi_x(y)) \partial^\beta \psi_\delta(x-y) dy \right| \ , \\
&= I_1 + I_2 \ .
\end{align*}
To control $I_1$, we integrate by parts,
\begin{align*}
I_1 \leq \left| \int \partial^\beta_y (\grad K(y) \phi_x(y)) \psi_\delta(x-y) dy \right| \ .
\end{align*}
As $\phi_x(y)$ is only nonzero for $|y|/|x| > 1/4$, we only need to bound $I_1$ for $\delta \leq |x| < 4|y|$. For any multiindex $\gamma$, $|\partial_y^\gamma \phi_x(y)| \leq C |x|^{-|\gamma|}$  and by Assumption \ref{kernel assumption}, $|\partial^\gamma \grad K(y)| \leq C |y|^{s - |\gamma|} \leq C |x|^{s - |\gamma|}$ for $|x| < 4 |y|$. Combining these facts with the product rule gives $|\partial^\beta_y ( \grad K(y)\phi_x(y))| \leq C |x|^{s-|\beta|}$. 
Since $\psi_\delta\in L^1(\Rd)$, this shows
\begin{align*}
I_1 \leq C |x|^{s - |\beta|} \ .
\end{align*}
We now turn to $I_2$. Since $1- \phi_x(y)$ is nonzero for $|y|/|x| < 1/2$, and ${|1- \phi_x(y)| \leq 1}$,
\begin{align*}
I_2 \leq \left| \int_{|y| \leq |x|/2} \grad K(y) \partial^\beta \psi_\delta(x-y) dy \right| \ .
\end{align*}
For $y$ in this range, $|x-y| \geq |x| - |y| \geq |x|/2$. By Assumption \ref{mollifier assumption} on the regularity of the mollifier, $|x|^{d +|\beta|} |\partial^\beta \psi(x)| \leq C$. Therefore,
\[ |\partial^\beta \psi_\delta (x-y) | = \delta^{-d -|\beta|} \left|\partial^\beta \psi \left(\frac{x-y}{\delta} \right) \right| \leq C \delta^{-d -|\beta|} |(x-y)/\delta|^{-d -|\beta|} \leq C |x|^{-d - |\beta|}  \ . \]
Furthermore, since $s \geq 1-d$,
\[ I_2 \leq C |x|^{-d - |\beta|} \left| \int_{|y| \leq |x|/2} \grad K(y) dy \right| \leq C |x|^{- d - |\beta|} \int_0^{|x|/2} r^{s} r^{d-1} dr  \leq C |x|^{s - |\beta|} \ . \]
This completes the proof of the pointwise bounds for $1-d \leq s \leq 0$.

Now we prove the pointwise estimate on $|\partial^\beta K_\delta (x)|$ for $\delta \leq  |x| $ when $s >0$. First, note that for any multiindex $\gamma$ such that  $|\gamma| \leq |\beta|$ and $s - |\gamma| \geq -1$,
\begin{align*}
&| \partial^\beta \grad K_\delta(x) | \\
&\quad \leq \left| \int_{|y|>\epsilon}  \grad K(x-y) \partial^{\beta} \psi_\delta(y) dy \right| +  \left| \int_{|y|\leq\epsilon}  \grad K(x-y) \partial^{\beta} \psi_\delta(y) dy \right| \\
&\quad \leq \left| \int_{|y|>\epsilon}  \partial^\gamma  \grad K(y) \partial^{\beta-\gamma} \psi_\delta(x-y) dy \right| + C_\delta\sum_{m = 0}^{|\gamma|-1} \left| \int_{|y| = \epsilon}  |y|^{s - m} dy \right| + C_\delta \left| \int_{|y|\leq\epsilon}  |y|^s dy \right| \\
&\xrightarrow{\epsilon \to 0} \left|\int \partial^\gamma  \grad K(y) \partial^{\beta-\gamma} \psi_\delta(x-y) dy \right|
\end{align*}
If there is some $|\gamma| \leq |\beta|$ so that $s - |\gamma| \leq 0$, then applying the previous argument for $s\leq 0$ to $|\partial^\gamma \grad K| \leq C|x|^{s-|\gamma|}$,
\[ |\grad^\beta K_\delta(x)| \leq |(\partial^\gamma \grad K) * (\partial^{\beta-\gamma} \psi_\delta)(x)| \leq C |x|^{-1 - (|\beta| -|\gamma|)} = C |x|^{|\gamma|-1 - |\beta|} =  C |x|^{s - |\beta|} \ . \]

On the other hand, if $s - |\beta| > 0$, integrating by parts $|\beta|$ times and applying Assumption \ref{mollifier assumption} on the decay of the mollifier when $s >0$ leaves us with
\begin{align*}
| \partial^\beta \grad K_\delta(x) | &= \left|\int \partial^\beta \grad K(y) \psi_\delta(x-y) dy \right| \leq C\int |y|^{s - |\beta|} |  \psi_\delta(x-y)| dy \\
&\leq C \int_{|y|\leq 2|x|} |y|^{s - |\beta|} | \psi_\delta(x-y) |dy + C \delta^{-d}\int_{|y|> 2|x|} |y|^{s - |\beta|} \frac{ \delta^{d +s + \epsilon}}{|x-y|^{d +s+ \epsilon} }dy
\end{align*}
The first term is bounded by $C |x|^{s - |\beta|}$. The second term is bounded for $|x| \geq \delta$ by
\begin{align*}
& C \delta^{s +\epsilon}\int_{|y|> 2|x|} |y|^{s - |\beta|} \frac{1}{(|y|-|x|)^{d+s + \epsilon}} dy \leq C \delta^{s +\epsilon} \int_{|y|> 2|x|} |y|^{s - |\beta|} \frac{1}{|y|^{d+s + \epsilon}} dy \\
&\leq C \delta^{s +  \epsilon} \int_{2|x|}^\infty r^{- |\beta|- d - \epsilon} r^{d-1} dr \leq C \delta^{s + \epsilon} |x|^{-|\beta|-\epsilon} \leq C |x|^{s- |\beta|} \ .
\end{align*}
This completes the proof of the pointwise estimates.

Finally, we apply the pointwise estimates to obtain Lemma \ref{reg ker est lem}. We define $\delta' = (C' + 1)\delta$, and without loss of generality, we assume $R \geq \delta'$. First, decompose the integral as
\begin{align*} 
&\int_{B_R} | \partial^\beta \grad K_\delta( x + g(x))| dx \\
&\quad \leq \int_{B_R \cap B_{\delta'}} | \partial^\beta \grad K_\delta(x + g(x))| dx + \int_{B_R \setminus B_{\delta'}} | \partial^\beta \grad K_\delta(x + g(x))| dx \\
&\quad= I_3 + I_4 
\end{align*}
When $|x| \leq \delta'$, $|x + g(x)| \leq 2\delta' = 2(C'+1)\delta$ and $
I_3 \leq C \delta^{s -|\beta|} \delta^d \leq C \delta^{s + d - |\beta|}$. When $|x| > \delta'$, we use that $|x+g(x)| \geq \delta' - C' \delta = \delta$ to conclude
\begin{align*}
I_4 &\leq C \int_{B_R \setminus B_{\delta'}} |x+g(x)|^{s - |\beta|} dx \ . 
\end{align*}
If $s - |\beta| >0$, the fact that $|x+g(x)| \leq 2R $ implies the integral is bounded by a constant which depends on $R$.
 If $s - |\beta| \leq 0$,
\begin{align*}
I_4 &\leq C \int_{B_R \setminus B_{\delta'}} (|x|-|g(x)|)^{s - |\beta|} dx  \leq C \int_{\delta'}^R (r-C'\delta)^{s - |\beta|} r^{d-1} dr \\
& \leq C \int_{\delta}^{R- C' \delta} r^{s - |\beta|} (r + C' \delta)^{d-1} dr 
\leq C \int_{\delta}^{R} r^{s + d - |\beta| -1} (1+C')^{d-1} dr \\
&\leq
\left\{\begin{alignedat}{3}
    & C (R^{s +d - |\beta|} - \delta^{s +d - |\beta|}) &&\leq C &\quad \mbox{if } s +d > |\beta| \ , \\
    & C (\log(R)-\log(\delta))&&\leq  C |\log \delta| &\quad \mbox{if } s +d = |\beta| \ , \\
    & C(\delta^{s +d - |\beta|} - R^{s +d - |\beta|} ) &&\leq C \delta^{s+d - |\beta|} &\quad \mbox{if } s +d < |\beta| \ .
  \end{alignedat}\right.
\end{align*}
The constant depends on the kernel, mollifier, dimension, $\beta$, and $R$.
\end{proof}
\textbf{Acknowledgements:} The authors would like to thank MSRI, where part of this work was completed. The authors would also like to thank Prof. Jos\'e Carrillo and Prof. Jeff Eldredge for very helpful conversations.
\bibliographystyle{amsplain}
\bibliography{ResearchProposalReferences}

\end{document}